%% file: fe2-arxiv.tex
\documentclass[final, 12pt, leqno]{amsart}
\usepackage[utf8x]{inputenc}
\usepackage{amsmath, amsthm, amsfonts, amssymb, enumerate, wrapfig,
tikz,  color, rotating, colortbl,   datetime,  float,  stmaryrd,
hyperref, bookmark, cleveref, ifthen, tikz-3dplot, setspace, mathtools}
\usetikzlibrary{decorations.markings, decorations.pathreplacing, decorations.pathmorphing}
\usepackage{amscd}
\title{Link Homology and Frobenius Extensions II}
 \author{Mikhail Khovanov} 
 \address{Department of Mathematics, Columbia University, New York, NY 10027, USA}
 \email{\href{mailto:khovanov@math.columbia.edu}{khovanov@math.columbia.edu}}
 \author{Louis-Hadrien Robert}
 \address{Université de Genève, Section de Mathématiques, rue du
 lièvre 2--4, 1227 Genève, Switzerland}
 \email{\href{mailto:louis-hadrien.robert@unige.ch}{louis-hadrien.robert@unige.ch}}
 \date{May 19, 2020}
 \hypersetup{
    pdftoolbar=true,        
    pdfmenubar=true,        
    pdffitwindow=false,     
    pdfstartview={FitH},    
    pdfkeywords={}, 
    pdfnewwindow=true,      
    colorlinks=true,       
    linkcolor=red,          
    citecolor=teal,        
    filecolor=magenta,      
    urlcolor=violet,          
    linkbordercolor=red,
    citebordercolor=teal,
    urlbordercolor=violet, 
    linktocpage=true}
\usetikzlibrary{arrows, decorations.markings, decorations, patterns, positioning, decorations.pathreplacing, decorations.pathmorphing}
\tikzset{->-/.style={decoration={markings, mark=at position .5 with {\arrow{>}}},postaction={decorate}}}
\tikzset{-<-/.style={decoration={markings, mark=at position .5 with {\arrow{<}}},postaction={decorate}}}
\let\oldtocsubsection\tocsubsection
\renewcommand\tocsubsection[3]{\hspace{0.5cm}\oldtocsubsection{#1}{#2}{#3}}
\let\oldtocsubsubsection\tocsubsubsection
\renewcommand\tocsubsubsection[3]{\hspace{1cm}\oldtocsubsubsection{#1}{#2}{#3}}

\let\emptyset\varnothing

\newcounter{res}[section]
\numberwithin{res}{section}

\newtheorem{lem}[res]{Lemma}

\newtheorem{prop}[res]{Proposition}
\newtheorem{cor}[res]{Corollary}

\theoremstyle{definition}

\newtheorem{dfn}[res]{Definition}

\newtheorem{exa}[res]{Example}

\newcommand{\imagesfolder}{.}
\def\co{\colon\thinspace}

\newcommand{\NB}[1]{\ensuremath{\vcenter{\hbox{#1}}}}

\newcommand{\ZZ}{\ensuremath{\mathbb{Z}}}
\newcommand{\CC}{\ensuremath{\mathbb{C}}}

\newcommand{\RR}{\ensuremath{\mathbb{R}}}

\renewcommand{\SS}{\ensuremath{\mathbb{S}}}
 \newcommand
{\Id}{\operatorname{Id}} \newcommand{\id}{\mathrm{Id}}

\newcommand{\Ker}{\mathop{\mathrm{Ker}}\nolimits}
\newcommand{\Hom}{\mathop{\mathrm{Hom}}}
\newcommand{\End}{\mathop{\mathrm{End}}}

\newcommand{\rk}{\mathrm{rk}}

\newcommand\kup[1]{\left\langle #1 \right\rangle}
\newcommand\kupo[1]{\left\langle #1 \right\rangle_{\omega}}

\newcommand{\Z}{\mathbb{Z}}

\newcommand{\lra}{\longrightarrow}

\newcommand{\mcD}{\mathcal{D}}
\newcommand{\mcF}{\mathcal{F}}

\newcommand{\adm}{\mathrm{adm}}
\newcommand{\Cob}{\mathsf{Cob}}

\newcommand{\undell}{\underline{\ell}}

\newcommand{\mtH}{\mathrm{H}}
\newcommand{\Smod}{\ensuremath{S\textrm{-}\mathsf{mod}}}

\newcommand{\circled}[1]{\NB{\tikz[font=\tiny]{\draw[ very thin, black] (0,0) circle (1.1mm) node[black, scale =0.72] {\bf #1};}}}

\newcommand{\SeSu}{\ensuremath{\mathsf{SeSu}}}

\newcommand{\cfCircleSplit}[3]{\NB{\tikz[scale=#1]{
\begin{scope}
  \begin{scope}[yshift= 1.5cm]
    \draw (1,0) arc (0:360:1cm and 0.3cm);
    \draw[thin] (1,0) arc (0:-180:1cm and 1.2cm);
    \node at (0, -0.75) {$#2$};
  \end{scope}
  \begin{scope}[yshift =-1.5cm]
    \draw (1,0) arc (0:-180:1cm and 0.3cm); 
    \draw[densely dotted] (1,0) arc (0:180:1cm and 0.3cm); 
    \draw (1,0) arc (0:180:1cm and 1.2cm);
    \node at (0, 0.75) {$#3$};
  \end{scope}
\end{scope}}}}

\newcommand{\cfCircleId}[2][]{\NB{\tikz[scale=#2]{
\begin{scope}
  \begin{scope}[yshift= 1.5cm]
    \draw (1,0) arc (0:360:1cm and 0.3cm)coordinate [pos=0](e) coordinate [pos=0.5](f);
  \end{scope}
  \node at (0,0) {$#1$};
  \begin{scope}[yshift =-1.5cm]
    \draw (1,0) arc (0:-180:1cm and 0.3cm)coordinate [pos=0](g) coordinate [pos=1](h);
    \draw[densely dotted] (1,0) arc (0:180:1cm and 0.3cm); 
  \end{scope}
  \draw (e) -- (g);
  \draw (f) -- (h); 
\end{scope}}}}

\newcommand{\cfDecSquare}[2]{\NB{\tikz[scale=#1]{
\begin{scope}
  \node at (0,0) {$#2$};
  \draw (-1, -1) -- (-1,1) -- (1,1)  -- (1, -1) -- (-1,-1);
\end{scope}}}}

\newcommand{\cfDoubleSeamA}[1]{\NB{\tikz[scale=#1]{
\begin{scope}[decoration={border,segment length=1mm,amplitude=#1mm,angle=90}]
  \draw (-1, -1) -- (-1,1) -- (1,1)  -- (1, -1) -- (-1,-1);
  \draw[red, postaction={draw, decorate}] (-1, 0) .. controls +
  (0.5,0) and  +(0,0.5) .. (0, -1);
  \draw[red, postaction={draw, decorate}] ( 1, 0) .. controls +
  (-0.5,0) and  +(0,-0.5) .. (0,  1);
\end{scope}}}}

\newcommand{\cfDoubleSeamB}[1]{\NB{\tikz[scale=#1]{
\begin{scope}[decoration={border,segment length=1mm,amplitude=#1mm,angle=90}]
  \draw (-1, -1) -- (-1,1) -- (1,1)  -- (1, -1) -- (-1,-1);
  \draw[red, postaction={draw, decorate}] (1, 0) .. controls +
  (-0.5,0) and  +(0,0.5) .. (0, -1);
  \draw[red, postaction={draw, decorate}] (-1, 0) .. controls +
  (0.5,0) and  +(0,-0.5)  .. (0,  1);
\end{scope}}}}

\newcommand{\cfSeamSquareTwo}[3]{\NB{\tikz[rotate= -90, scale=#1]{
\begin{scope}[decoration={border,segment length=1mm,amplitude=#1mm,angle=90}]
  \node at (-.5,0) {$#2$};
  \node at (0.5,0) {$#3$};
  \draw (-1, -1) -- (-1,1) -- (1,1)  -- (1, -1) -- (-1,-1);
  \draw[red, postaction={draw, decorate}] (0,-1) -- (0, 1);
\end{scope}}}}

\newcommand{\cfSeamSquare}[3]{\NB{\tikz[scale=#1]{
\begin{scope}[decoration={border,segment length=1mm,amplitude=#1mm,angle=90}]
  \node at (-.5,0) {$#2$};
  \node at (0.5,0) {$#3$};
  \draw (-1, -1) -- (-1,1) -- (1,1)  -- (1, -1) -- (-1,-1);
  \draw[red, postaction={draw, decorate}] (0,1) -- (0, -1);
\end{scope}}}}

\newcommand{\cfGenusTwoBdy}[1]{\NB{\tikz[scale=#1]{
\begin{scope}
  \draw (0,0) circle (1cm and 0.3cm);
  \draw[very thin] (-1, 0)
  .. controls +(0,-0.3) and +(0,1) .. (-3, -2)
  .. controls +(0,-2) and +(-1,0) .. (0, -2)
  .. controls +(1, 0) and +(0,-2) .. (3, -2)
  .. controls +(0, 1) and +(0,-0.3) .. (1, 0);
  \draw[very thin] (-1.1, -1) arc (340: 280: 1.5) coordinate [pos =0.9]
  (a);
  \draw[very thin] (a) arc (154: 106: 1.5);
  \draw[very thin] (1.1, -1) arc (200: 260: 1.5) coordinate [pos =0.9]
  (a);
  \draw[very thin] (a) arc ( 26: 74: 1.5);
\end{scope}
}}}

\newcommand{\cfGenusTwoCup}[2]{\NB{\tikz[scale=#1]{
\begin{scope}
  \draw (0,0) circle (1cm and 0.3cm);
  \draw[very thin] (1,0) arc (360:180:1cm and 1cm);
  \node at (0, -2) {$#2$}; 
  \draw[very thin](-3, -2)
  .. controls +(0,-2) and +(-1,0) .. (0, -2.5)
  .. controls +(1, 0) and +(0,-2) .. (3, -2)
  .. controls +(0,2) and +(1,0) .. (0, -1.5)
  .. controls +(-1, 0) and +(0,2) .. (-3, -2);
  \draw[very thin] (-1, -2) arc (300: 240:1.5) coordinate[pos=0.85]
  (a);
  \draw[very thin] (a) arc (111: 69:1.5);
  \draw[very thin] ( 1, -2) arc (240: 300:1.5) coordinate[pos=0.85]
  (b);
    \draw[very thin] (b) arc (69: 111:1.5);
\end{scope}
}}}

\newcommand{\cfCup}[2]{\NB{\tikz[scale=#1]{
\begin{scope}
  \draw (0,0) circle (1cm and 0.3cm);
  \draw[very thin] (1,0) arc (360:180:1cm and 1.2cm);  
  \node at (0, -0.75) {$#2$};
\end{scope}
}}}

\newcommand{\cfCap}[2]{\NB{\tikz[scale=#1]{
\begin{scope}
  \draw (0,1) arc (0:-180: 1cm and 0.3cm);
  \draw[densely dotted] (0,1) arc (0:180: 1cm and 0.3cm);
  \draw[very thin] (0,1) arc (0:180:1cm and 1.2cm);  
  \node at (0, 0.75) {$#2$};
\end{scope}
}}}

\newcommand{\cfSphere}[2]{\NB{\tikz[scale=#1]{
\begin{scope}
  \draw[very thin] (1,0) arc (0:-180: 1cm and 0.3cm);
  \draw[very thin, densely dotted] (1,0) arc (0:180: 1cm and 0.3cm);
  \draw[very thin] (1,0) arc (360:0:1cm and 1cm);  
  \node at (0, -0.65) {$#2$};
\end{scope}
}}}

\newcommand{\cfSphereSpecial}[3]{\NB{\tikz[scale=#1]{
\begin{scope}
  \draw[very thin] (1,0) arc (0:-180: 1cm and 0.3cm);
  \draw[very thin, densely dotted] (1,0) arc (0:180: 1cm and 0.3cm);
  \draw[very thin] (1,0) arc (360:0:1cm and 1cm);  
  \node at (0, -0.65) {$#2$};
  \node at (0, 0.65) {$#3$};
\end{scope}
}}}

\newcommand{\cfGenusOneBdy}[1]{\NB{\tikz[scale=#1]{
\begin{scope}
  \draw (0,0) circle (1cm and 0.3cm);
  \draw[very thin] (-1, 0)
  .. controls +(0,-0.3) and +(0,1) .. (-2, -2)
  .. controls +(0,-1.5) and +(0,-1.5) .. (2, -2)
  .. controls +(0, 1) and +(0,-0.3) .. (1, 0);
  \draw[very thin] (0, -1.8) arc (270: 300: 1.8) coordinate [pos =0.8]
  (a);
  \draw[very thin] (0, -1.8) arc (270: 240: 1.8);
  \draw[very thin] (a) arc (66: 114: 1.8);
\end{scope}
}}}

\newcommand{\cfDelta}[2]{\NB{\tikz[scale=#1]{
  \begin{scope}
    \draw[densely dotted] (1, -1) arc (0:180:1cm and 0.3cm);
    \draw (1, -1) arc (0:-180:1cm and 0.3cm);
    \draw (-1, 1) arc (0:360:1cm and 0.3cm);
    \draw (3, 1) arc (0:360:1cm and 0.3cm);
    \draw[very thin] (1,  1) .. controls +(0,-0.5) and + (0, -0.5) .. +(-2,0) node [midway, below] {$#2$};
    \draw[very thin] (-1,  -1) .. controls +(0,0.3) and + (0, -0.3) .. +(-2,2);
    \draw[very thin] (1,  -1) .. controls +(0,0.3) and + (0, -0.3) .. +(2,2);
  \end{scope}
}}}

\newcommand{\cfMu}[2]{\NB{\tikz[scale=#1]{
  \begin{scope}
    \draw[densely dotted] (3, -1) arc (0:180:1cm and 0.3cm);
    \draw (3, -1) arc (0:-180:1cm and 0.3cm);
    \draw[densely dotted] (-1, -1) arc (0:180:1cm and 0.3cm);
    \draw (-1, -1) arc (0:-180:1cm and 0.3cm);
    \draw (1, 1) arc (0:360:1cm and 0.3cm);
    \draw[very thin] (1,  -1) .. controls +(0,0.5) and + (0, 0.5) .. +(-2,0) node [midway, below] {$#2$};
    \draw[very thin] (-3,  -1) .. controls +(0,0.3) and + (0, -0.3) .. +(2,2);
    \draw[very thin] (3,  -1) .. controls +(0,0.3) and + (0, -0.3) .. +(-2,2);
  \end{scope}
}}}

\newcommand{\cfGenusThree}[1]{\NB{\tikz[scale=#1]{
\begin{scope}
  \draw (270:1) .. controls +(0:1) and +(240:1)  .. (330:3);
  \draw (30:1) .. controls +(-60:1) and +(60:1)  .. (330:3);
  \draw (333:1)  arc (210:270:1.3cm) coordinate[pos=0.9] (a);
  \draw (a) arc (36: 84: 1.3cm);
\end{scope}
\begin{scope}[rotate=120]
  \draw (270:1) .. controls +(0:1) and +(240:1)  .. (330:3);
  \draw (30:1) .. controls +(-60:1) and +(60:1)  .. (330:3);
  \draw (333:1)  arc (210:270:1.3cm) coordinate[pos=0.9] (a);
  \draw (a) arc (36: 84: 1.3cm);
\end{scope}
\begin{scope}[rotate=-120]
  \draw (270:1) .. controls +(0:1) and +(240:1)  .. (330:3);
  \draw (30:1) .. controls +(-60:1) and +(60:1)  .. (330:3);
  \draw (333:1)  arc (210:270:1.3cm) coordinate[pos=0.9] (a);
  \draw (a) arc (36: 84: 1.3cm);
\end{scope}}}}

\newcommand{\cfTubeTT}[2]{\NB{\tikz[scale=#1]{
  \begin{scope}
    \draw (-1, 1) arc (0:360:1cm and 0.3cm);
    \draw (3, 1) arc (0:360:1cm and 0.3cm);
    \draw[very thin] (1,  1) .. controls +(0,-0.7) and + (0, -0.7) .. +(-2,0) node [midway, below] {$#2$};
    \draw[very thin] (-3,  1) .. controls +(0,-2.1) and + (0, -2.1) .. (3,1);
  \end{scope}
}}}

\newcommand{\circlein}[1]{\NB{\tikz[scale=#1]{
  \begin{scope}[scale =#1, decoration={border,segment length=1mm,amplitude=#1mm,angle=90}]
    \draw[red, postaction={draw, decorate}] (0,0) circle (0.6cm);
  \end{scope}
  }}}

\newcommand{\circleout}[1]{\NB{\tikz[scale=#1]{
  \begin{scope}[scale =#1, decoration={border,segment length=1mm,amplitude=#1mm,angle=-90}]
    \draw[red, postaction={draw, decorate}] (0,0) circle (0.6cm);
  \end{scope}
  }}}

\newcommand{\warning}[1]{\NB{\tikz[scale=#1]{
\node[scale=#1, font=\normalsize] at (0,0.05) {$\mathbf{!}$};
\draw (-30:0.35) --(90:0.35) -- (210:0.35) -- cycle;
}}}

\newcommand{\feSeamedCup}[2]{\NB{\tikz[scale=#1]{
\begin{scope}[decoration={border,segment length=#1mm,amplitude=#1mm,angle=-90}]
  \begin{scope}
    \draw (1,0) arc (0:360:1cm and 0.3cm)coordinate [pos=0](e)
    coordinate [pos=0.5](f) coordinate[pos=0.65] (bt)
    coordinate[pos=0.85] (at);
    \draw[thin] (1,0) arc (0:-180:1cm and 1.2cm);
    \node at (0,-0.5) {$#2$};
  \end{scope}
  \draw[red, postaction={draw, decorate}] (at) .. controls +(0, -1)  and +(0, -1) .. (bt);
\end{scope}
}}}

\newcommand{\feSeamedTwistedTube}[3]{\NB{\tikz[scale=#1]{
\begin{scope}
\begin{scope}  [decoration={border,segment length=#1mm,amplitude=#1mm,angle=-90}]
  \begin{scope}[xshift=0cm]
    \draw (1,0) arc (0:360:1cm and 0.3cm)coordinate [pos=0](e)
    coordinate [pos=0.5](f) coordinate[pos=0.65] (bl)
    coordinate[pos=0.85] (al);
    \node at (0,-0.5) {$#3$};
    \draw[very thin] (-1,0) arc (180:270:1cm and 1.2cm) coordinate[pos=1] (LB);
  \end{scope}
  \begin{scope}[xshift=#2cm]
    \begin{scope}[xshift=2cm]
    \draw (1,0) arc (0:360:1cm and 0.3cm)coordinate [pos=1](f)
    coordinate [pos=0.5](f) coordinate[pos=0.65] (br)
    coordinate[pos=0.85] (ar);
    \draw[very thin] (1,0) arc (0:-90:1cm and 1.2cm) coordinate[pos=1]
    (RB);
  \end{scope}
  \draw[very thin] (LB) -- (RB);
  \draw[very thin] (f) .. controls +(0, -0.4) and +(0, -0.4) .. (e);
  \draw[red, postaction={draw, decorate}] (al) .. controls +(0, -0.6)  and +(0, -0.6) .. (br);
  \draw[red, postaction={draw, decorate}] (ar) .. controls +(0, -1)  and +(0, -1) .. (bl);
\end{scope}
\end{scope}
\end{scope}
}}}

\newcommand{\feMarkedPoint}[1]{\NB{\tikz[scale=#1]{
  \begin{scope}       \draw (0,0) --  +(0.5,0);
   \draw [red, -latex, thick] (0.25, 0) -- ( 0.35,0);
  \end{scope}
}}}

\newcommand{\feOrientation}[1]{\NB{\tikz[scale=#1]{
  \begin{scope}       \draw (0,0) --  +(0.5,0);
   \draw [->] (0.25, 0) -- ( 0.27,0);
  \end{scope}
}}}


\begin{document}
\begin{abstract}
The first two sections of the paper provide a  convenient scheme and additional diagrammatics for working with Frobenius extensions responsible for key flavours of equivariant SL(2) link homology theories. The goal is to clarify some basic structures in the theory and propose a setup 
to work over sufficiently non-degenerate base rings. The third section works out two related SL(2) evaluations for seamed surfaces.   
\end{abstract}
\maketitle
\tableofcontents

\section{A cube of four rank two Frobenius extensions} \label{sec:algebrasl2}

\subsection{Road map for Frobenius extensions in \texorpdfstring{$U(2)$}{U(2)}-equivariant Khovanov homology} \label{sec_road_map}
A building block for the Khovanov homology~\cite{MKCatJones, BNCatJones}
is a specific $(1+1)$-dimensional TQFT associated to a commutative
Frobenius algebra of rank two. Different versions of this TQFT exist,
including those studied in  \cite{BNTangle,MKFrobenius} and in many
other papers
\cite{CaprauTangle,vogel2015functoriality,ClarkMorrisonWalkerFunctoriality,BNMorrisonLee}. These different
versions may yield different types of information. The first striking
example is the construction by Rasmussen~\cite{RasmussenSlice} of a lower bound on the
slice genus of knots based on degenerating one of these TQFTs into the Lee $(1+1)$-dimensional TQFT~\cite{Lee}. An attempt to sort through some of these theories, including Bar-Natan theories~\cite{BNTangle}, had been made in~\cite{MKFrobenius}. In this note, without trying to be all-inclusive, we go through several key variations on the usual $U(2)$-equivariant Frobenius pair $(R,A)$, with $R$ and $A$ being the  equivariant $U(2)$ cohomology groups with $\Z$ coefficients of a point and a 2-sphere, respectively. We collect the four Frobenius pairs into Figure~\ref{fig:first-cube} below; see  also Figure~\ref{fig_bigger-cube} for a more detailed view of this cube.

\begin{figure}[ht]
    \centering
    \NB{\tikz[yscale=1.2, xscale=1.4]{\input{\imagesfolder/cf_cube}}}
    \caption{Cube of four Frobenius extensions; red numbers $1-4$ on the sides provide links to descriptions of these extensions.}
    \label{fig:first-cube}
\end{figure}
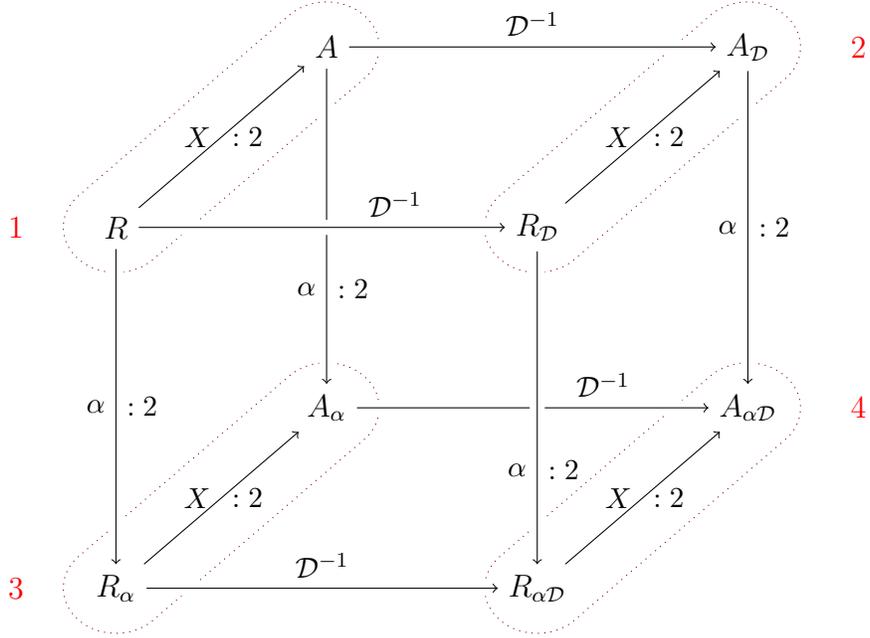

The diagram consists of four graded Frobenius extensions of degree two  
\begin{equation}\label{eq_4_extensions} 
(R,A), \   (R_{\mcD},A_{\mcD}), \ (R_{\alpha},A_{\alpha}), \ (R_{\alpha\mcD},A_{\alpha\mcD})
\end{equation} 
and maps between them. In each of these extensions $(R_{\ast},A_{\ast})$ Frobenius algebra $A_{\ast}$ is a free graded $R_{\ast}$ module of rank two with a basis $\{1,X\}$, with the trace map $\epsilon:A_{\ast}\lra R_{\ast}$ 
given by 
\begin{equation} \label{eq_trace_map} 
    \epsilon(1)=0, \ \  \epsilon(X)=1,
\end{equation} 
and the comultiplication
\begin{equation}\label{equation_comult}
\begin{array}{rcl} 
     \Delta(1) & = & X \otimes 1 + 1\otimes (X-E_1), \\  \Delta(X) & = & X\otimes X - E_2 1\otimes 1.
\end{array}
\end{equation}
Inclusions of rings $R_{\ast}\subset A_{\ast}$ in  the extensions are shown by  the four  arrows that go diagonally pointing northeast. Horizontal and vertical arrows represent certain inclusions of ground rings $R_{\ast}$ and induced inclusions of rings $A_{\ast}$.  

\vspace{0.1in} 

\hypertarget{frobext1}{{\bf 1.}} The first extension $(R,A)$ consists of graded rings  
\begin{itemize}
    \item $R = \Z[E_1,E_2]$, $\deg(E_1)=2, \deg(E_2)=4,$
    \item $A = R[X]/(X^2-E_1 X + E_2)$, $\deg(X)=2$. 
\end{itemize}
Algebra $A$ is a commutative Frobenius $R$-algebra of rank two with a basis $\{1,X\}$. The trace and comultiplication are given by formulas (\ref{eq_trace_map}), (\ref{equation_comult}). 

Algebra $A$ is a free graded $R$-module with homogeneous basis $\{1,X\}$. 
Elements $X$ and $E_1-X$ in $A$ are roots of the  polynomial  $x^2-E_1 x + E_2$ with coefficients in $R$, and there is an $R$-linear involution $\sigma$ of the algebra $A$ transposing $X$ and $E_1-X$. We can somewhat restore this symmetry between $X$ and $E_1-X$ in the notation and denote 
\begin{equation}
X_1 \ = \ X, \ \ X_2 \ = \ E_1 - X, 
\end{equation} 
so that 
\begin{equation*}
    X_1+X_2 = E_1, \ \  X_1 X_2 = E_2 
\end{equation*}
and 
\begin{equation}
    \sigma(X_1) = X_2, \ \ \sigma(X_2)=X_1, \ \ \sigma(1)=1, 
\end{equation}
with $\sigma(a+bX_1)=a+bX_2$ for $a,b\in R$. 
We have 
\begin{eqnarray}\label{equation_comult_roots}
    \Delta(1) & = & X_1 \otimes 1 - 1 \otimes X_2 \ = \ - (X_2
\otimes 1 - 1 \otimes X_1),  \\
\Delta(X_1) & = & X_1 \otimes X_1 - E_2\, 1\otimes 1, \\
\Delta(X_2) & = & E_2 \, 1\otimes 1 - X_2 \otimes X_2  .
\end{eqnarray}

Comultiplication introduces a sign which breaks the symmetry between the two roots, and comultiplication  does not commute with $\sigma$. Likewise, the trace map breaks the symmetry between the roots, 
\begin{equation} \label{eq_trace_map_roots} 
    \epsilon(X_1)=1,\ \  \epsilon(X_2)=-1,
\end{equation}
and $\epsilon\sigma = - \epsilon$. 

\vspace{0.1in} 

In Figure~\ref{fig:frob} we recall the usual diagrammatic conventions for 2-dimensional TQFTs and commutative Frobenius algebras, with our distinguished  generator $X$ denoted by a dot.
 
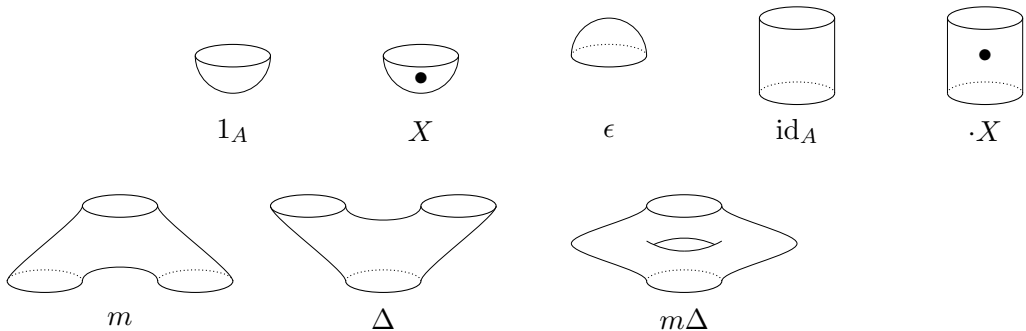
\begin{figure}[ht]
  \centering
  \NB{\tikz[scale = 0.5]{\input{\imagesfolder/cf_frob}}}
  \caption{Cobordisms and the Frobenius structure maps dictionary.}
  \label{fig:frob}
\end{figure}
The definition of $\epsilon$ reads diagramatically as follows:
\begin{equation}\label{eq:spheres}
    \cfSphere{0.5}{} =0 \qquad \text{and} \ \  \   
    \cfSphere{0.5}{\bullet} =1.
\end{equation}

The composition of comultiplication and multiplication
\begin{equation}\label{eq_com_mul} 
    m \Delta \ : \ A \lra A\otimes A \lra A , \ 
    m\Delta(1) = m(X_1 \otimes 1 - 1 \otimes X_2)=X_1-X_2
\end{equation}
is an $A$-module map taking  $1$ to the difference  of roots $X_1-X_2$. Consequently, this map is the multiplication by $X_1-X_2$. Thus, a genus one cobordism with one boundary component, when viewed as a cobordism from the empty one-manifold to $\SS^1$, represents the element $X_1-X_2=m\Delta(1)\in A$. 
A two-torus with two boundary components, when viewed as a cobordism between the circles, describes the map $A\lra A$ which is the multiplication by $X_1-X_2$.  

The trace map $\epsilon$ determines an $R$-linear symmetric bilinear form on $A$, 
\[ (,) \ : \ A \otimes A \lra R , \ (a,b) = \epsilon(ab).
\] 
This is a unimodular form on $A$ with  dual bases $(1,X)$ and $(X-E_1,1)$. This pair of dual bases can also be written $(1,X_1)$ and $(-X_2,1)$.  

The identity map on $A$ decomposes in these bases as usual: \begin{equation}\label{eq_id_map}
    a \ = \  X_1\otimes \epsilon(a) - 1 \otimes \epsilon(X_2 a) \ = \  - X_2\otimes \epsilon(a) + 1 \otimes \epsilon(X_1 a), \ a\in A. 
\end{equation}
We can write this identity diagrammatically as the \emph{neck-cutting} relation. 

\begin{equation}\label{eq_neck_cut}
\cfCircleId{0.5}=
\cfCircleSplit{0.5}{\bullet}{} -
\cfCircleSplit{0.5}{}{\circ}
=
\cfCircleSplit{0.5}{}{\bullet} -
\cfCircleSplit{0.5}{\circ}{},
\qquad \text{with} \qquad
\cfDecSquare{0.5}{\circ} :=
E_1\,\cfDecSquare{0.5}{} 
-
\cfDecSquare{0.5}{\bullet} \ . 
\end{equation}

Notice that we use a dot of the second type (hollow dot, a little circle with white area inside) to denote the second root $X_2$, borrowing the idea and notation from \cite{beliakova2019functoriality}.
The following notation may also be useful: 
\begin{equation}
X_{\bullet}:=X_1=X, \ \  X_{\circ}:=X_2=E_1-X,
 \end{equation}
where we use $\bullet$ and $\circ$ as subscripts for root elements instead of $1,2$. 

A solid dot $\bullet$ placed on a vertically positioned annulus denotes the
multiplication by $X_1=X$ operator on $A$, see the rightmost diagram in Figure~\ref{fig:frob}. Defining relation in $A$ translates into the solid dot reduction relation
\begin{equation}\label{eq_dot_red} 
    \cfDecSquare{0.5}{\bullet\,\bullet} = 
    E_1 \ \cfDecSquare{0.5}{\bullet} - E_2 \ 
    \cfDecSquare{0.5}{}
\end{equation}

The dual (or
hollowed) dot ($\circ$) denotes the multiplication by $X_2=E_1-X$. A torus with a boundary component is the difference of a dot and a dual dot on a disk with the same boundary.

\begin{equation}\label{eq_torus} 
\cfGenusOneBdy{0.4} \, = \, \cfCup{0.5}{\bullet} - \cfCup{0.5}{\circ}=
\cfCup{0.5}{\star}
 \quad \text{with} \quad \cfDecSquare{0.5}{\star} = \cfDecSquare{0.5}{\bullet} - \cfDecSquare{0.5}{\circ}.
\end{equation} 
Relation (\ref{eq_dot_red}) holds with solid dot replaced by hollow dot. 
We denote by  $\star$ the difference of the solid and hollow dots, and 
use a similar subscript in $X$: 
\begin{equation}
    X_{\star} :=  X_1-X_2 = X_{\bullet}-X_{\circ} =  2X-E_1 = m\Delta(1) \in A. 
\end{equation}

Note that $X_{\star}$ is in $A$ but not in $R$. 
With  our notations, a torus with one boundary component  equals a disk with a star dot on it, see equation (\ref{eq_torus}). In this sense, having a star dot on a surface is equivalent to adding a handle to it. 

Recall that the discriminant of the quadratic polynomial 
    $x^2-E_1 x + E_2$ is the square of the difference $X_1-X_2$ of roots and belongs to $R$, 
\begin{equation}\label{eq_discriminant} 
        \mcD = (X_1 - X_2)^2 = X_{\ast}^2= E_1^2 - 4 E_2  \ \in R. 
\end{equation}
The degree $\deg(\mcD)=4$. Diagrammatically, 
\begin{equation}\label{eq_torus_2} 
 \cfDecSquare{0.5}{\star^2}\  =\  \cfDecSquare{0.5}{\star\star} \ = \ \mcD \ \ \cfDecSquare{0.5}{}\ .
\end{equation} 
A closed surface of genus two evaluates to $\epsilon(\mcD)=0$, as in~\cite[Section 9.2]{BNTangle}. 
A closed surface of genus three evaluates to  $2\mcD$:

\begin{equation} \label{eq:gen3surface}
\cfGenusThree{0.5}\ =\ 2 \mcD \ .
\end{equation}

Significance of closed genus three surfaces and their evaluation was pointed out by Bar-Natan in~\cite[Section 9.2]{BNTangle}, where the Lee theory was tied to specializing the value of the closed genus three surface to $8$.  

 Likewise, a surface of genus two with one boundary circle equals a disk times $\mcD$, see below and equations (\ref{eq_torus}), (\ref{eq_torus_2}): 
\[
\cfGenusTwoBdy{0.5}\,=\,
\cfGenusTwoCup{0.5}{\bullet}\,=\, \cfCup{0.5}{\star\star} \,=\,\mcD\,\,
\cfCup{0.5}{}.
\]
A closed surface of odd genus $2n+1$ evaluates to $2 \mcD^n$, of even genus $2n$ to zero. A surface of odd genus $2n+1$ with a dot evaluates to $E_1\mcD^n$, of even genus $2n$ with a dot to $\mcD^n$.

We have 
\begin{equation}
    \sigma(X_{\star})= - X_{\star}, \ \ \epsilon(X_{\star})=2, \ \  \Delta(X_{\star}) = X_1\otimes X_1+X_2\otimes X_2.
\end{equation}

Exact sequences of $A$-bimodules 
\begin{equation}
 0 \lra A\{4\} \stackrel{\Delta}{\lra} A\otimes_R A\{2\} \stackrel{\ell_X-r_X}{\lra} A\otimes A \stackrel{m}{\lra} A \lra 0,  
\end{equation}
where $\{k\}$ is the grading shift up by $k$,
glue into an infinite 2-periodic resolution of the identity $A$-bimodule $A$ by free $A$-bimodules 
\begin{equation}
   \dots \stackrel{\ell_{X_1}-r_{X_1}}{\lra}
   A\otimes_R A\{4\}  \stackrel{m \,\Delta}{\lra} A\otimes_R A\{2\} \stackrel{\ell_{X_1}-r_{X_1}}{\lra} A\otimes_R A \stackrel{m}{\lra} A\lra 0 .
\end{equation}

Two-dimensional TQFT $(R,A)$, with its applications to link homology, was introduced by Bar-Natan~\cite{BNTangle}. It's also the theory labelled $(R_5,A_5)$ in~\cite{MKFrobenius}, with parameters $h,t$ there corresponding to $E_1,-E_2$. 
\vspace{0.1in} 

\hypertarget{frobext2}{{\bf 2.}}
Frobenius  extension $(R_{\mcD},A_{\mcD})$ consists of graded rings   
\begin{itemize}
    \item $R_{\mcD}=R[\mcD^{-1}]=\Z[E_1,E_2,\mcD^{-1}]$, where 
    $\mcD$ is the discriminant, see equation (\ref{eq_discriminant}). Since $\deg(\mcD^{-1})=-4$, graded algebra $R_{\mcD}$ is nontrivial in all even degrees. 
    \item $A_{\mcD}=R_{\mcD}[X]/(X^2-E_1 X + E_2)$.
\end{itemize}
Thus, algebra $R_{\mcD}$ is obtained from $R$ by inverting the discriminant element (localizing along it). Algebra $A_{\mcD}= A \otimes_R R_{\mcD}$, so that the Frobenius extension $(R_{\mcD},A_{\mcD})$ is obtained from $(R,A)$ by inverting the discriminant $\mcD$. 
Algebra $A_{\mcD}$ is a free $R_{\mcD}$-module of rank two with a homogeneous basis $\{1,X\}$. 
Formulas (\ref{eq_trace_map}) and (\ref{equation_comult}) for the trace and the comultiplication hold in this Frobenius extension,  as well as all other formulas derived above for the extension $(R,A)$.

Since $\mcD$ is invertible of degree $4$, there is an  isomorphism  of graded free $R_{\mcD}$-modules $R_{\mcD}\cong R_{\mcD}\{4\}$, so that in the Grothendieck group of graded modules one can work at most over $\Z[q]/(q^4-1)$ rather than over $\Z[q,q^{-1}]$ for graded $R$-modules. 

\vspace{0.07in} 

{\it Separability:} Given a unital homomorphism $\psi: S \lra T$ of commutative rings, one can form the enveloping algebra $T^e=T\otimes_S T$, acting on $T$ by  
\[ (t_1\otimes t_2)t= t_1 t t_2
\] 
and a homomorphism of $T^e$-modules 
\[ m : T\otimes_S T \lra T, \ \ m(a\otimes b) = ab.
\]
Homomorphism $\psi$ is called \emph{separable} if there exists a 
homomorphism $\mu:  T\lra T\otimes_S T$ of $T^e$-modules such that 
$m \circ \mu = \mathrm{Id}_T,$
\[ m\circ \mu  \ : \ T  \lra T\otimes_S T \lra T.
\] 

\emph{Remark:} Separability can be defined for noncommutative $S$-algebras $T$ as well, changing the notion of  the enveloping algebra to $T^e=T\otimes_S T^o,$ where $T^o$ is the opposite algebra of $T$. 

Separability is equivalent to the existence of an idempotent $e\in T\otimes_S T$ such that $m(e)=1$ and $e u=0$ for any $u\in\ker(m)$.
Such $e$ is called a \emph{separability idempotent}. Separability is also equivalent to the condition that the exact sequence of $T^e$-modules 
\[ 0 \lra \ker(m) \lra T\otimes_S T \stackrel{m}{\lra} T \lra 0 
\] 
splits. $T^e$-submodule $\ker(m)$ is generated by the complementary idempotent $1-e$, in the separable case. We refer to \cite{Kadison, FordBook} for more information on separable extensions. 

\begin{prop} The ring extension $R_{\mcD}\lra A_{\mcD}$ is separable.
\end{prop} 

\begin{proof} Specializing the above notations to our case, we have the algebra 
\[
A_{\mcD}^e = A_{\mcD} \otimes_{R_{\mcD}} A_{\mcD}.
\] 
Recall that  $X_{\ast}=X_1-X_2$  is the difference of roots, denoted by star dot $\star$ on a surface. From equation (\ref{eq_discriminant}) 
we see that $\pm X_{\ast}$ are the two square roots in $A$ of the discriminant $\mcD\in R$. 

With $\mcD$ invertible, we can define the inverse of the $\star$ dot as the ratio $X_{\star}/\mcD,$
\begin{equation}\label{eq_inv_star}
\cfDecSquare{0.5}{\star^{-1}} = \mcD^{-1}\ 
\cfDecSquare{0.5}{\star}
\end{equation}
This diagram describes $X_{\star}^{-1} = \mcD^{-1} X_{\star}\in A_{\mcD}$. 
Since $m\Delta=(X_1-X_2)\Id_{A_{\mcD}}=X_{\star}\Id_{A_{\mcD}}$, see equations (\ref{eq_com_mul}) and (\ref{eq_torus}), and $X_{\star}$ is invertible in $A_{\mcD}$, we can rescale the comultiplication to the map 
\begin{equation}
\label{eq_delta_d}\Delta_{\mcD} \ := \ \ell_{X_{\star}}\, \mcD^{-1}\, \Delta = \ell_{X_{\star}^{-1}}\, \Delta \ : \ 
A_{\mcD} \lra A_{\mcD}^e 
\end{equation}
to split the short exact sequence 
\begin{equation} \label{eq_short_ex} 
0 \lra \ker(m) \lra A_{\mcD}\otimes_{R_{\mcD}} A_{\mcD} \stackrel{m}{\lra} A_{\mcD} \lra 0 .
\end{equation}
Here $\ell_{X_{\star}}$ denotes the left multiplication by $X_{\star}$ on 
$A_{\mcD}^e$, but the right multiplication by the same element works as well, since $\Delta$ is a homomorphism of $A_{\mcD}^e$-modules.

Diagrammatically, $\Delta_{\mcD}$ is given by dotting the copants with the star inverse,  
\begin{equation}\label{eq_copants_delta}
    \Delta_{\mcD}=  \cfDelta{0.5}{\star^{-1}}=
    \mcD^{-1} \cfDelta{0.5}{\star},   
\end{equation}
Note that placing a star dot on a surface is equivalent to adding a handle, see (\ref{eq_torus}). One can think of the inverse $\star^{-1}$ on a surface as an \emph{antihandle} (of genus minus one). Informally,  $\Delta_{\mcD}$ is given by copants of genus $(-1)$. In the theory $(R_{\mcD},A_{\mcD})$ we can turn a connected component of a surface into a "negative genus" surface by placing a suitable negative power of the $\star$ dot on it. 

We have $m\Delta_{\mcD}=\mathrm{id}_{A_{\mcD}}$. Diagrammatically, $m\Delta$ adds a handle to the annulus representing the identity map of $\mathrm{id}_{A_{\mcD}}$, the state space of the circle. Presence of the star dot inverse on the surface for $\Delta_{\mcD}$ cancels out that handle. 

The separability idempotent is 
\begin{align}\label{eq:idemp-sep}
e &= \Delta_{\mcD}(1) = \mcD^{-1}\, \Delta(X_{\star}) = 
\mcD^{-1}(X_1\otimes X_1 + X_2 \otimes X_2) \\
  &=\mcD^{-1}(X_1\otimes 1 - 1 \otimes
  X_2)^2 = \mcD^{-1} \Delta(1)^2 \in A_{\mcD}^e. \nonumber \qedhere
\end{align}
\end{proof} 

Diagrammatically, $e$  is a genus one  cobordism from the empty 1-manifold to the union of two circles times the scalar $\mcD^{-1}$. The handle in the genus one cobordism can be substituted by $X_{\star}$, see below, writing $e$ as tube cobordism decorated by $\star$ dot times $\mcD^{-1}$. Equivalently, we can write $e$ as a tube with $\star^{-1}$ dot on it, thus a 'genus minus one' cobordism with two boundary circles at the top.   
\[
e= \mcD^{-1} 
\cfTubeTT{0.5}{\star} = \cfTubeTT{0.5}{\star^{-1}} 
\]

Idempotent $e$ gives a factorization
\[
A_{\mcD}^e = A^e_{\mcD}e \times   A^e_{\mcD}(1-e) \cong A_{\mcD}
\times A_{\mcD}
\]
of the algebra $A_{\mcD}^e$ into the direct product of two copies of $A_{\mcD}$. 

Idempotent $e$ is \emph{strongly separable}, that is, fixed by the involution of $A^e_{\mcD} $ which permutes the two factors in the tensor product \cite[section 5.3]{Kadison}. This is due to $\Delta_{\mcD}(1)$ being invariant under this involution. 

\vspace{0.05in} 

{\it Remark:}
Separability is closely related to vanishing of higher Hochschild homology groups. For a  separable extension $S\subset T$ and any $T$-algebra $A$, there is an isomorphism $\mathrm{HH}^T_n(A)\cong \mathrm{HH}_n^S(A)$, see \cite[Theorem 1.2.13]{Loday}. 
Hochschild had shown \cite{Hochschild}
 that an extension $k\subset T$ of a field $k$ is separable iff the Hochschild homology $\mathrm{HH}_1(T,M)=0$ for all $T$-bimodules $M$. Hochschild homology plays an important role in link homology, starting with the work of Przytycki~\cite{JPKhoHomHochHom} on the relation between Hochschild homology and the homology of $(2,n)$-torus links. 

\vspace{0.05in} 

Our original extension $R\lra A$ is not separable, since it's easy to check that the multiplication map $m:A\otimes A \lra A$ does not split.

\vspace{0.1in} 

\hypertarget{frobext3}{{\bf 3.}}
The third extension $(R_{\alpha},A_{\alpha})$ consists of graded rings   
\begin{itemize}
    \item $R_{\alpha} = \Z[\alpha_1,\alpha_2]$, $\deg(\alpha_1)=\deg(\alpha_2)=2$, 
    \item $A_{\alpha} = R_{\alpha}[X]/((X-\alpha_1)(X-\alpha_2))$, $\deg(X)=2$. 
\end{itemize}
The involution $\alpha_1 \leftrightarrow \alpha_2$, denoted $\sigma_{\alpha}$ and  induced by the permutation of indices,  acts on $R_{\alpha}$, and we  identify the subring of invariants under this involution  with the ring  $R=\Z[E_1,E_2]$ via identifications $E_1= \alpha_1+\alpha_2, E_2 =\alpha_1\alpha_2$. Under the inclusion $R\subset R_{\alpha}$, the latter is a free graded $R$-module of rank $2$ with a homogeneous basis $\{ 1,\alpha_1\}$, for example. 

Recall that $A$ is a graded $R$-algebra as well. Graded $R$-algebras $A$ and $R_{\alpha}$ are isomorphic, and there are two possible isomorphisms, taking $X\in A$ to either $\alpha_1$ or $\alpha_2$ in $R_{\alpha}$. Notice that $X=X_1,E_1-X=X_2$ in $A$ and $\alpha_1, \alpha_2$ in $R_{\alpha}$ are roots of the polynomial  
$y^2-E_1y+E_2$ with coefficients in $R$. Any $R$-algebra isomorphism $A \cong R_{\alpha}$ will take roots $X_1,X_2$ to roots $\alpha_1,\alpha_2$ in some order. 
Despite the existence of these isomorphisms, rings $A$ and $R_{\alpha}$  have different origins and carry different topological interpretations in our story.  This pair of isomorphisms between $A$ and $R_{\alpha}$ does not immediately carry over to the topological side.   

We can rewrite the defining relation in $A_{\alpha}$ as 
\[ 0=(X-\alpha_1)(X-\alpha_2) = X^2 - E_1 X + E_2, 
\] 
recognizing the defining relation in $A$, but over a larger ground ring, to
 obtain a canonical isomorphism 
\begin{equation}\label{eq_A_alpha} 
 A_{\alpha} \ = \ R_{\alpha}\otimes_R A,
\end{equation}
which can be taken as the definition of $A_{\alpha}$. 

Elements $X_1,X_2\in A$ and $\alpha_1,\alpha_2\in R_{\alpha}$ are roots in $A_{\alpha}$ of $y^2-E_1y+E_2$. Furthermore, we have 
\[ 
X_1+X_2 = \alpha_1 + \alpha_2=E_1, \ X_1X_2 = \alpha_1\alpha_2=E_2.
\] 
The four elements $X_i-\alpha_j$, $i,j\in\{1,2\}$, are zero divisors in $R_{\alpha}$: 
\[ (X_1-\alpha_1)(X_1-\alpha_2) = (X_2-\alpha_1)(X_2-\alpha_2) = 0.
\] 
Up to a sign, that's only two zero divisors, since $X_2-\alpha_1= - (X_1 -\alpha_2)$ and $X_2-\alpha_2= - (X_1 -\alpha_1).$

The comultiplication map can be simply written using $X-\alpha_i$, $i=1,2$:  
\begin{eqnarray*}
\Delta(1) & = & (X-\alpha_1) \otimes 1 + 1 \otimes (X-\alpha_2) = (X-\alpha_2) \otimes 1 + 1 \otimes (X-\alpha_1), \\
    \Delta(X-\alpha_1) & = & (X-\alpha_1)\otimes (X-\alpha_1), \\
    \Delta(X-\alpha_2) & = & (X-\alpha_2)\otimes (X-\alpha_2).
\end{eqnarray*}
We can think of $X-\alpha_i$ as shifted dots and denote them on diagrams by a small circle with $i$ in it, $\circled{1}$ and $\circled{2}$. We display below some skein relations on shifted dots. 
In particular, the comultiplication formulas can be interpreted as simple neck-cutting relations.  

\begin{gather*}
\cfCircleId{0.5} =
\cfCircleSplit{0.5}{\circled{1}}{}
+
\cfCircleSplit{0.5}{}{\circled{2}}
=
\cfCircleSplit{0.5}{\circled{2}}{}
+
\cfCircleSplit{0.5}{}{\circled{1}}, \qquad\quad
\cfCircleId[\circled{1}]{0.5}=
\cfCircleSplit{0.5}{\circled{1}}{\circled{1}}, \qquad\quad
\cfCircleId[\circled{2}]{0.5}=
\cfCircleSplit{0.5}{\circled{2}}{\circled{2}}, \\[4pt]
\quad
\cfGenusOneBdy{0.5}= \cfCup{0.5}{\circled{1}} + \cfCup{0.5}{\circled{2}},
\qquad\qquad
\cfSphere{0.6}{\circled{1}} =
\cfSphere{0.6}{\circled{2}} = 1, \\[4pt]
\cfDecSquare{0.5}{\circled{1}\circled{2}} =0,\qquad\quad
\cfDecSquare{0.5}{\circled{1}\circled{1}}=
(\alpha_2-\alpha_1)\cfDecSquare{0.5}{\circled{1}}, \qquad\quad
\cfDecSquare{0.5}{\circled{2}\circled{2}}=
(\alpha_1-\alpha_2)\cfDecSquare{0.4}{\circled{2}}.
\end{gather*}
It may also be convenient to denote $X_{\circled{1}}=X-\alpha_1, X_{\circled{2}}=X-\alpha_2 \in A_{\alpha}$. Relations above can be rewritten algebraically, for example $X_{\circled{1}}X_{\circled{2}}=0.$

Then $X_{\star}=X_{\circled{1}}+X_{\circled{2}}$, that is, 
$ \cfDecSquare{0.5}{\star} = \cfDecSquare{0.5}{\circled{1}} + \cfDecSquare{0.5}{\circled{2}}.$

In the ring $R_{\alpha}$ discriminant factorizes, 
$\mcD= (\alpha_1 - \alpha_2)^2,$ which is just like its factorization $\mcD=(X_1-X_2)^2=X_{\star}^2$ in the isomorphic ring $A$. By analogy, one may also denote $\alpha_{\star}=\alpha_1-\alpha_2$. 

\vspace{0.1in} 

{\it Involutions on $R_{\alpha}$ and $A_{\alpha}$.}
Compared to $R$, the ring $R_{\alpha}$ has an additional symmetry, an involution $\sigma_{\alpha}$, mentioned above, permuting $\alpha_1$ and $\alpha_2$. 
This $R$-linear 
involution extends to $A_{\alpha}$  transposing the roots 
$\alpha_1, \alpha_2$ and fixing $X$:  
\[ \sigma_{\alpha}(\alpha_1)=\alpha_2, \ 
\sigma_{\alpha}(\alpha_2)=\alpha_1, \
\sigma_{\alpha}(X)=X. 
\] 
$\sigma_{\alpha}$ fixes the subalgebra $R$ and permutes $X_{\circled{1}}$ and $X_{\circled{2}}$,
\begin{equation}
    \sigma_{\alpha}(X_{\circled{1}}) = X_{\circled{2}}, \ \ \sigma_{\alpha}(X_{\circled{2}}) = X_{\circled{1}}. 
\end{equation}

We can extend $R$-algebra involution $\sigma$ on $A$ to an $R_{\alpha}$-algebra involution of $A_{\alpha}$, also denoted $\sigma$, with 
\[ \sigma (X -\alpha_1) = \alpha_2 -X, \ 
\sigma(X-\alpha_2) = \alpha_1 - X,
\ \sigma(1)= 1. 
\] 
 Algebra involutions $\sigma$ and $\sigma_{\alpha}$ of $A_{\alpha}$ commute, and their composition 
$\sigma\sigma_{\alpha}$ is an $R$-linear involution which negates shifted dots,
\[ \sigma \sigma_{\alpha} (X_{\circled{1}}) =   -X_{\circled{1}}, \ 
\sigma \sigma_{\alpha}(X_{\circled{2}}) =  -X_{\circled{2}}. 
\] 

We write down a summary of these involutions of algebra $A_{\alpha}$ in Table~\ref{tab:sigmas}.

\begin{table}[ht]
\centering
 \begin{tabular}{| c | c | c | c | c | c | c | c |} 
 \hline
 Involution & linearity & $\alpha_1$ \rule{0pt}{2.5ex}& $\alpha_2$ & $X_1$ & $X_2$ & $X_{\circled{1}}$ &  $X_{\circled{2}}$ \\ [0.5ex] 
 \hline 
 $\sigma$ \rule{0pt}{2.5ex} & $R_{\alpha}$-linear & $\alpha_1$ & $\alpha_2$ & $X_2$ & $X_1$ & $-X_{\circled{2}} $ & $X_{\circled{1}}$ \\ [0.5ex] 
 \hline
 $\sigma_{\alpha} $ \rule{0pt}{2.5ex}& $A$-linear & $\alpha_2$ & $\alpha_1$ & $X_1$ & $X_2$ & $X_{\circled{2}}$ &  $X_{\circled{1}}$ \\ 
 [0.5ex] 
 \hline
 $\sigma\sigma_{\alpha} \rule{0pt}{2.5ex}$ & $R$-linear & $\alpha_2$ & $\alpha_1$ & $X_2$ & $X_1$ & $-X_{\circled{1}}$ &  $-X_{\circled{2}}$ \\
 [0.7ex] 
 \hline
\end{tabular}

\vspace{0.2cm}
\caption{Summary of involutions $\sigma$, $\sigma_\alpha$ and $\sigma\sigma_\alpha$.}
\label{tab:sigmas}
\end{table}
It may also be convenient to relabel $\sigma$ into $\sigma_X$, by analogy with $\sigma_{\alpha}$. 

\vspace{0.1in} 

\emph{Equivariant cohomology:} Passing from $(R,A)$ to $(R_{\alpha},A_{\alpha})$ corresponds to passing from $U(2)$-equivariant cohomology of a point and a 2-sphere to $U(1)\times U(1)$-equivariant cohomology of these topological spaces, where $U(1)\times U(1)$ is a maximal torus in $U(2)$. Group $U(2)$ acts on $\SS^2$ via the identification of the latter with $\mathbb{CP}^1$ and descending from the standard action of $U(2)$ on $\CC^2$.  

\vspace{0.2in} 

\hypertarget{frobext4}{{\bf 4.}}
The fourth extension $(R_{\alpha\mcD},A_{\alpha\mcD})$ consists of graded rings   
\begin{itemize}
    \item $R_{\alpha\mcD} = \Z[\alpha_1,\alpha_2,(\alpha_1-\alpha_2)^{-1}]$, $\deg(\alpha_1)=\deg(\alpha_2)=2$, 
    \item $A_{\alpha\mcD} = R_{\alpha\mcD}[X]/((X-\alpha_1)(X-\alpha_2))$, $\deg(X)=2$. 
\end{itemize}
This extension is obtained from the third extension $(R_{\alpha},A_{\alpha})$ by inverting the discriminant $\mcD= (\alpha_1-\alpha_2)^2.$ Equivalently, one can invert $\alpha_1-\alpha_2$. 
An element of $A_{\alpha\mcD}$ can be written   uniquely as $f_1\cdot 1 + f_2\cdot X$ where $f_1,f_2\in R_{\alpha\mcD}$. 
Let 
\begin{equation}\label{eq_idemp}
    e_1 \ = \frac{X-\alpha_1}{\alpha_2-\alpha_1}, 
    \ \ e_2 \ = \frac{X-\alpha_2}{\alpha_1-\alpha_2}.
\end{equation}
These two elements of $A_{\alpha\mcD}$  are mutually-orthogonal complementary idempotents, 
\begin{equation}
    1= e_1 + e_2, \ e_1^2 = e_1, \ e_2^2=e_2, \ e_1 e_2 = e_2 e_1 = 0. 
\end{equation}
Consequently, the ring $A_{\alpha\mcD}$ decomposes as the direct product 
\begin{equation}\label{eq_dir_prod} 
    A_{\alpha\mcD} \ = \ R_{\alpha\mcD}e_1 \times R_{\alpha\mcD}e_2. 
\end{equation}
The entire Frobenius algebra structure decouples as well:  
\begin{eqnarray}
    & & \Delta(e_1) \ = \ (\alpha_2-\alpha_1)\cdot e_1\otimes e_1   , \ \ \Delta(e_2) \ = \ (\alpha_1-\alpha_2)\cdot e_2\otimes e_2  ,\\
    & & \epsilon(e_1)= (\alpha_2-\alpha_1)^{-1}, \ \ \epsilon(e_2)= (\alpha_1-\alpha_2)^{-1}. 
\end{eqnarray}
Both $X_{\star}=X_1-X_2$ and $\alpha_{\star}$ are invertible in $A_{\alpha\mcD}$, since both square to $\mcD$. The ratio 
\[
\frac{X_{\star}}{\alpha_{\star}}= X_{\star}\alpha_{\star}^{-1}= \frac{2X-E_1}{\alpha_1-\alpha_2}= e_2-e_1 
\]
is a degree zero invertible element other than $-1$ that squares to $1$. Elements 
\[ 
\{ 1, X_{\star}\alpha_{\star}^{-1}, -1, 
- X_{\star}\alpha_{\star}^{-1} \}  = \{ 1, e_2-e_1,-1, e_1-e_2 \}
\] 
constitute a subgroup of degree zero invertible elements in $A_{\alpha\mcD}$ isomorphic to $\Z/2\times \Z/2$. Since the ring $A_{\alpha\mcD}$ is not an integral domain, it may contain a finite non-cyclic subgroup of invertible elements. Since $X_{\star}$ is invertible, it can be put in the denominator instead of $\alpha_{\star}$ to recover the same idempotents: 
\begin{equation}
    e_1 \ = \frac{\alpha_1-X_1}{X_2-X_1}, 
    \ \ e_2 \ = \frac{\alpha_1-X_2}{X_1-X_2}.
\end{equation}

\vspace{0.1in}

\subsection{Summary of basic properties} In Figure~\ref{fig_bigger-cube} we redraw the cube of Frobenius extensions, adding some information about the rings involved. 

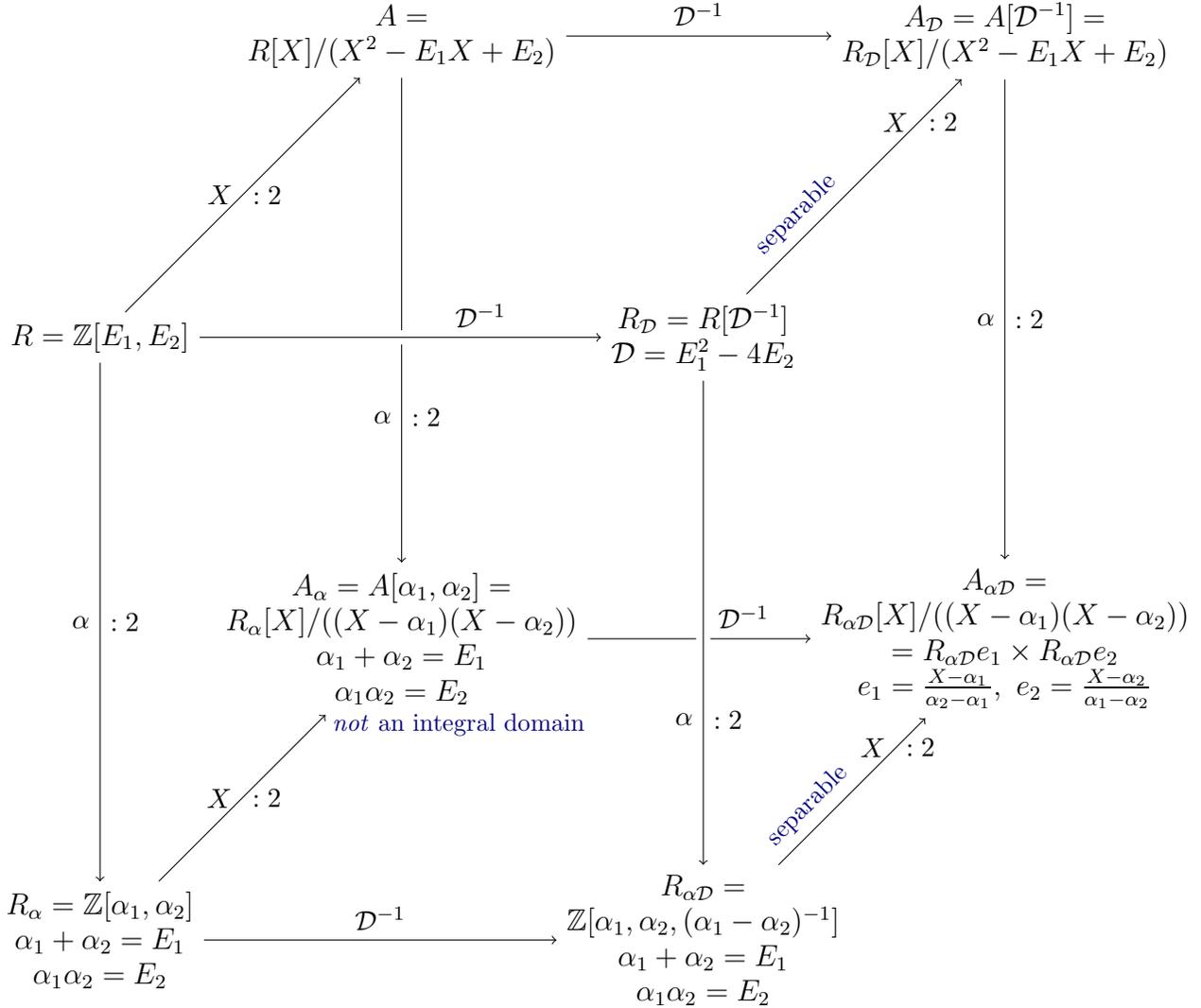
\begin{figure}
    \centering
\NB{\tikz[yscale =1.06, xscale=1.06]{\input{\imagesfolder/fe_biggercube}}}
    \caption{Cube of the four Frobenius extensions, with the rings defined}\label{fig_bigger-cube}
\end{figure}

Let us summarize this diagram. In it, three of the Frobenius extensions listed in (\ref{eq_4_extensions}) are obtained by base changes from the original extension $(R,A)$: 
\begin{equation} 
(R_{\alpha},A_{\alpha})\cong R_{\alpha}\otimes_R(R,A), \   (R_{\mcD},A_{\mcD})\cong R_{\mcD}\otimes_R(R,A), \ (R_{\alpha\mcD},A_{\alpha\mcD})\cong R_{\alpha\mcD}\otimes_R(R,A). 
\end{equation} 
The original extension $(R,A)$ is given by the upper left diagonal arrow. 

\vspace{0.05in} 

{\it Diagonal arrows: }
Diagonal arrows in the diagram, pointing northeast, denote inclusions $R_{\ast}\subset A_{\ast}$ for each of the four Frobenius extensions (\ref{eq_4_extensions}). In each case $A_{\ast}$ is a free $R_{\ast}$-module of rank two with a basis $\{1,X\}$. The rank and the additional basis (and generating) element $X$ are indicated by writing "$:2$" and "$X$" next to the arrows. 

\vspace{0.05in} 

{\it Horizontal arrows:} The four horizonal arrows correspond to localizing each of the four commutative rings $R,A,R_{\alpha},A_{\alpha}$ in the vertices of the left square facet of the cube by inverting the discriminant $\mcD\in R$. Each arrow is the inclusion of one of these four rings into the localized ring. We put $\mcD^{-1}$ by these arrows to indicate this operation. The two Frobenius extensions on the right side of the diagram are separable, but not the two on the left. 

\vspace{0.05in} 

{\it Downward arrows:}
The four vertical downward arrows correspond to enlarging each of the four rings in the vertices of the top square by adding roots $\alpha_1,\alpha_2$ of polynomial $y^2-E_1 y + E_2$. This extension of rings  is a rank two extension, with the larger ring a free rank two module over the smaller ring with a basis $\{1,\alpha_i\}$ for $i\in \{1,2\}$. We indicated this transformation by writing "$\alpha$" and "$:2$" next to each vertical arrow. In particular, the leftmost vertical arrow denotes the inclusion 
\[ R = \Z[E_1,E_2]  \ \subset R_\alpha = \Z[\alpha_1,\alpha_2], \  \alpha_1+\alpha_2 = E_1, \ \alpha_1\alpha_2 = E_2. 
\] 
Thus, downward vertical arrows denote an extension where we tensor each of the  four rings $R,A,R_{\mcD},A_{\mcD}$ with $R_\alpha$ over $R$. 
In particular, 
\[ A_{\alpha} = A\otimes_R R_{\alpha}, \ \
    R_{\alpha \mcD}= R_{\mcD}\otimes_R R_{\alpha}, \ \ 
    A_{\alpha \mcD} = A_{\mcD}\otimes_R R_{\alpha}.
\] 
As we've mentioned, $R_{\alpha}$ is a free rank two graded $R$-module with a basis $\{1,\alpha_1\}$ (or $\{1,\alpha_2\}$). In fact, $(R,R_{\alpha})$ is a rank two Frobenius extension isomorphic to the extension $(R,A)$. There are two $R$-algebra isomorphisms $A\cong R_{\alpha}$: one of them takes $X$ to $\alpha_1$ and, necessarily, $E_1-X$ to $\alpha_2$, while the other takes the roots $(X,E_1-X)$ of $x^2-E_1x +E_2$ to $(\alpha_2,\alpha_1)$, correspondingly. We will not be using these isomorphisms in the paper. 

\vspace{0.05in}

{\it Zero divisors and idempotents:} Ring $A_{\alpha}$ contains zero divisors $\pm(X-\alpha_1),$ $\pm(X-\alpha_2)$, with $(X-\alpha_1)(X-\alpha_2)=0$. In particular, it's not an integral domain. Additionally, with the determinant inverted, these zero divisors rescale into idempotents $e_1,e_2$ in $A_{\alpha\mcD}$, where they decompose the  ring  into the direct product of two copies of the  ground ring $R_{\alpha\mcD}$. 

\vspace{0.05in}

{\it Pushouts:} Rings $R_{\alpha\mcD}$ and $A_{\alpha\mcD}$, together with the maps into them in the front and back squares of the diagram in Figures~\ref{fig:first-cube} and~\ref{fig_bigger-cube}, are pushouts (and colimits) of commutative ring diagrams 
\begin{equation}
 R_{\alpha}\longleftarrow R\lra R_{\mcD}, \ \ \ A_{\alpha}\longleftarrow A\lra A_{\mcD},
\end{equation}
correspondingly. In the extension $R\subset R_{\alpha}$, the ring $R_{\alpha}$ is free of rank two as an  $R$-module. In particular, this is a flat and a faithfully flat extension, ditto for $A\subset A_{\alpha}$. 
The inclusion $R\subset R_{\mcD}$ is a localization of commutative rings and thus a flat extension. 
The back square of the four $A_{\ast}$'s in the diagram is obtained via the base change $R\lra A$ from the front square pushout of commutative rings $R_{\ast}$'s.

\vspace{0.05in}

{\it Diagonal plane symmetries of the cube:} As mentioned earlier, rings $A$ and $R_{\alpha}$ are isomorphic as $R$-algebras, via two possible algebra automorphisms that take $X$ to $\alpha_1$ or $\alpha_2$ and $E_1-X$ to $\alpha_2$ or $\alpha_1$, correspondingly: 
\begin{eqnarray*}
& & \tau_1: A\stackrel{\cong}{\lra} R_{\alpha}, \ \ 
 \tau_1(X) = \alpha_1, \ \tau_1(E_1-X) =\alpha_2 , \\
& & \tau_2: A\stackrel{\cong}{\lra} R_{\alpha},\ \ \tau_2(X) = \alpha_2, \ \tau_2(E_1-X) =\alpha_1 .
\end{eqnarray*}
Composition $\tau_1\tau_2^{-1}=\tau_2\tau_1^{-1}$ is an $R$-linear automorphism of $R_{\alpha}$ that transposes $\alpha_1,\alpha_2$. Likewise, 
$\tau_1^{-1}\tau_2=\tau_2^{-1}\tau_1$ is an $R$-linear automorphism of $A$ that transposes $X_1=X$ and $X_2=E_1-X$. 

Ring $A_{\alpha}$ is a free $R$-module of rank four with a basis, for instance, $\{1,X,\alpha_1,\alpha_1 X\}$. Isomorphisms $\tau_1,\tau_1^{-1}$ extend to a algebra involution  of $A_{\alpha}$, also denoted $\tau_1$, given by  
\begin{equation}
    \tau_1(\alpha_1)=X, \ \tau_1(X)=\alpha_1, \ \tau_1(\alpha_2)= E_1-X, \ 
    \tau_1(E_1-X) =\alpha_2, 
\end{equation}
or, in the other notation, 
\begin{equation}
\tau_1: \alpha_1\leftrightarrow X_1, \ \alpha_2 \leftrightarrow X_2. 
\end{equation}
Thus, $\tau_1$ adds sign to $X-\alpha_1$ but fixes $X-\alpha_2$ in $A_{\alpha}$. It acts with eigenvalue $-1$ on the subspace  $R(X-\alpha_1)$ and  with eigenvalue $1$ on $R(X-\alpha_2)$. Likewise, $\tau_2(X-\alpha_1)=X-\alpha_1$ and $\tau_2(X-\alpha_2)=-(X-\alpha_2)$.  

Isomorphisms $\tau_1,\tau_2$ of $A_{\alpha}$ extend to isomorphisms from $A_{\mcD}$ to $R_{\alpha\mcD}$ and to automorphisms of $A_{\alpha\mcD}$. They also restrict to identity automorphisms of both $R$ and $R_{\mcD}$. 

In this way, they act on the entire cube in Figure~\ref{fig_bigger-cube}, as automorphisms of algebras $R,A_{\alpha},R_{\mcD}$, and $A_{\alpha\mcD}$ (in the four vertices of the cube that lie on a plane through two horizontal edges),  isomorphisms between $A$ and $R_{\alpha}$ and isomorphisms between $A_{\mcD}$ and $R_{\alpha\mcD}$.  

\vspace{0.05in}

{\it Properties of link homology for these four extensions.} 

Given an oriented link $L$ and its diagram $D$ with $n$ crossings, we can  construct $2^n$ resolutions of $D$ into diagrams of planar circles and then use the Frobenius pair $(R,A)$ and associated 2-dimensional TQFT to build a commutative cube with tensor powers of $A$ (with grading shifts) placed in the vertices. Collapsing into a complex, we obtain a complex $C(D)$ of free graded $R$-modules. Choosing a base point on $D$ which is not a crossing turns $C(D)$ into a complex of graded $A$-modules. Homology $\mtH(D)$ of $C(D)$ is a bigraded $R$-module, which is almost never free. With a choice of base point, $\mtH(D)$ is a graded $A$-module. 

To the Reidemeister moves $D_1\sim D_2$ one assigns chain homotopy equivalences $C(D_1)\cong C(D_2)$ between complexes of free graded $R$-modules, with induced isomorphisms on bigraded homology group $\mtH(D_1)\cong \mtH(D_2)$. With a base point and a Reidemeister move away from the base point, homotopy equivalences and homology isomorphisms become that of $A$-modules. For  a $k$-component unlink $U_k$, the homology  $\mtH(U_k)\cong A^{\otimes k}$, a free $A$-module of rank $2^{k-1}$. 

For two basepoints $p_1,p_2$ located on the same component and separated by a single crossing, multiplication maps on $C(D)$ by $X_1$ at $p_1$ and by $X_2$ at $p_2$ are  chain homotopic. 

\vspace{0.1in} 

Base changes from $R$ to $R_{\alpha}$, $R_{\mcD}$, $R_{\alpha\mcD}$ allow to define chain complexes and the corresponding homology groups 
\begin{eqnarray}  
    C_{\alpha}(D) & = & R_{\alpha}\otimes_R C(D) , \ \ \ \mtH_{\alpha}(D)  =  \mtH(C_{\alpha}(D)),  \\
    C_{\mcD}(D) & = & {R_{\mcD}}\otimes_R C(D) , \ \ \ \mtH_{\mcD}(D)  =  \mtH(C_{\mcD}(D)), \\
    C_{\alpha\mcD}(D) & = & {R_{\alpha\mcD}}\otimes_R  C(D) , \ \ \mtH_{\alpha\mcD}(D)  =  \mtH(C_{\alpha}(D)).
\end{eqnarray}
We can arrange these four types of complexes and groups into a three-dimensional cube, where diagonal wiggly arrows denote passage to homology. 
\[
\NB{\tikz[scale=1]{\input{\imagesfolder/fe_cube-homologies}}}
\]

The four terms in the vertices of the front square are complexes $C_{\ast}(D)$
obtained from $C(D)$ by suitable base changes $R\lra R_{\ast}$. They are complexes of free graded $R_{\ast}$-modules. Define the four homology  groups $\mtH_{\ast}(D)$ in the back square as the homology of these complexes. They are naturally modules over $R_{\ast}$ and, if a base point is  picked, modules  over $A_{\ast}$. 

Flatness  of the extensions $R\lra R_{\ast}$ implies that the natural maps 
\begin{equation} \label{eq_D_iso}
\mtH(D) \otimes_R R_{\ast} \lra \mtH_{\ast}(D) 
\end{equation}  
are isomorphisms of $R_{\ast}$-modules, also giving isomorphisms passing  from diagrams to links 
\begin{equation}\label{eq_L_iso}
\mtH(L) \otimes_R R_{\ast} \stackrel{\cong}{\lra} \mtH_{\ast}(L)
\end{equation} 

These homology theories extend to tangles as usual \cite{KhovFVIT}, by first extending $SL(2)$-equivariant rings $H^n$ over $R$ (with $H^1\cong A$) via base change $R\lra R_{\ast}$ to rings $H^n_{\ast}$. Bimodules and bimodule homomorphisms for flat tangles and their cobordisms extend as well, resulting in suitable 2-functors from the 2-category of flat tangle cobordisms to the 2-category of homogeneous bimodules over $H^n_{\ast}$, over all $n$. 

Extension to tangle cobordisms also works as usual. Due to  flatness of our base changes, invariance of maps induced by tangle cobordisms up to an overall sign follows from  the corresponding invariance for the theory built out  of $(R,A)$. 
For the invariance of link and tangle cobordism  maps we refer to the original papers~\cite{BNTangle, Jacobsson, KhTangleCob}.
This invariance up to a sign can be stated, for tangle cobordisms, in the language of maps  between complexes of graded  bimodules in the homotopy category. 

Isomorphisms (\ref{eq_L_iso}) are functorial in the link cobordism category, up to an overall sign. 
 
 \vspace{0.1in} 

Taking care of the sign requires modifying the theory and 
working with seamed circles or defect lines as in Caprau~\cite{CaprauTangle,CaprauWebFoams},  Clark--Morrison--Walker~\cite{ClarkMorrisonWalkerFunctoriality} and Vogel~\cite{vogel2015functoriality}, or $GL(2)$ foams as in Blanchet~\cite{Blanchet}. Defect lines in these theories are discussed  below, in Section~\ref{sec:ev}, mostly via the  evaluation  approach simplified from $GL(2)$ foams to surfaces  with seamed circles. 

\vspace{0.1in} 

{\it  $\alpha$-homology:}
As a complex of graded $R$-modules, 
\[ C_{\alpha}(D) \cong C(D)\cdot 1 \oplus C(D)\cdot \alpha_1 , 
\] 
thus isomorphic to two copies of $C(D)$, with a shift. Multiplication by $\alpha_1$ commutes with  taking cohomology. As graded $R$-modules, 
\[
\mtH_{\alpha}(D) \cong \mtH(D) \oplus \mtH(D)\cdot \alpha_1,
\] 
and, with a choice of a base point in $D$, that's an isomorphism of graded $A$-modules. 

We  see that the $\alpha$-homology $\mtH_{\alpha}(L)$, as a graded $A$-module, is isomorphic to the sum of two copies of $\mtH(L)$, one with a grading shift by $\{2\}$, the degree of $\alpha_1$, but has an additional symmetry induced by the transposition of $\alpha_1$ and $\alpha_2$ and boasts zero divisors in the homology of the unknot, viewed as a commutative Frobenius algebra via the pants, cup and cap cobordisms. 

\vspace{0.1in} 

{\it $\alpha\mcD$-homology:} Homology theory $\mtH_{\alpha\mcD}$ associated to the pair $(R_{\alpha\mcD},A_{\alpha\mcD})$ is essentially the  Lee homology. 
The  decomposition of $A_{\alpha\mcD}$ into  the direct  product (\ref{eq_dir_prod}) via the  idempotents in equation (\ref{eq_idemp}) parallels corresponding decomposition in the Lee homology. In  particular, 
arguments in Lee~\cite{Lee} and Rasmussen~\cite{RasmussenSlice} apply here as well and show that $\mtH_{\alpha\mcD}(L)$ is a free $R_{\alpha\mcD}$-module of rank $2^k$, where $k$ is the number of connected components of link $L$. 

Multiplication by $\mcD$ is a degree four isomorphism of $\mtH_{\alpha\mcD}(L)$. Quantum grading, lifting powers of $q$, is $\Z/4$-periodic in this theory. Invertibility  of $\mcD$ can be used to reduce the homology to one with a $\Z/4$-grading in the $q$-direction, see the remark below.  

Furthermore, 
\begin{equation}
  \mtH_{\alpha\mcD}(L) \cong \mtH_{\mcD}(L) \otimes_{R_{\mcD}} R_{\alpha\mcD} = 
  \mtH_{\mcD}(L)\cdot 1 \oplus \mtH_{\mcD}(L)\cdot  \alpha_1.
\end{equation}
Consequently, $\mtH_{\mcD}(L)$ is "half the size" of $\mtH_{\alpha\mcD}(L)$. 

\vspace{0.1in} 

\emph{Remark:} There are versions of
$C_{\mcD}(D)$ and  $\mtH_{\mcD}(D)$ chain complexes and homology groups with the  $q$-grading reduced to $\ZZ/4$. To define this theory, form the $\ZZ/4$-graded quotient ring $R_{\lambda}=R/(\mcD-\lambda)$ (another convenient notation is $R_{\lambda,\mcD}$), where $\lambda\in \{1,-1\}$ is an invertible degree zero element of $R$ (or a more general invertible element of a ground ring different  from $R$). We pick invertible $\lambda$ to make $\mcD$  invertible in $R_{\lambda}$.  Ring $A_{\lambda}=A\otimes_R R_{\lambda}$, 
which can also be denoted $A_{\lambda,\mcD}$, is defined via the base  change. The pair $(R_{\lambda},A_{\lambda})$ is a rank two Frobenius extension. Since $\deg(\mcD)$ is four, rings $R_{\lambda},A_{\lambda}$ are  $\ZZ/4$-graded only, ditto for the complexes  $C_{\lambda}(D)=C(D)\otimes_R R_{\lambda}$ and $\mtH_{\lambda}(D)=\mtH(C_{\lambda}(D)).$ A similar quotient  construction works for the Lee homology $\mtH_{\alpha\mcD}(D)$.

\section{Ring and module involutions and defect lines} \label{sec_fine_str}

\subsection{Involution \texorpdfstring{$\sigma$}{sigma} and Galois action.}
\label{subsec_invol_galois} 
    
Algebra $A$ carries an $R$-linear involution $\sigma$ given by $\sigma(X)=E_1-X$. As an algebra involution, it satisfies $\sigma(1)=1$. It transposes the roots $X_1=X,X_2=E_1-X$ in $A$ of the
polynomial $y^2-E_1y+E_2$ with coefficients in $R\subset A$, 
$\sigma(X_1)=X_2, \sigma(X_2)=X_1$. Ground ring $R$ is also the subring of $\sigma$-invariants of $A$. Notice that the $(-1)$-eigenspace of $\sigma$ is zero, which is not surprising since we are not working over a field. 

Via a base change $R\lra R_{\ast}$ this involution extends to an involution 
$\sigma$ of the $R_{\ast}$-algebra $A_{\ast}$, for any of the four Frobenius extensions $(R_{\ast},A_{\ast})$ in (\ref{eq_4_extensions}). In fact, it works for any base change $R\lra R_{\ast}$ of commutative rings, but we restrict to these four cases.  

Involution $\sigma$ acts $R_{\ast}$-linearly on the free $R_{\ast}$-module $A_{\ast}$ with a basis $\{1,X\}$. In this basis the action is given by the matrix 
\begin{equation}
    \sigma \ \longmapsto \ \left( \begin{matrix} 1 & E_1 \\ 0 & -1 \end{matrix} \right) \ .
\end{equation}

Let $A_{\ast}[\sigma]$ be the crossed product of $A_{\ast}$ with the group ring of the order two group $\ZZ_2$ generated by $\sigma$. Elements of $A_{\ast}[\sigma]$ have the form $a_0\cdot 1 + a_1\cdot \sigma$, with the multiplication rule that  moving  $\sigma$ to the right  of $a\in A_{\ast}$ produces $\sigma(a) \cdot \sigma$: 
\begin{equation}\label{eq_move_sigma0}
    b \cdot \sigma \cdot a = b\,  \sigma(a)\cdot \sigma,  \ a,b\in A_{\ast},\ \ \sigma^2=1. 
\end{equation}
This gives an $R_{\ast}$-linear action of the cross product $A_{\ast}[\sigma]$ on $A_{\ast}$, that is, a homomorphism 
from the cross-product to the endomorphism algebra
\begin{equation}
    A_{\ast}[\sigma] \lra \mathrm{End}_{R_{\ast}}(A_{\ast}) \cong 
    \mathrm{Mat}(2,R_{\ast}), 
\end{equation}
the latter isomorphic to the algebra of $2\times 2$ matrices with coefficients in $R_{\ast}$, where we picked the basis $\{1,X\}$ of $A_{\ast}$ as a free $R_{\ast}$-module. 

Algebra $A_{\ast}[\sigma]$ is four-dimensional over $R_{\ast}$, with a basis $\{1, \sigma, X, X \sigma\}$. Its basis elements act on $A_{\ast}=R_{\ast}\cdot 1 \oplus R_{\ast}\cdot X$ as the following $2\times 2$ matrices.
\begin{equation} \label{eq_4_matrices} 
1 \ \longmapsto \ \left( \begin{matrix} 1 & 0 \\ 0 & 1 \end{matrix} \right) , \ 
    \sigma \ \longmapsto \ \left( \begin{matrix} 1 & E_1 \\ 0 & -1 \end{matrix} \right) , \ 
    X \ \longmapsto \ \left( \begin{matrix} 0 & -E_2 \\ 1 & E_1 \end{matrix} \right) , \ 
    X\sigma \ \longmapsto \ \left( \begin{matrix} 0 & E_2 \\ 1 & 0 \end{matrix} \right) . 
\end{equation}
In the definition of a Galois extension of a commutative ring below we specialize the the case of a cyclic group of order two. 
\begin{dfn}[{\cite{AusGold, ChaseHarriRosen},{\cite[Theorem 12.2.9]{FordBook}}}]
\label{def_sep}
Suppose given an extension of commutative rings $S\subset T$ with an involution $\tau$ acting $S$-linearly on $T$ with $S$ the fixed subring and $T$ a projective $S$-module.  $(S,T)$ is called a \emph{Galois extension} of commutative rings with the cyclic Galois group $C_2=\{1,\tau\}$ of order two if the natural $S$-algebra homomorphism $T[\tau]\lra \End_S(T)$ is an isomorphism. 
Here $T[\tau]$ is the cross product of $T$ with the group ring of the order two group $\{1,\tau\}$. 
\end{dfn}
  
\begin{prop} $(R_{\ast},A_{\ast})$ is a Galois extension with the cyclic group $\{1,\sigma\}$ if and only if the discriminant $\mcD=E_1^2-4E_2$ is invertible in $R_{\ast}$. 
\end{prop} 

\begin{proof} The question is whether the four matrices in  (\ref{eq_4_matrices}) are a basis of the free $R_{\ast}$-module $\mathrm{Mat}(2,R_{\ast})$. Writing down these matrices in the column form, we get the matrix 
  \[
   U =  \begin{pmatrix}
      1 & 1 &0 & 0 \\
      0 & E_1 & -E_2 & E_2 \\
      0 & 0 & 1 & 1 \\
      1 & -1 & E_1 & 0
    \end{pmatrix}.
  \]
  A square matrix with coefficients in a  commutative ring $S$ is invertible if and only if its determinant is invertible in $S$. 
The determinant of $U$ is $-\mcD$, thus invertible iff $\mcD$ is. 
\end{proof}

Among our four rings $R_{\ast},$ the discriminant is invertible in $R_{\mcD}$ and $R_{\alpha \mcD}.$

\begin{cor} Extensions $(R_{\mcD},A_{\mcD})$ and $(R_{\alpha\mcD},A_{\alpha\mcD})$ are Galois with the involution $\sigma$. Extensions $(R,A)$ and $(R_{\alpha},A_{\alpha})$ are not Galois with this involution. 
\end{cor} 

\vspace{0.1in} 

{\it Field extensions of degree two:} Let $F$ be a field. A homomorphism $\psi:R\lra F$ is determined by $a_1=\psi(E_1),a_2=\psi(E_2)$ in  $F$. To $\psi$ we can associate the base change ring $A_{\psi}= A\otimes_R F,$ which is an $F$-vector space with basis $\{1,X\}$ and multiplication $X^2=a_1 X - a_2$. The pair $(F,A_{\psi})$ is a field extension of degree two iff the 
polynomial $f(y) = y^2-a_1 y + a_2$ is irreducible in $F$, that is, does not have a root. 

Assume that field $F$ has characteristic other than two. Then 
for the $F$-algebra $A_{\psi}$ there are three possibilities: 
\begin{enumerate}
    \item $A_{\psi}$ is a field. This happens exactly when $\psi(\mcD)=a_1^2-4a_2$ is not a square in $F$, and then $A_{\psi}\cong  F[\sqrt{\psi(\mcD)}].$
    \item $A_{\psi}=Fe_1\times Fe_2$ is the product of two copies of $F$. This happens when $f(y)$ has two distinct roots $y_1,y_2$ in $F$, with the idempotents 
    \begin{equation}
        e_1=\frac{X-y_1}{y_2-y_1} \ \  \mathrm{and} \ \  
    e_2=\frac{X-y_2}{y_1-y_2}, \ \ 1 = e_1+e_2, \ e_i e_j = \delta_{i,j} e_i, \ i,j\in\{1,2\}.
    \end{equation} 
    Equivalently, $\psi(\mcD)\not=0 $ has a square root in  $F$. 
    \item $A_{\psi}\cong F[z]/(z^2)$ is a nilpotent extension of $F$ by a nilpotent element of order $2$. Equivalently, 
    $\psi(\mcD)=0$, that is, $a_1^2-4a_2=0$, and we can take $z= X -\frac{a_1}{2}$ for this   generator. 
\end{enumerate}
Cases (1) and (2) happen when the homomorphism $\psi:R\lra F$ extends to localization 
$\psi:R_{\mcD}\lra F$. 
 
Assume that $\mathrm{char}(F)=2$. The image  of the discriminant $\psi(\mcD)=a_1^2$ is a square. For  the $F$-algebra $A_{\psi}$ there are the following possibilities:
\begin{enumerate}
    \item $A_{\psi}$ is a field, that is, $y^2+a_1y+a_2$ has no roots in $F$. There are two cases. 
    \begin{enumerate}
        \item $a_1\not= 0$. Replacing $y$ by $a_1 z$ reduces the equation to $z^2+z+c=0$, $c=a_2 a_1^{-2}$, with no solutions in $F$ (which requires $c$ not to be in the image of the  map  $F\to F, z \mapsto z^2+z$).  Extension $(F,A_{\psi})$ is separable. 
        \item $a_1=0$ and $a_2$ is not a square in $F$. Extension $(F,A_{\psi})$ is inseparable.
    \end{enumerate}
    \item $A_{\psi}=Fe_1\times Fe_2$ is the product of two copies of $F$. This happens when $f(y)=y^2+a_1y+a_2$ has two distinct roots $y_1,y_2$ in $F$. The idempotents for the direct product decomposition are 
    \begin{equation}
        e_1=\frac{X-y_1}{y_2-y_1}, \ \  \mathrm{and} \ \  
    e_2=\frac{X-y_2}{y_1-y_2}.
    \end{equation} 
    \item $A_{\psi}\cong F[z]/(z^2)$ is a nilpotent extension of $F$ by a nilpotent element of order $2$. For this we need $a_1=0$ and $a_2=a^2$ to have a square root in $F$. One can  then take $z= X+a$ for the above isomorphism. 
\end{enumerate}
Among these four cases in characteristic two, (1a) and (2) correspond to $\psi(\mcD)\not=0$, giving a separable field extension and a direct product decomposition, respectively. When $\psi(\mcD)=0$, we are either in (1b) or (3), that is, $A_{\psi}$ is an inseparable field extension or a nilpotent extension of $F$. 

Looking across all characteristics, the case $\psi(\mcD)=0$ corresponds to a nilpotent extension (3) or an inseparable extension (1b). 
\vspace{0.1in}

\vspace{0.1in} 

\emph{Remark:} As is well-known, separability property in Definition~\ref{def_sep} for  extensions $S\subset T$ with the Galois group $\Z/2\cong\{1,\tau\}$
is equivalent to the following:  $S=T^{\tau}$ and there exist $x_1,x_2,y_1,y_2\in T$ such that 
\begin{equation}
     x_1 g(y_1)+ x_2 g(y_2)= \left\{  \begin{array}{ll} 1 & \mathrm{if }\ g=1 \\ 0 & \mathrm{if }\ g=\tau 
    \end{array} \right. 
\end{equation}
It is similar but not the same as our neck-cutting relations for $(R,A)$, see equations (\ref{eq_id_map}) and  (\ref{eq_neck_cut}). For comparisons and analogies between Frobenius and separable extensions we refer to~\cite{Kadison, CaenIonMili} and references therein.

\vspace{0.1in}

\subsection{Involution \texorpdfstring{$\sigma$}{sigma+-} and defect lines.} \label{sec_inv_d}
For a TQFT interpretation, it's convenient to see $\sigma$ as an
involution of the $R$-module $A$, not of $A$ as an algebra. Denote by $\sigma_+=\sigma$ this $R$-module map and
by $\sigma_-$ the $R$-module map $-\sigma$. We have 
\[
  \sigma_+(1)=1, \quad \sigma_+(X) = E_1-X,\qquad   \sigma_-(1)=-1, \quad
  \sigma_-(X) = X- E_1. 
\]
Also, $\sigma_{\pm}(a) = \pm (\epsilon(X a) - X\epsilon(a)),$ for $a\in A$. 
Pictorially, this reads 
\[
  \sigma_+ = \cfCircleSplit{0.5}{}{\bullet} -
 \cfCircleSplit{0.5}{\bullet}{}
 \qquad \text{and} \qquad
   \sigma_- = \cfCircleSplit{0.5}{\bullet} {}-
 \cfCircleSplit{0.5}{}{\bullet}.
\]
We introduce a diagrammatic notation for $\sigma_\pm$ as a defect line on an annulus with a choice of co-orientation, that is, a preferred side:
\[
  \sigma_+ = \NB{\tikz[scale=0.5]{\input{\imagesfolder/fe_sigmaplus}}},\qquad \qquad
  \sigma_- = \NB{\tikz[scale=0.5]{\input{\imagesfolder/fe_sigmaminus}}}.
\]
Co-orientation at this defect line can be reversed at the cost of adding a minus sign. 
 
 \[
  \NB{\tikz[scale=0.5]{\input{\imagesfolder/fe_sigmaplus}}} =-
  \NB{\tikz[scale=0.5]{\input{\imagesfolder/fe_sigmaminus}}}.
\]
One can also call these defect or seam lines  \emph{$\sigma$-defect lines}.
Notice that the diagrams for $\sigma_+$ and $\sigma_-$ are not diffeomorphic rel boundary. One can be taken to the other by reflection  in a horizontal plane, but that map is not the identity on the boundary. 
Sliding a dot through a $\sigma$-defect line converts it to the dual dot. Namely, as endomorphisms of $A$, 
\[  \sigma_+ \, X_1 = X_2 \, \sigma_+, \ \ \sigma_- \, X_1 = X_2 \, \sigma_-,
\]
and likewise with indices $1,2$ transposed.  Diagrammatically, 
\[
\cfSeamSquareTwo{0.6}{\bullet}{} = 
    \cfSeamSquareTwo{0.6}{}{\circ}\ , 
    \ \ \ 
   \cfSeamSquareTwo{0.6}{\circ}{} = 
    \cfSeamSquareTwo{0.6}{}{\bullet} \ .
    \] 
Sliding a star dot through a $\sigma$-defect line flips its co-orientation.  As endomorphisms of $A$, 
\[  \sigma_+ \, X_{\star} = X_{\star} \, \sigma_-, \ \ \sigma_- \, X_{\star} = X_{\star} \, \sigma_+ . 
\]
Alternatively, one can 
add a minus sign as star dot is slid:
\begin{equation}\label{eq_star_dot_flip}
\cfSeamSquareTwo{0.6}{\star}{} = - \ 
    \cfSeamSquareTwo{0.6}{}{\star}\ .
\end{equation} 

\vspace{0.1in} 

The involutory property of $\sigma_{\pm}$ says that two parallel defect circles, with co-orientation pointing in the same direction, can be removed. Equivalently, if two parallel defect circles both point either into or out of the annulus region separating them, they can be removed with a minus sign, see equation (\ref{eq:remove-2-circles}) below.  

Although 
$\sigma_+$ is an algebra involution, while $\sigma_-$ is a coalgebra involution, so that 
\begin{eqnarray}
  & & \sigma_{\pm}(1) =\pm 1, \ \
\sigma_{\pm} \circ m \circ (\sigma_{\pm}\otimes \sigma_{\pm})  = \pm m,\label{eq_sigma_m} \\
      & & \epsilon \circ \sigma_{\pm} = \mp \epsilon, \ \ 
       (\sigma_\pm  \otimes \sigma_\pm) \circ \Delta \circ \sigma_\pm
      =  \mp\Delta,\label{eq_sigma_d} 
\end{eqnarray}
 it's more natural to think of $\sigma_+$ and $\sigma_-$ as $R$-module isomorphisms only. Relations between $\sigma_{\pm}$, multiplication and comultiplication can be rewritten as the equations for the removal of three defect circles around the three holes of a thrice-punctured 2-sphere, see equation (\ref{eq:remove-3-circles}). 
 Likewise, the relations on $\sigma_{\pm}$ and the unit and the counit maps, see equations (\ref{eq_sigma_m}) and  (\ref{eq_sigma_d}) left, 
 are the removal relations for a defect circle that bounds a disk, see equation (\ref{eq:remove-1-circle}). 

\begin{gather}\label{eq:remove-1-circle}
  \cfDecSquare{0.6}{\circlein{0.7}} \ = - \ 
  \  \cfDecSquare{0.6}{\circleout{0.7}} 
  \  = -\ \ \cfDecSquare{0.6}{} \\[4pt]
  \label{eq:remove-2-circles}
  \NB{\tikz[scale =0.7]{\input{\imagesfolder/fe_sigmaplusminus}}} =
  \NB{\tikz[scale =0.7 ]{\input{\imagesfolder/fe_sigmaminusplus}}} = - \ \  \cfCircleId{0.7}{} \\[4pt]
  \label{eq:remove-3-circles}
  \NB{\tikz[scale=0.4]{\input{\imagesfolder/fe_3holed-sphere}}} = - \ \ 
  \NB{\tikz[scale=0.4]{\input{\imagesfolder/fe_3holed-sphere-2}}} =  - \ \ 
  \NB{\tikz[scale=0.4]{\input{\imagesfolder/fe_3holed-sphere-clean}}} 
\end{gather}

\vspace{0.1in} 

Together with the involutory  property, this allows to remove a collection of $n$ distinct defect circles bounding an $n$-punctured 2-sphere facet without dots, perhaps with a sign, depending on co-orientations. When all circles are co-oriented into a dotless facet, the sign is $(-1)$. When all circles are co-oriented out of a dotless facet, the sign is $(-1)^{n+1}$. In general, if $k$ circles are co-oriented out of the facet, the sign is $(-1)^{k+1}$. In the special case $n=0$ there is still consistency since the facet is a dotless 2-sphere, evaluating to zero. If a region bounding an $n$-punctured sphere contains dots, they can be reduced to a linear combination of the same diagrams with at most one dot in the region. A dot can be flipped over a defect line to a dual dot on the other side, see equation (\ref{eq_star_dot_flip}) above. 

If a defect circle bounds a facet on both sides, the cobordism evaluates to zero, since the configuration of defect lines is not even (see Section~\ref{sec:ev} for details). 

To prove that a cobordism equipped with $\sigma$-defect lines gives a well-defined map between tensor powers of $A$ (which is implicitly assumed in this discussion), it's convenient to use evaluation of cobordisms with  $\sigma$-defect lines. We do this a little later, in Section~\ref{sec:ev}. 

Maps $\sigma_{\pm}:A \lra A$ also decompose via the dual dot as 
\begin{equation}\label{eq_sigma_cut}
\sigma_{\pm}(a)\ = \pm \left( \epsilon(X_2 \, a) - \ X_2 \, \epsilon(a) \right) .  
\end{equation}
Furthermore, we can interpret the equation 
\begin{equation}
    (\sigma_{\pm}\otimes 1)\Delta(1) = 
    (1\otimes \sigma_{\mp})\Delta(1)
\end{equation}
via the movement of a defect line from the left to the right side of a
tube. The tube has both boundary circles at the top, and the
co-orientation goes from pointing up to pointing down as the defect
line is isotoped, as illustrated below. 

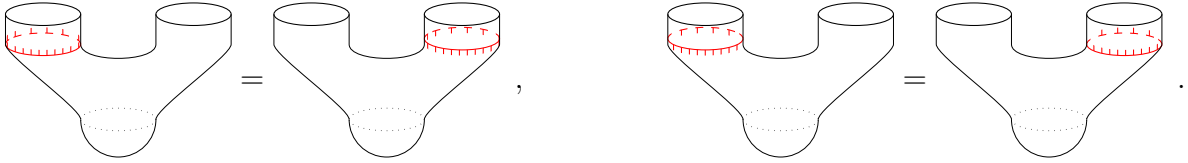
\[
  \NB{\tikz[scale=0.5]{\input{\imagesfolder/fe_move-sigma-delta1}}} =
  \NB{\tikz[scale=0.5]{\input{\imagesfolder/fe_move-sigma-delta2}}},\qquad \qquad 
  \NB{\tikz[yscale= 0.5, xscale=-0.5]{\input{\imagesfolder/fe_move-sigma-delta2}}} =
  \NB{\tikz[yscale= 0.5, xscale=-0.5]{\input{\imagesfolder/fe_move-sigma-delta1}}}.
\]

We have 
\begin{equation}\label{eq_circ_sigma} 
 m \circ (\sigma_{\pm}\otimes 1) \circ \Delta = 0 = m \circ (1\otimes \sigma_{\pm}) \circ \Delta .
\end{equation}
Pictorially, this reads:
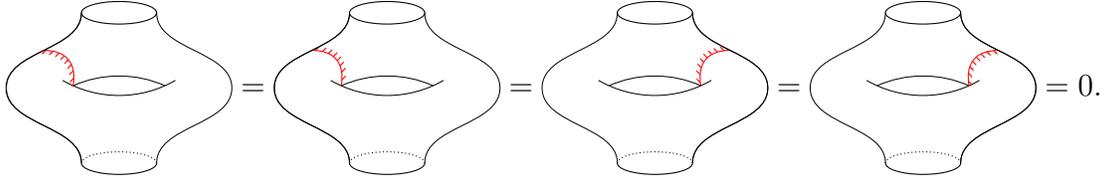
\[
  \NB{\tikz[xscale=0.5, yscale =0.5]{\input{\imagesfolder/fe_mdeltasigma2}}}=
  \NB{\tikz[xscale=0.5, yscale =0.5]{\input{\imagesfolder/fe_mdeltasigma1}}}=
  \NB{\tikz[xscale=0.5, yscale =0.5]{\input{\imagesfolder/fe_mdeltasigma4}}}=
  \NB{\tikz[xscale=0.5, yscale =0.5]{\input{\imagesfolder/fe_mdeltasigma3}}}=0.
\]
In these cobordisms, the defect circle bounds the same region on both sides. We'll see in the next section that any cobordism with this property defines the zero map. More generally, if a cobordism admits a circle that intersects defect lines odd number of times, the map associated to such cobordism is zero, see Section~\ref{sec:ev}.

\section{Two evaluations of seamed surfaces}
 \label{sec:evaluation}
 
In this section we continue to use variables $\alpha_1,\alpha_2$ generating the ring $R_{\alpha}$ as above, with symmetric functions in $\alpha_1,\alpha_2$ giving us the subring $R=\Z[E_1,E_2]$. To connect the evaluation formulas below with those in~\cite{RW1,KR1} one should replace our $\alpha_1,\alpha_2$ by $X_1,X_2$ as in~\cite{RW1,KR1}  and replace our $X_1,X_2$, heavily used in Section~\ref{sec:algebrasl2}, by a different notation, for example $Y_1,Y_2$.

\subsection{An evaluation over \texorpdfstring{$R$}{R} with defect lines}
\label{sec:ev}
\newcommand{\seam}{\ensuremath{\Theta}}

\begin{dfn}
  A \emph{closed seamed surface} is a closed compact surface $F$
  equipped with a PL embedding in $\RR^3$ and with finitely many
  disjoint simple closed curves on $F$ and dots. Each of these curves
  comes with a co-orientation, that is, a preferred side. Such a curve
  is also called a \emph{seam} or a \emph{defect}. The set of seams of
  $F$ is denoted $\seam(F)$.
\end{dfn}
The set $\seam(F)$ of seams may be empty. Denote by $f(F)$ the set of
connected components of $F$ with all seams removed. Its elements are called
\emph{facets}, or, more precisely, \emph{open facets}. We consider both open
facets and their closures in $F$. The closure of a facet may contain
one or more seams. For simplicity, by a facet we usually mean a closed
facet.

Dots on $F$, if any, are placed away from the seams. A dot may float freely in its facet but cannot cross a seam.

Preferred side at a seam may be indicated by a short interval or an arrow pointing
from a point on the seam into the corresponding side. Alternatively, we can draw a whole comb of spaced out intervals from the seam and into that side.

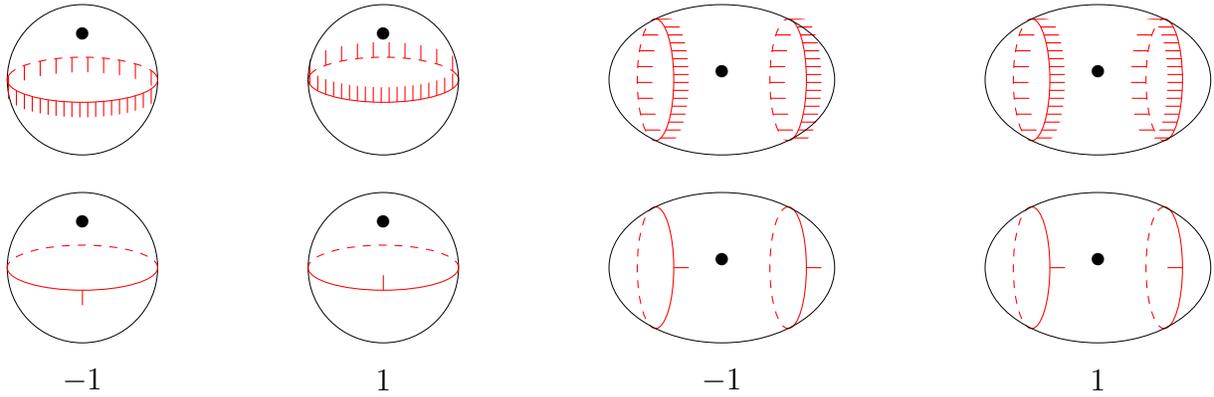
\begin{figure}[ht]
  \centering \NB{\tikz[]{\input{\imagesfolder/fe_theta}}}
  \caption{Four examples of seamed and dotted spheres represented with both  comb and
  segment notation. Their $\kup{\cdot}$-evaluation is given at the bottom. See example~\ref{example_comp} for  details in evaluating the fourth diagram.}
  \label{fig:exa-seamed-surface}
\end{figure}

Using the identities for neck-cutting (\ref{eq_neck_cut}), dot reduction (\ref{eq_dot_red}), removal of contractible seam (\ref{eq:remove-1-circle}), and evaluation of a seamless sphere with at most one dot (\ref{eq:spheres}), one
can define a \emph{combinatorial evaluation} of closed seamed
surfaces. That is a map from the set of closed seamed surfaces to the ring $R$. One would need to check that this is well-defined, which follows
from a straightforward computation. 

A feature of such reductions is that they may not be implementable while keeping all surfaces in the intermediate steps embedded in $\RR^3$. One example is to take a torus embedded in $\RR^3$ as the boundary of a knotted solid torus and add several seamed circles in parallel along the longitude of the torus. The first step in the  combinatorial evaluation is to do surgeries along annuli separating seamed circles on the torus. This step cannot be done in a natural local way inside $\RR^3$ and requires forgetting the embedding but, at least in some similar cases, making sure that all intermediate surfaces are orientable. This is not a serious obstacle, though.  

Instead, to show that the map is well-defined, we now exhibit a closed
formula for this combinatorial evaluation.

Two facets $f_1,f_2$ are called \emph{adjacent} if they share at least
one seam. A facet is called \emph{self-adjacent} if it comes to some
seam from both sides.

We call a seamed surface \emph{even} iff the union of its seams represents zero
in $\mathrm{H}_1(F,\Z/2)$. Otherwise a surface is called \emph{odd} (and will evaluate to zero). A seamed surface with a self-adjacent facet is necessarily odd. We
can refine this terminology and call each connected component of $F$
\emph{even} or \emph{odd} depending on this property of its
seams. Note that a component $F_k$ is \emph{even} if each circle in
generic position to the seamed circles in $F_k$ has even number of
intersections with the union of seamed circles.

\begin{dfn}
  A \emph{checkerboard coloring} of a closed seamed surface $F$ is a
  map $c: f(F) \to \{1,2\}$ from the set of facets of $F$ to the two-element set $\{1,2\}$  such that, along every seam, its two
  sides have different colors. The set of checkerboard colorings
  of $F$ is denoted $\adm(F)$. 
\end{dfn}

A coloring $c$ of the facets
induces a coloring of the seams, by assigning to a seam the color of the facet into which its co-orientation points. For a seamed circle $\gamma$ denote its induced coloring by  $c(\gamma)\in\{1,2\}$. 

Note that $F$ with a self-adjacent facet has no checkerboard coloring,
$\adm(F)=\emptyset$. More generally, $F$ admits a checkerboard
coloring iff it is an even seamed surface.

Denote by $|F|$ the number of connected components of $F$. An even
surface admits $2^{|F|}$ checkerboard colorings.
Denote by $\theta(F)=|\seam(F)|$ the number of seams of $F$. 

For $i= 1,2$  denote by
$F_i(c)$ the union of facets of $F$ colored by $i$. Boundary circles of facets are included in both $F_1(c),F_2(c)$. 
Denote 
by $\seam_i(c)$ the set 
of seams of $F$ colored by $i$ and by $\theta_i(c)$ its cardinality, 
$\theta_i(c)=|\Theta_i(c)|$. We
have $\seam(c)=\seam_1(c)\sqcup \seam_2(c)$ and 
\[ \theta(F) = \theta_1(c)+ \theta_2(c)
\] 
is the number of singular circles of $F$, which does not depend on the coloring. 

Denote by $d_i(c)$ the number of dots on facets of 
color $i$ for a coloring $c$ and by $d(F)$ the number of dots on $F$. Necessarily, $d(F)=d_1(c)+d_2(c)$ for any coloring $c$. Define the degree 
\begin{equation}
    \deg(F) = -\chi(F)+ 2 d(F).
\end{equation}
Since $F_i(c)\subset F\subset \RR^3$, and $F$ is a closed surface in $\RR^3$, both $F$ and $F_i(c)$ are orientable.   Hence, the Euler
characteristic of $F_i(c)$ has the same parity as the number of connected
component of its boundary. Define 
\begin{equation} \label{eq_s_bar}
s(F,c) =\theta_1(c) +\frac{\chi(F_1(c))
+ \theta(F)}2 =\theta_1(c)+\frac{\chi(\overline{F}_1(c))}{2}\in \ZZ, 
\end{equation} 
where  $\chi(\overline{F}_1(c))=\chi(F_1(c))
+ \theta(F)$ is the Euler characteristic of the closed surface 
$\overline{F}_1(c)$ given by attaching  2-disks to $F_1(c)$ along all boundary circles. 

Recall the rings $R=\Z[E_1,E_2]$, $R_{\alpha}=\Z[\alpha_1,\alpha_2]$, and $R_{\alpha\mcD}=R_{\alpha}[(\alpha_1-\alpha_2)^{-1}]$, the localization of $R_{\alpha}$ at $\alpha_1-\alpha_2$. There are ring inclusions 
\[ R \subset R_{\alpha}\subset R_{\alpha\mcD}.
\] 
Finally, the \emph{$c$-evaluation}
$\kup{F,c}$ (or the evaluation of $F$ at $c$) and the \emph{evaluation} $\kup{F}$ are given by the
following formulas:

\begin{align}
  \label{eq_ev1}
  \kup{F, c}&= 
  (-1)^{s(F,c)}\frac{\alpha_1^{d_1(c)}\alpha_2^{d_2(c)}}{(\alpha_2 - \alpha_1)^{\chi(F)/2}}\in R_{\alpha\mcD},  \\ \label{eq_ev2}
  \kup{ F }&= \sum_{c\in \adm(F)} \kup{F,c}. 
\end{align}
The denominator term does not depend on a choice of coloring. 

The idea to extend the Robert--Wagner foam evaluation~\cite{RW1} to the case when $F_i(c)$ are not closed surfaces by capping their boundary circles with disks and taking the Euler characteristic of the resulting surfaces $\overline{F}_i(c)$ was proposed by Yakov Kononov \cite{KononovPC}, who also pointed out that  such closure constructions are used implicitly in the physics TQFT literature.  The  alternative is to use the Euler characteristic of $F_1(c)$, which may be an odd integer. This requires adding $\sqrt{-1}$ to the ground ring, see Section~\ref{subsec:ev}. A deformation of evaluation (\ref{eq_ev1}) appears in~\cite{InprepKhovKitchKon}, with much larger 
 state spaces associated to collections of marked circles in the plane.

\emph{Remark:} Equation (\ref{eq_ev1}) has $\alpha_2-\alpha_1$ in the denominator, compared to $X_i-X_j$ for $i<j$ in~\cite{RW1}. This is done to make the 2-sphere with one dot evaluate to $1$ rather than $-1$. The same answer can be achieved by changing to $\alpha_1-\alpha_2$ in the denominator of (\ref{eq_ev1}) and to $\overline{F}_2(c)$ in place of $\overline{F}_1(c)$ in formula (\ref{eq_s_bar}), so that 
\begin{equation}
  \label{eq_ev3}
  \kup{F, c} = 
  (-1)^{\theta_1(c)+\chi(\overline{F}_2(c))/2}\frac{\alpha_1^{d_1(c)}\alpha_2^{d_2(c)}}{(\alpha_1 - \alpha_2)^{\chi(F)/2}}. 
\end{equation}

If $F$ and $G$ are closed seamed surfaces, it follows directly from the
definitions that 
\begin{equation} \label{eq:product}
     \kup{F\sqcup G} = \kup{F} \cdot \kup{G}.
\end{equation}

\begin{lem}[{Compare with \cite[Proposition 2.18]{RW1} and \cite[Theorem 2.17]{KR1}}] \label{lem:ev-poly}
  $\kup{F}$ is an homogeneous symmetric
  polynomial in $\alpha_1$ and $\alpha_2$ of degree $\deg(F)$, for any seamed surface $F$.
\end{lem}
\begin{proof} Let us show that $\kup{F}$ is a polynomial. Because of (\ref{eq:product}),
  we can assume that $F$ is connected. Since $F$ is embedded in
  $\RR^3$, it is orientable. In particular, $\chi(F) \leq 0$ unless
  $F$ is a sphere, and this is the only case for which the statement is
  non-trivial. Since $H_1(\SS^2, \ZZ_2)$ is trivial, $F$ is even
  regardless of its seams and admits exactly two checkerboard
  colorings. Denote them $c$ and $c'$.
  One has
  $d_i(c') = d_{3-i}(c)$, $\theta_i(c') = \theta_{3-i}(c)$ and
  $F_{i}(c') = F_{3-i}(c)$  for $i=1,2$.
  Hence,
  \begin{align*}
    \kup{F} &=\frac{(-1)^{s(F,c)}}{\alpha_2    -\alpha_1}
    \left(
    \alpha_1^{d_1(c)}\alpha_2^{d_2(c)} + (-1)^{s(F,c') + s(F,c)}
      \alpha_1^{d_1(c')}\alpha_2^{d_2(c')} 
      \right) \\
     &=\frac{(-1)^{s(F,c)}}{\alpha_2-\alpha_1}
     \left(
     \alpha_1^{d_1(c)} \alpha_2^{d_2(c)} + (-1)^{\theta_1(c) + \theta_2(c) +
      \frac{\chi(F_1(c) + \chi(F_2(c))}{2} +\theta(F) }
      \alpha_1^{d_2(c)}\alpha_2^{d_1(c)} 
      \right) \\
     &=\frac{(-1)^{s(F,c)}}{\alpha_2-\alpha_1}
     \left(
     \alpha_1^{d_1(c)}\alpha_2^{d_2(c)} + (-1)^{\frac{\chi(F)}{2}}
      \alpha_1^{d_2(c)}\alpha_2^{d_1(c)} \right) \\
    &= {(-1)^{s(F,c)}}\frac{\alpha_1^{d_1(c)}
      \alpha_2^{d_2(c)} -
      \alpha_1^{d_2(c)}\alpha_2^{d_1(c)}}{{\alpha_2 -\alpha_1}},
  \end{align*} 
  and $\kup{F}$ is a symmetric polynomial in $\alpha_1$ and $\alpha_2$.
  The statement about the degree follows directly from the definition of
  degree of a seamed surfaces and of the $c$-evaluation.

That $\kup{F}$ is symmetric follows directly from
the identity $\tau(\kup{F,c}) = \kup{F,\tau(c)}$ for the transposition $\tau=(1,2)$ acting on both the set of colorings and on the ring of rational functions in $\alpha_1,\alpha_2$. 
\end{proof}
Thus, $\kup{F}\in R$ for any seamed surface $F$. 

\emph{Remark:} If $F$ has no checkerboard colorings ($F$ has an odd
component), then $\langle F \rangle = 0$.

\begin{exa} \label{example_comp}  In the following computation, colorings are depicted
  directly on seamed surfaces: facets with color 1 are hashed and
  facets with color $2$ are plain white.
  \begin{align*}
    \kup{\NB{\tikz[scale =0.8]{\input{\imagesfolder/fe_eval-exa-1}}}} &=
    \kup{\NB{\tikz[scale =0.8]{\input{\imagesfolder/fe_eval-exa-2}}}} +
    \kup{\NB{\tikz[scale =0.8]{\input{\imagesfolder/fe_eval-exa-3}}}} \\ &=
    (-1)^{2 + \frac{0+2}{2}} \frac{\alpha_1}{(\alpha_2 -
    \alpha_1)^{\frac{2}{2}}} + (-1)^{0 + \frac{2+2}{2}}
\frac{\alpha_2}{(\alpha_2 -
    \alpha_1)^{\frac{2}{2}}} 
    \\ &= \frac{-\alpha_1}{\alpha_2 -
    \alpha_1}+\frac{\alpha_2}{\alpha_2 -
    \alpha_1} =1,
  \end{align*}
  using formulas (\ref{eq_s_bar})-(\ref{eq_ev2}). 
  In more details, $\chi(F)=\chi(\SS^2)=2$, so the denominator is $\alpha_2-\alpha_1$ in both terms. For the first coloring $c_1$ (the annulus facet has color $1$), 
  \[ \theta_1(c_1)= 2, \  \chi(\overline{F}_1(c_1))=\chi(\SS^2) = 2, \  s(F,c_1)= 2+1=3, \ d_1(c_1)=1, \ d_2(c_1)=0. 
  \] 
  For the second coloring $c_2$ (the annulus facet has color $2$), 
  \[ \theta_1(c_2)= 0, \  \chi(\overline{F}_1(c_2))=\chi(\SS^2\sqcup \SS^2) = 4, \  s(F,c_2)= 0+2=2, \ d_1(c_2)=0, \ d_2(c_2)=1, 
  \] 
  so the contributions of $\kup{F,c_1}$ and $\kup{F,c_2}$ to the sum are as above, and $\kup{F}=1$.  
\end{exa}

\begin{lem} \label{lem:ev-loc-rel}
  Evaluation of seamed surfaces satisfies the following local relations:
  \begingroup
\allowdisplaybreaks
  \begin{gather}\label{eq:neck-cutting-kup}
    \kup{\cfCircleId{0.5}{}} =
    \kup{\cfCircleSplit{0.5}{\bullet}{}} +
    \kup{\cfCircleSplit{0.5}{}{\bullet}} -E_1
    \kup{\cfCircleSplit{0.5}{}{}}, \\[4pt]
\label{eq:neck-cutting-kup-2}    \kup{\NB{\tikz[scale=0.5]{\input{\imagesfolder/fe_sigmaplus}}}}= 
     \kup{\cfCircleSplit{0.5}{}{\bullet}} -
 \kup{\cfCircleSplit{0.5}{\bullet}{}}, 
    \\[4pt]
    \label{eq:dot-migration-kup-1}
    \kup{\cfSeamSquare{0.6}{\bullet}{}} +
    \kup{\cfSeamSquare{0.6}{}{\bullet}} =
    E_1 \kup{\cfSeamSquare{0.6}{}{}}, \\[4pt]
    \label{eq:dot-migration-kup-2}
    \kup{\cfSeamSquare{0.6}{\bullet}{\bullet}} =
    E_2 \kup{\cfSeamSquare{0.6}{}{}},\\[4pt]
    \label{eq:dot_reduction}
    \kup{\cfDecSquare{0.6}{ \bullet\ \bullet}} =
    E_1\kup{\cfDecSquare{0.6}{\bullet}} -
    E_2\kup{\cfDecSquare{0.6}{}}\\[4pt]
    \label{eq:circle-removal-1}
    \kup{\cfDecSquare{0.6}{\circlein{0.7}}} = -
  \kup{\cfDecSquare{0.6}{\circleout{0.7}}} 
  = - \kup{\cfDecSquare{0.6}{}},
  \\[4pt]
   \label{eq:kup-ev-sphere}
   \kup{\cfSphere{0.5}{}} = 0, \qquad
      \kup{\cfSphere{0.5}{\bullet}} = 1.
  \end{gather}
  \endgroup
\end{lem}
\begin{proof}
  The only identities which is not straightforward are
  (\ref{eq:neck-cutting-kup}) and    (\ref{eq:neck-cutting-kup-2}). The proof of these identities are similar
  to that of \cite[Proposition 4.7]{KhovKitch} and \cite[Proposition 2.22]{KR1}. We only prove (\ref{eq:neck-cutting-kup}),  the other one is analogous.

  Let us denote by $F$ the seam surface on the left-hand side of the
  identity and by $G_t$, $G_b$ and $G$ the seamed surfaces on the right-hand
  side, where $t$ and $b$ stand for dot on  \textbf{t}op and dot on
  \textbf{b}ottom. As surfaces, $G_t$ and $G_b$ and $G_t$ are identical,
  they only differ by their dot distributions.  There is a canonical
  one-to-one correspondence between $\adm(G)$, $\adm(G_b)$ and $\adm(G_t)$. Let
  $c$ be an element of $\adm(G)$. There are $4$ possible
  local models for $c$. They are given in Figure~\ref{fig:col-circle}.
  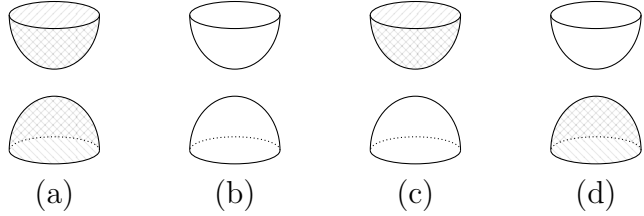
\begin{figure}[ht]
    \centering
    \NB{\tikz[scale=0.6]{\input{\imagesfolder/fe_split-circle-col}}}
    \caption{The four possible local models for an coloring $c$ of
    $G$. Hash means facet has color $1$ while solid white means
    facet has color $2$.}
    \label{fig:col-circle}
  \end{figure}

  In cases (c) (resp.{} (d)), one has
  \begin{align*}
    &\kup{G_t, c}= \alpha_1\kup{G,c}
    \quad\text{and}\quad 
    \kup{G_b, c}= \alpha_2\kup{G,c}
    \\ (\text{resp.{}}
        &\kup{G_t, c}= \alpha_1\kup{G,c}
    \quad\text{and}\quad 
    \kup{G_b, c}= \alpha_2\kup{G,c})
  \end{align*}
  therefore, in both of these cases, $\kup{G_t, c} +  \kup{G_b, c} 
  -E_1 \kup{G,c} = 0$.
  The admissible colorings of $G$ of types (a) and (b) are in a one-to-one
  correspondence with admissible colorings of $F$. Let $c$ be a
  coloring of type (a) of $G$, with the corresponding coloring of $F$ 
  still denoted by $c$. On the one hand, one has:
  \begin{gather*}
    \theta_1(G,c) = \theta_1(F,c), \quad \chi(G) = \chi(F)+2, \quad \chi(G_1(c)) =
    \chi(F_1(c)) +2,
     \quad
    \theta(G) = \theta(F),
  \end{gather*}
  On the other hand,  $\kup{G_t, c} = \kup{G_b,c} = \alpha_1
  \kup{G,c}$. Hence
  \[
\kup{G_t,c} + \kup{G_b,c} - E_1\kup{G,c} = (\alpha_1
-\alpha_2)\kup{G,c}= - \frac{\alpha_1 - \alpha_2}{\alpha_2 - \alpha_1}
\kup{F,c} = \kup{F,c}.
  \]
   If the coloring $c$ is of type (d), a similar computation gives as well 
  $\kup{F,c} = \kup{G_t,c} + \kup{G_b,c} - E_1 \kup{G,c}$.
  Summing over all admissible colorings, one obtains $\kup{F} =
  \kup{G_t} + \kup{G_t} - E_1\kup{G}$.
\end{proof}

\vspace{0.1in}

Formulas (\ref{eq:neck-cutting-kup}), (\ref{eq:dot_reduction}), and (\ref{eq:kup-ev-sphere}) coincide with the corresponding formulas for the 2-dimensional TQFT assigned to the Frobenius pair $(R,A)$, see Section~\ref{sec_road_map}. 

As in Section~\ref{sec_road_map}, one can introduce the hollow and star  dots, which will satisfy the same relations as earlier.  

We'll see in Section~\ref{sec:categories} that the state spaces associated to collections of planar circles in this theory coincide with those for the $(R,A)$ TQFT. Seams  correspond to $\sigma$-defect lines in Section~\ref{sec_inv_d}. Allowing them to end on the boundary enriches $(R,A)$-TQFT, see examples at the end of Section~\ref{sec:ssfunc}.

\vspace{0.1in}

\subsection{An evaluation over \texorpdfstring{$R_{\omega}$}{R omega} with defect lines}
\label{subsec:ev}

Instead of completing the surface $F_1(c)$ to a closed surface $\overline{F}_1(c)$, which has Euler characteristic divisible by two, one can keep $\chi(F_1(c))$ in the evaluation formulas (\ref{eq_s_bar}), 
(\ref{eq_ev1}) at the cost of adding $\sqrt{-1}$ to the ground ring $R$ to make sense of $(-1)^{\chi(F_1(c))/2}.$ 

Denote by 
$\ZZ_{\omega}=\ZZ[\omega]/(\omega^2+1)$ the ring of Gaussian integers and let  
\begin{equation}
R_{\omega}=\ZZ_{\omega}\otimes_{\ZZ} R=\ZZ_{\omega}[E_1,E_2]
\end{equation}
be the  ring  obtained from $R$ by formally adding a square root $\omega$ of $-1$.

Furthermore, denote $A_{\omega} = R_{\omega}\otimes_R A$.
The  pair $(R_{\omega},A_{\omega})$ is a Frobenius extension of rank two, with $\{1,X\}$ a basis of $A_{\omega}$ as a free graded $R_{\omega}$-module.

We give an
alternative evaluation of seamed surfaces in this new algebraic
context, using the ring $R_{\omega}$, together with the rings 
\[ R_{\omega\alpha} = \ZZ_{\omega}\otimes_{\ZZ} R_{\alpha} = \Z[\alpha_1,\alpha_2,\omega]/(\omega^2+1)
\] 
and 
\[ R_{\omega\alpha\mcD} = R_{\alpha\mcD}[\omega]/(\omega^2+1) = \Z_{\omega}[\alpha_1,\alpha_2,(\alpha_1-\alpha_2)^{-1}].
\] 
There are ring inclusions
\[ R_{\omega}\subset R_{\omega\alpha}
\subset R_{\omega\alpha\mcD}.
\] 
Define the evaluation of $F$ at a coloring $c$ and the overall evaluation of $F$ by 
\begin{align}
  \label{eq:evo-col}
  \kupo{F,c} &:= \omega^{2\theta_1(c) +
  \chi(F_1(c))}\frac{\alpha_1^{n_1(c)}\alpha_2^{n_2(c)}}{(\alpha_2 -
  \alpha_1)^{\chi(F)/2}} = \omega^{-\theta(F)}\kup{F,c}, 
  \\ \label{eq:evo} \kupo{ F }&:=
  \sum_{c\in \adm(F)} \kupo{ F,c} = \omega^{-\theta(F)}\kup{F}.
\end{align}
Note that $ \kupo{F,c}\in R_{\omega\alpha\mcD}$ 
while $\kupo{F}\in R_{\omega}$ for any  seamed surface $F$. 
Thus, $\kup{F}$ is an homogeneous symmetric polynomial in $\alpha_1$ and
$\alpha_2$ of degree $\deg(F)$ with coefficients in the ring of Gaussian integers $\ZZ_{\omega}=\ZZ[\omega]/(\omega^2+1)$. 

\begin{lem} \label{lem:evo-loc-rel}
  Evaluation of seamed surfaces satisfies the following local relations:
  \begingroup
\allowdisplaybreaks
      \begin{gather}\label{eq:neck-cutting-kup-omega}
    \kup{\cfCircleId{0.5}{}}_{\omega} =
    \kup{\cfCircleSplit{0.5}{\bullet}{}}_{\omega} +
    \kup{\cfCircleSplit{0.5}{}{\bullet}}_{\omega} -E_1
    \kup{\cfCircleSplit{0.5}{}{}}_{\omega}, \\[4pt]
    \label{eq:neck-cutting-kup-omega-2}    \kupo{\NB{\tikz[scale=0.5]{\input{\imagesfolder/fe_sigmaplus}}}}= 
 \omega\kupo{\cfCircleSplit{0.5}{}{\bullet}} -
     \omega\kupo{\cfCircleSplit{0.5}{\bullet}{}}, 
    \\[4pt]
    \label{eq:dot-migration-kup-omega-1}
    \kup{\cfSeamSquare{0.6}{\bullet}{}}_{\omega} +
    \kup{\cfSeamSquare{0.6}{}{\bullet}}_{\omega} =
    E_1 \kup{\cfSeamSquare{0.6}{}{}}_{\omega}, \\[4pt]
    \label{eq:dot-migration-kup-omega-2}
    \kup{\cfSeamSquare{0.6}{\bullet}{\bullet}}_{\omega} =
    E_2 \kup{\cfSeamSquare{0.6}{}{}}_{\omega},\\[4pt]
    \kup{\cfDecSquare{0.6}{ \bullet\ \bullet}}_{\omega} =
    E_1\kup{\cfDecSquare{0.6}{\bullet}}_{\omega} -
    E_2\kup{\cfDecSquare{0.6}{}}_{\omega}\\[4pt]
    \label{eq:circle-removal-omega-1}
    \kup{\cfDecSquare{0.6}{\circlein{0.7}}}_{\omega} = -
  \kup{\cfDecSquare{0.6}{\circleout{0.7}}}_{\omega} 
  = - \omega\kup{\cfDecSquare{0.6}{}}_{\omega}, \\[4pt]
  \label{eq:seam-saddle}
   \kup{\cfDoubleSeamA{0.6}}_\omega=\omega
   \kup{\cfDoubleSeamB{0.6}}_\omega,\\[4pt]
   \label{eq:kupo-ev-sphere}
   \kupo{\cfSphere{0.5}{}} = 0, \qquad
      \kupo{\cfSphere{0.5}{\bullet}} = 1.
  \end{gather}
  \endgroup
\end{lem}

From relation (\ref{eq:neck-cutting-kup-omega-2}) we see that the  seam line is now  $\sigma_{\pm}$ scaled by $\omega$ rather than  just $\sigma_{\pm}$ as in equation (\ref{eq:neck-cutting-kup-2}). Multiplicative factor $\omega^{-\theta(F)}$ in the equation (\ref{eq:evo}) tells us how other formulas, including equations (\ref{eq:remove-2-circles}) and (\ref{eq:remove-3-circles}), will modify in this evaluation.

\vspace{0.1in}

\subsection{Universal construction} \label{sec:uc}

The \emph{universal construction} constructs
functors from a cobordism category to an algebraic category.
It was introduced by
Blanchet--Habegger--Masbaum--Vogel~\cite{BHMV} and used to build foam state spaces in~\cite{SL3,RW1,KR1}. For this construction, one needs
an evaluation of closed objects, such  as cobordisms, that is, a map from the isomorphism classes of 
closed $n$-manifolds to some commutative ring. In the cases we consider the evaluation is $\kup{\cdot}$ or $\kup{\cdot}_\omega$. 

In favorable situations, functors obtained by the universal construction are
TQFTs. However, this is not always the case. In particular these
functors can fail to be monoidal.

Suppose $\Cob$ is a category of cobordisms, $S$ is a commutative unital ring and
\[\tau\co \mathrm{End}_{\Cob}(\emptyset) \to S
\] 
 a monoid homomorphism. Here $\End_{\Cob}(\emptyset)$ is the monoid of isomorphism classes of cobordisms with the empty boundary as elements and the  
disjoint union as the composition.
The map $\tau$ should take composition of cobordisms to  the product of corresponding elements in $S$. In particular it
maps the empty $n$-cobordism $\emptyset_n = \id_{\emptyset_{n-1}}$ to $1_S$. In these notations we distinguish between the empty $n$-cobordism $\emptyset_n$ and the empty $(n-1)$-cobordism $\emptyset_{n-1}$. The former is the identity endomorphism of the latter. 

Let $M$ be an object of
$\Cob$. Define $\widetilde{\mcF_\tau}(M)$ to be the free $S$-module
generated by $\Hom_\Cob(\emptyset, M)$. For any $W$ in $\Hom(M,
\emptyset)$, define the $S$-linear map $\varphi_W\co
\widetilde{\mcF_\tau}(M) \to S$ on basis elements $V \in
\Hom_\Cob(\emptyset, M)$ by $\varphi_W(V) = \tau(W \circ V)$. Finally, let 
\[
  \mcF_\tau(M) = \widetilde{\mcF}_\tau(M)\left/\bigcap_{W \in
    \Hom_\Cob(\emptyset, M)} \Ker(\varphi_W)\right. 
\]
be the quotient module, $\mcF_\tau(M)\in \Smod$. 
Thus, $\mcF_\tau(M)$ is the quotient of the free $S$-module $\widetilde{\mcF_\tau}(M)$ by the kernel of a suitable bilinear form. 
For a cobordism $W$ representing a generator of $\widetilde{\mcF}_\tau(M)$, denote by $[W]$ its equivalence
class in $\mcF_\tau(M)$. One extends $\mcF$ to a functor by defining
for any $W \in \Hom_\Cob(M_1, M_2)$ and $V \in \widetilde{\mcF}_\tau(M_1)$:
\[
\mcF_\tau([W])([V]) = [W\circ V].
\]

\emph{Bilinear pairing.} 
   For every object $M$ in $\Cob$, evaluation $\tau$ induces an $R$-bilinear pairing $(\cdot, \cdot)_\tau$ on
  $\mcF_\tau(M)$: If $W_1$ and $W_2$ are two elements of
  $\Hom_{\Cob}(\emptyset, M)$, define
  \begin{equation}({[W_1] , [W_2]})_\tau = \tau( \overline{W_1}\circ W_2) = \varphi_{\overline{W_1}}(W_2), \label{eq:bil-pairing}\end{equation}
  where $\overline{W_1} \in \Hom_\Cob(M, \emptyset)$ is the mirror image of $W_1$.
  This pairing is non-degenerate at least on the right. We did not require skew-invariance of $\tau$ under the flip (often related to orientation-reversal of closed cobordisms), where ring $S$ would carry a bar involution,  with 
  $\overline{\tau(W)} = \tau(\overline{W})$ for closed cobordisms $W$. This condition would make the bilinear form skew-invariant as well. In our examples, the involution on the ring is the identity and $\tau(W)=\tau(\overline{W})$ for closed cobordisms $W$, making the bilinear form $(\cdot,\cdot)_{\tau}$ symmetric.

\vspace{0.1in} 

\subsection{Category of seamed surfaces}
\label{sec:categories}

Seamed surfaces can be extended to a category whose objects are finite disjoint
unions of \emph{marked circles}. Marks (or defects) on circles are endpoints of seams ending on the boundary of a surface. 

A \emph{marked circle} is a circle
equipped with a PL-embedding in $\RR^2$. It carries a finite number (possibly none) of
marked points, also called \emph{seam points}. Each marked point carries a co-orientation, that is a preferred direction in a circle at this point. 
Equivalently, we say  that a co-orientation is a preferred side of a circle near a marked point. Marked points and their co-orientations are depicted by red solid
arrows on circles. If one chooses an orientation $o$ of a seam circle,
co-orientation at each marked point either agrees or disagrees with
$o$. If the orientation agrees, respectively disagrees, with co-orientation, we label the point as $+$ point (plus point), respectively, $-$ point (minus point), relative to this orientation. Reversing the orientation of the circle flips plus and minus points. If there are as many plus as minus points, the circle is called \emph{balanced}. This notion does not depend on the choice of orientation. If a circle has
no marked point it is \emph{unmarked}.

An object of the category $\SeSu$ is a finite collection $C$ of disjoint marked circles in the plane. We call such $C$ a \emph{marked  embedded one-manifold} or \emph{meom}, for short.

Let us orient circles in $C$ so that outermost circles are oriented clockwise. When all outermost circles are removed, the outermost circles in the remaining one-manifold must have anticlockwise orientation, and so on.  Iterating this condition, we come to a canonical choice of orientation $o(C)$ for circles in any $C$. 

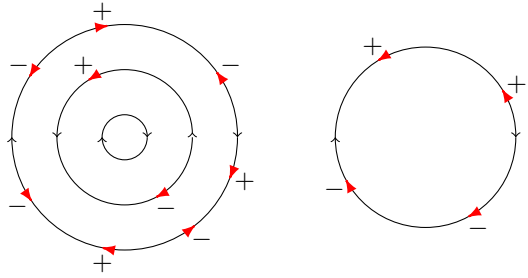
\begin{figure}[ht]
  \centering
  \NB{\tikz[scale=1]{\input{\imagesfolder/fe_marked-circles-exa}}}
  \caption[{Example of a seambedded one-manifold. Orientation are
  depicted by black and thin arrows and marks are indicated by red triangular arrows.}]
  {Example of a meom. Orientations are
  depicted by thin black arrows $(\feOrientation{0.8})$
  and marks are indicated by triangular red arrows
  $(\feMarkedPoint{0.8})$. The outer left circle is not balanced, the
  others are.}
  \label{fig:exa-marked-circles}
\end{figure}

Since marked points in $C$ carry co-orientations, each point is either compatibly oriented relative to $o(C)$ (a $+$ point) or oppositely oriented (a $-$ point). On a given circle in $C$, we can encode the sequence of orientations as a sequence of signs $\undell=(\ell_1,\dots, \ell_k)$, $\ell_i\in\{+,-\}$, up to cyclic order, as we go along the circle following its orientation. 

If an object $C$ is a single circle in the plane, necessarily clockwise oriented, with the sequence of signs $\undell$, we denote it by $\SS^1_{\undell}$.  

We say that $C$ is \emph{balanced} if it has as many plus signs as minus signs in the collection of  sequences for its circles. For instance, if $C$ has three circles, with cyclic sequences $\{(+++),(--+),(--)\}$, it is balanced. Note that $C$ may be balanced without individual circles having this property. Notice also that the assignment of pluses and minuses to marked points depends not only on their co-orientations but also on the  parity of the circle in its nesting in $C$, that is, whether its orientation in the plane is clockwise or anticlockwise. 

More generally, to $C$ we associate its \emph{weight} $w(C)$, the difference between the number of pluses and minuses on its circles. A meom $C$ is balanced if and only if $w(C)=0$. 

A circle is called \emph{odd}, respectively \emph{even}, if it has an odd, respectively even, number of seam points. A meom $C$ is called \emph{even} if each circle in it is even. 
 
If $C_0$ and $C_1$ are two meoms, a
\emph{seamed cobordism from $C_0$ to $C_1$} is a compact surface $F$
  equipped with a proper PL-embedding in $\RR^2\times [0,1]$ and with finitely many
  disjoint simple curves with co-orientations in $F$ (\emph{seams}) and dots such that:
  \begin{itemize}
  \item The embedding of $F$ is transverse to $\RR^2 \times \{0,1\}$.
  \item The seams of $F$ are properly embedded and transverse to the boundary.
  \item The boundary of $F$ is equal to $C_0 \times \{0\}\sqcup C_1\times
    \{1\}$.
  \item Marked points of $C_0$ and $C_1$ coincide with the
    intersection of seams of $F$ and the boundary. Their
    co-orientations agree with the ones induced by seams.  
  \end{itemize}

  \begin{figure}[ht]
    \centering
    \NB{\tikz[]{\input{\imagesfolder/fe_conv_orientations}}}
    \caption{Diagrammatic summary of conventions for orientation and
    co-orientations.  }
    \label{fig:conv-orientations}
  \end{figure}
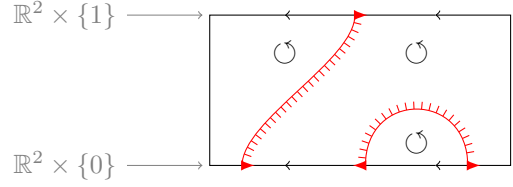

Our orientation conventions for surfaces and their top and bottom boundaries are the following. Orientation of a surface $F$ induces an orientation of its top boundary $\partial_1 F$ by sticking the first vector of an orientation basis out of the surface, see Figure~\ref{fig:conv-orientations}. The second vector then shows the direction for the boundary orientation. For the bottom boundary $\partial_0 F$ the convention is the opposite: when the second vector of an orientation basis points out, the first vector shows induced orientation of the boundary. We also adopt the convention that 
\[ \partial F = \partial_1 F  \sqcup (-\partial_0 F).
\] 

In Figure~\ref{fig:conv-orientations} top and bottom horizontal lines indicate parts of circles in $\RR^2\times\{1\}$ and $\RR^2\times \{0\}$, respectively. Marked points, shown as red triangles at these boundary lines, inherit co-orientations from those of seamed arcs. Top marked point is a minus point, the three marked points at the bottom edge have signs $(-+-)$, reading from left to right. The seam with both endpoints on the bottom connects a $+$ and a $-$ endpoints (different signs). An edge betweeen the top and the bottom boundary connects two $-$ enpoints (same sign on both).

Meoms and seamed cobordisms between them, up to rel boundary isotopies in $\RR^2\times [0,1]$, form a category denoted
$\SeSu$. Composition is given by superposition and rescaling.

In $\SeSu$ a morphism from $C_0$ to $C_1$ exists iff $w(C_0)=w(C_1)$, that is, if they have the same weight. Indeed, such a morphism may have several seamed arcs connecting marked points on the same boundary $\partial_i F$, $i=0,1$, and connecting marked points of $\partial_0 F$ to the points of $\partial_1 F$. Seamed arcs of the first type connect a plus point and a minus point, each contributing zero to $w(C_i)$. Seamed arcs of the second type connect points of $\partial_0 F$ and $\partial_1 F$ of the same sign, contributing zero to $w(C_0)-w(C_1)$. Consequently, $w(C_0)=w(C_1)$. 
\vspace{0.1in} 

\emph{Remark:} It's important to point out that our setup with orientations of circles and surfaces carries a secondary role, and these orientations can almost be ignored. We are orienting collections $C$ of circles in 
$\RR^2$ so that any surface $F$ embedded in $\RR^2\times [0,1]$ can  be compatibly oriented, together with its boundary. These orientations also allow to assign signs to all marked points on the boundary of $F$ 
In fact, we need less data. What's useful is the relative index of two marked points on $\partial F$ carrying co-orientations  to see if they can be boundary  points of a single seam.   Marked points of $\partial F$, with co-orientations, decompose as a disjoint union of  two sets. For any two points $p_0$ and $p_1$ in the opposite sets it's  possible to replace $F$ by a surface $F'$ in $\RR^2\times[0,1]$ with the same boundary and boundary co-orientations as in  $F$, such that $p_0$ and  $p_1$ are connected by a seam in $F'$. Equivalently, $F$ can be completed by a surface $F''$ with $\partial F=\partial F''$ and $F''$ lying 'outside' of $F$ with $F\cup F''$ a closed surface in $\RR^3$, such that $p_0$ and $p_1$ are connected by a seam in $F''$ and, consequently, belong to the same defect circle (seamed circle) in $F\cup F''$. Points $p_0$ and $p_1$ from the same  subset cannot be connected by a seam in $F$ or in any such replacement $F'$ or complement $F''$. 

\vspace{0.1in}

Applying the universal construction to $\kup{\cdot}$ and 
$\kup{\cdot}_\omega$ and the category $\SeSu$ yields functors,
denoted $\kup{\cdot}$ and $\kup{\cdot}_\omega$, respectively. These are functors 
\[ \kup{\cdot} : \SeSu \lra  R\mathrm{-gmod}, \ \ \  \kup{\cdot}_\omega : \SeSu \lra  R_{\omega}\mathrm{-gmod},
\]
from $\SeSu$ to the category of graded $R$-modules, respectively graded $R_{\omega}$-modules, and homogeneous module homomorphisms. 
The
images of objects of $\SeSu$ by these functors, that is, $\kup{C}$ and $\kup{C}_{\omega}$, are called \emph{state
spaces} of $C$. 

Suppose $F\co C_0\to C_1$ is a seam surface. It follows from the definition of evaluation of seam surface that the $R$-module map $\kup{F}$ (resp.{} $R_\omega$-module map $\kupo{F}$) is homogenous and its degree is given by the formula:
\[
\deg \kup{F} = \deg \kupo{F} = -\chi(F) + 2d(F).
\]

\subsection{State spaces for functors \texorpdfstring{$\kup{\cdot}$}{<.>} and \texorpdfstring{$\kupo{\cdot}$}{<.>omega}}

\label{sec:ssfunc}

The state space $\kup{\emptyset_1}\cong R$ of the empty collection of circles is a rank one free $R$-module generated by $[\emptyset_2]$, where $\emptyset_2$ is the empty surface. If a meom $C$ is unbalanced (that is, $w(C)\not=0$), there are no cobordisms from the empty meom $\emptyset_1$ into it and $\kup{C}=0$, $\kup{C}_{\omega}=0$. 

Given meoms $C$ and $C'$, there is a natural injective graded $R$-module homomorphism 
\begin{equation}
    \kup{C}\otimes_R \kup{C'} \lra \kup{C\sqcup C'} 
\end{equation}
intertwining monoidal structures on the category $\SeSu$ and  the category of graded $R$-modules. This functor is not monoidal, though. The above homomorphism is not an isomorphism, in general, since we can take $C$ and $C'$ unbalanced, with zero state spaces, but make $C\sqcup C'$ balanced, with $\kup{C\sqcup C'}\not=0$. A simple example is choosing an unbalanced circle for $C$, with even number of marked points, and taking $C'=C^!$, the mirror image of $C$. The natural tube cobordism from $\emptyset_1$ to $C\sqcup C^!$, composed with its reflection, evaluates to $\pm 2$. This implies nontriviality of the state space $C\sqcup C^!$. 

\begin{lem}\label{lem:kup-circle-unmarked}
  The state space $\kup{\SS^1}$ of a single unmarked circle is isomorphic to  $A$ as a graded $R$-module. 
\end{lem}
\begin{proof}
  We first construct a map from $A$ to
  $\kup{\SS^1}$. Recall that $A$ is a free $R$-module and that
  $1, X$ is a basis of $A$. Define $\Phi:A\lra \kup{\SS^1}$ as the $R$-module map which
  maps $1$ to the class of $\cfCup{0.4}{}$ and $X$ to the class of
  $\cfCup{0.4}{\bullet}$.

 Let us first show that the map $\Phi$ is surjective. By $R$-linearity, it is enough to show that any element of the form
 $[F]$ with $F \in \Hom_{\SeSu}( \emptyset, \SS^1)$ has a
  preimage by $\Phi$. Using the neck-cutting relation (\ref{eq:neck-cutting-kup}) one
  obtains that $[F]$ equals an $R$-linear combination of three
  surfaces. Two of them are disjoint unions of $\cfCup{0.4}{}$ and 
 closed seamed surfaces. One of them is the disjoint union of
 $\cfCup{0.4}{\bullet}$ and a closed seamed surface. Evaluate
 closed seam surfaces, we see that $[F]$ is the an $R$-linear combination of $\cfCup{0.4}{}$ and
 $\cfCup{0.4}{\bullet}$. Hence, $\Phi$ is surjective.

  Let us now show that $\Phi$ is injective. It is enough to show that
  $\left[\cfCup{0.4}{}\right]$ and $\left[\cfCup{0.4}{\bullet}\right]$ are linearly
  independent. We can compute the matrix of the $R$-bilinear pairing
  (\ref{eq:bil-pairing}) on this set:
  \[
    \begin{pmatrix}
     \kup{\cfSphere{0.5}{\color{white} \bullet}} &    \kup{\cfSphere{0.5}{\bullet}} \\[0.5cm]
      \kup{\cfSphereSpecial{0.5}{}{\bullet}} &
      \kup{\cfSphereSpecial{0.5}{\bullet}{\bullet}}
    \end{pmatrix}=
    \begin{pmatrix}
      0 & 1 \\
      1 & E_1
    \end{pmatrix}.  
  \]
This matrix is invertible, proving that $\left[\cfCup{0.4}{}\right]$
and $\left[\cfCup{0.4}{\bullet}\right]$ are linearly independent.
\end{proof}

We remind the reader that in this lemma $A$ is viewed as a module, not as a ring, and its generator $1$ lives in degree $-1$.

One can easily adapt the previous proof to obtain the following
proposition.
\begin{prop}\label{prop:kupp-add-circle}
  Let $C$ be a meom in $\SeSu$ obtained by inserting an unmarked circle $\SS^1$ into one of the regions of a meom $C'$ as an innermost circle. Then $\kup{C} \simeq A
  \otimes_R \kup{C'}$.  
\end{prop}

\begin{cor}
The state space $\kup{\sqcup_{i=1}^k\SS^1}$ of a collection of $k$ unmarked circles is isomorphic to  $A^{\otimes k}$ as a graded $R$-module.   
\end{cor}

The same isomorphism holds for arbitrarily nested collections of unmarked circles.

Each circle corresponds to a tensor factor. Using relations in  Lemma~\ref{lem:ev-loc-rel}, maps associated to generating cobordisms between unmarked circles can be identified with the following maps:
\begin{gather*}
  \kup{\NB{\tikz[scale=0.5]{\input{\imagesfolder/fe_sigmaplus}}}}= \sigma_+ \co A\to A, \qquad
  \kup{\NB{\tikz[scale=0.5]{\input{\imagesfolder/fe_sigmaminus}}}}= \sigma_- \co A\to A, \\[4pt]
  \kup{\cfCap{0.5}{}} = \epsilon\co A\to R, \qquad
  \kup{\cfCup{0.5}{}} = \begin{array}{rcl} R &  \to & A \\ 1_R &\mapsto &1_A
  \end{array}, \qquad
  \kup{\cfCup{0.5}{\bullet}} = \begin{array}{rcl} R &  \to & A \\ 1_R &\mapsto &X
  \end{array}, \\[4pt]
  \kup{\cfDelta{0.5}{}} = \Delta\co A\to A\otimes A,   \qquad
  \kup{\cfMu{0.5}{}} = m \co A\to A\otimes A.   \\
\end{gather*}
The maps are those of the (1+1)-dimensional TQFT $(R,A)$ and the $\sigma$-defect circles, so the theory \hyperlink{frobext2}{1} of Sections~\ref{sec:algebrasl2} and \ref{sec_fine_str} appears out of evaluation $\kup{\cdot}$. 
\vspace{0.1in}

Let us denote by $\SS^1_{+-}$ the balanced seamed circle in the plane with two seam
points. It is unique up to planar isotopy. 

\begin{lem}
  Suppose that $\gamma$ is a separating closed 
  curve on  a seamed surface $F$ that intersects seams transversely in exactly two points
  with opposite co-orientations. Then the following relation holds:
  \begin{gather}
  \label{eq:neckcut-2pt}  
 \kup{ \NB{\tikz[scale=0.7]{\input{\imagesfolder/fe_id-circle-2pt}}}} \ \stackrel{\warning{0.4}}{=}
  \ \kup{\NB{\tikz[scale=0.7]{\input{\imagesfolder/fe_split-circle-2pt-2}}}}\ - \ 
  \kup{\NB{\tikz[scale=0.7]{\input{\imagesfolder/fe_split-circle-2pt-1}}}}.
\end{gather}
Notation \warning{0.6} above the equal sign emphasizes the fact this \emph{local} relation is valid
provided that a \emph{global} condition ($\gamma$ being a separating curve) is satisfied.
\end{lem}
\begin{proof}
  Let us denote by $F$ the seam surface on the left-hand side of the
  identity and by $G_1$ and $G_2$ the seamed surfaces on the right-hand
  side. As surfaces $G_1$ and $G_2$ are identical, they only differ by
  their dot distributions. Denote by $G$ the seam surface which is
  identical to $G_1$ and $G_2$ but with the visible dot removed. There is a canonical
  one-to-one correspondence between $\adm(G)$, $\adm(G_1)$ and $\adm(G_2)$. Let
  $c$ be an element of $\adm(G)$. There are $4$ possible
  local types of $c$. They are given in Figure~\ref{fig:col-2pt-circle}.
  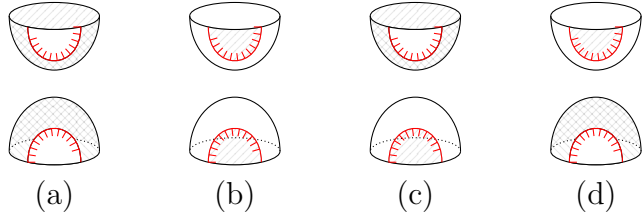
\begin{figure}[ht]
    \centering
    \NB{\tikz[scale=0.6]{\input{\imagesfolder/fe_split-circle-2pt-col}}}
    \caption{The four possible local types of a coloring $c$ of
    $G$. Hash means facet has color $1$ while solid white means
    facet has color $2$.}
    \label{fig:col-2pt-circle}
  \end{figure}

  In cases (c) (resp.{} (d)), one has
  \begin{align*}
  \kup{G_1, c}= \kup{G_2, c} &=
  \alpha_2\kup{G,c} \\ (\text{resp.{}} \quad \kup{G_1, c}= \kup{G_2, c} &=
  \alpha_1\kup{G,c}),
  \end{align*}
  therefore, in both of these cases, $\kup{G_2, c} - \kup{G_1, c}
  =0$.

  The admissible colorings of $G$ of types (a) and (b) are in a one-to-one
  correspondence with admissible colorings of $F$. Let $c$ be a
  coloring of type (a) of $G$, with the corresponding coloring of $F$ 
  still denoted by $c$. On the one hand, one has:
  \begin{gather*}
    \theta_1(G,c) = \theta_1(F,c), \quad \chi(G) = \chi(F)+2, \quad \chi(G_1(c)) =
    \chi(F_1(c)) +1,
     \quad
    \theta(G) = \theta(F) +1.
  \end{gather*}
  The last identity comes from $\gamma$ being a separating
  curve of $F$ and implies 
  \[ \kup{F,c} = -\frac{\kup{G,c} }{\alpha_2 - \alpha_1}.
  \] 
  On the other hand,  $\kup{G_1, c} = \alpha_1 \kup{G,c}$ and
  $\kup{G_2,c} = \alpha_2 \kup{G,c}$ so that, finally,
  $\kup{F,c} = \kup{G_1,c} - \kup{G_2,c}$.

  If the coloring $c$ is of type (b), a similar computation gives as well 
  $\kup{F,c} = \kup{G_1,c} - \kup{G_2,c}$.
  Summing over all admissible colorings, one obtains $\kup{F} = \kup{G_1} - \kup{G_2}$.
\end{proof}

\begin{cor}
  The state space $\kup{\SS^1_{+-}}$ 
  is isomorphic to $A$   as a graded $R$-module.
\end{cor}
\begin{proof}[Sketch of the proof]
  Same argument as for Lemma~\ref{lem:kup-circle-unmarked} shows that
  $\kup{\SS^1_{+-}}$ is generated by $\feSeamedCup{0.4}{}$ and $\feSeamedCup{0.4}{{}_\bullet}$. One shows
  similarly that these elements are $R$-linearly independent. 
\end{proof}

The behavior of the functor $\kup{ \cdot}$ on disjoint unions of 
circles with marked points is not fully monoidal, even if the circles are balanced.  

First, consider the seamed surface $\SS^1_{+-}\times \SS^1$ given by taking meom $\SS^1_{+-}$, a balanced circle with two marks in the plane, and multiplying by the circle $\SS^1$ to get a standardly embedded torus in $\RR^3$ with two seamed circles. The evaluation of this surface  
\[
\kup{\SS^1_{+-}\times \SS^1} = -2 \neq 2 =
\rk_R \kup{\SS^1_{+-}}
\] 
is minus two, different from two, which is the rank of $\kup{\SS^1_{+-}}$ as a free $R$-module. Usually, in the TQFT land, multiplying an $(n-1)$-manifold $M$ by a circle and evaluating gives the dimension of the state space associated to the manifold $M$. In the example above, the evaluation is $-2$ rather than $2$, perhaps implying the need to make the latter a super-module sitting in odd degree and hinting  that we cannot expect monoidality property on the nose without further modifications. 

Similarly, the evaluation $\kup{\SS^1_{++}\times \SS^1} = 2$, but the circle $\SS^1_{++}$ is unbalanced, with the trivial state space of zero dimension. Of course, a cross-section $\SS^1_{++}\sqcup \SS^1_{--}$ of $\SS^1_{++}\times \SS^1$ is balanced, while individual circles in it are not, 
with the state space $\kup{\SS^1_{++}\sqcup \SS^1_{--}}$ nontrivial, showing a failure of monoidality for trivial reasons. 

Next, consider the following six elements of $\kup{\SS_{+-} \sqcup
\SS_{+-}}$:
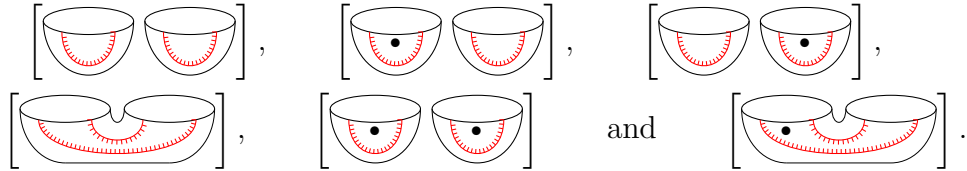
\begin{gather*}
  \left[ \feSeamedCup{0.6}{}\ \feSeamedCup{0.6}{}\right],
  \qquad 
  \left[ \feSeamedCup{0.6}{{}_\bullet}\ \feSeamedCup{0.6}{}\right],
  \qquad
  \left[ \feSeamedCup{0.6}{}\ \feSeamedCup{0.6}{{}_\bullet} \right], \\
  \qquad 
  \left[\feSeamedTwistedTube{0.6}{0.31}{}\right],
   \qquad
  \left[ \feSeamedCup{0.6}{{}_\bullet}\ \feSeamedCup{0.6}{{}_\bullet} \right]
  \qquad \text{and} \qquad
  \left[\feSeamedTwistedTube{0.6}{0.31}{{}_\bullet}\right].
\end{gather*} The surfaces are listed in the order of increasing degree, which is $-2,0,0,0,2,2$ in this order. 
  The matrix of the $R$-bilinear pairing (Gram matrix) on these elements is:
  \begin{equation}\label{eq:Gram-mat-S2S2}
    \begin{pmatrix}
      0&0&0&0&1   &  1  \\
      0&0&1&1& E_1 & E_1  \\
      0&1&0&1& E_1 & E_1 \\
      0&1&1&{-2}&E_1 &  -E_1\\
      1&E_1&E_1&E_1& E_1^2& E_1^2 - E_2 \\
      1&E_1&E_1&-E_1&E_1^2-E_2& -E_1^2 + 2E_2
    \end{pmatrix}   \end{equation}
Its determinant is $4E_1^2-16E_2=4(E_1^2-4E_2)=4\mcD$, which is not a zero divisor in $R$. Consequently, 
these six elements are
$R$-linearly independent in $\kup{\SS^1_{+-} \sqcup
\SS^1_{+-}}$ and span a graded submodule isomorphic to $R^6$, of graded rank $(q+q^{-1})^2+1+q^2$. 

The natural map of state spaces 
\begin{equation*}
    \kup{\SS^1_{+-} } \otimes_R \kup{\SS^1_{+-}} \lra \kup{\SS^1_{+-}  \sqcup\SS^1_{+-}} 
\end{equation*}
given by putting two seamed surfaces with boundary $\SS^1_{+-}$ next to each other is injective. It takes products of standard basis vectors for $\SS^1_{+-}$ to the first, second, third and fifth surfaces, among the six elements above. In particular, this map is not surjective, missing the free $R$-module generated by the two other surfaces, given by a seamed tube, either with a dot or undecorated. 

In particular, the state spaces  $\kup{\SS^1_{+-} \sqcup
\SS^1_{+-}}$ and $\kup{\SS^1_{+-} } \otimes \kup{\SS^1_{+-}}$ are not isomorphic. With more work, one can check that the above six surfaces are a basis of
$\kup{\SS^1_{+-} \sqcup \SS^1_{+-}}$. 

We see that the functor $\kup{\cdot}$, even restricted to meoms with every circle balanced, is not a monoidal functor from the category  of seamed surfaces with boundary embedded in $\RR^3$ to the category of graded $R$-modules, strengthening non-monoidality property observed at the beginning of this section. 

The matrix (\ref{eq:Gram-mat-S2S2}) becomes unimodular if we invert $2$
and  $\mcD$, but we don't explore this here. Note also that the value
of the closed genus three surface is $2\mcD$, see (\ref{eq:gen3surface}). 

Finally, let us inspect the state space $\kup{\SS^1_4}$, where
$\SS^1_4=\SS^1_{+-+-}$ is the circle with four seam points with alternating
co-orientations, see Figure~\ref{fig:S4}.

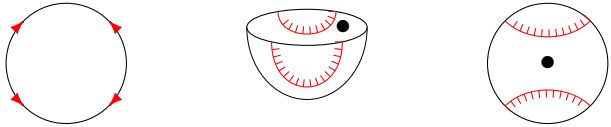
\begin{figure}[ht]
  \centering
  \NB{\tikz[scale=0.8]{\input{\imagesfolder/fe_S4}}}
  \caption{$\SS_4^1$ is on the left; in the middle is a seamed surface from the  empty web 
  $\emptyset$ to $\SS_4^1$, and on the right, its flat representation.}
  \label{fig:S4}
\end{figure}

One can show using the neck-cutting relation that the following four elements generate
$\kup{\SS_4^1}$:
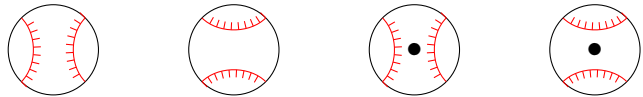
\[
\NB{\tikz[scale=0.6]{\input{\imagesfolder/fe_S4-elements}}}
\]

The matrix of the $R$-bilinear pairing on these vectors is:
  \[
    \begin{pmatrix}
      0 & 0 & 1 & -1 \\
      0 & 0 & 1 & 1   \\
      1 & 1 & E_1 & 0  \\
      -1 & 1 & 0 & E_1
    \end{pmatrix}\]
The determinant of this matrix is $4$, which is not a zero divisor.
Hence, $\kup{\SS_4^1}$ is a free $R$-module of rank $4$ and graded
rank $2(q+q^{-1})$.
At the same time,  $\kup{\SS^1_{+-}}$  and $\kup{\SS^1}$ are free $R$-modules of rank $2$ and graded rank $q+q^{-1}$. We see that rank and graded rank of the module $\kup{\SS^1_{\underline{\ell}}}$ assigned to a circle with a balanced sequence $\underline{\ell}=(\ell_1,\dots, \ell_{2k})$ of signs may depend on $k$. 
This is a more subtle phenomenon than the observation that 
$\kup{\SS^1_{\underline{\ell}}}=0$ for an unbalanced $\underline{\ell}$ due to absence of seamed embedded 
surfaces bounding $\SS^1_{\undell}$. 

Even more substantial dependence of $\kup{\SS^1_{\undell}}$ on the choice of a balanced sequence $\undell$ is investigated in~\cite{InprepKhovKitchKon} for a deformation of evaluation $\kup{\cdot}$.  

The rank of the state space $\kup{\SS^1_{\undell}}$ depends
not only on the length of $\undell$ but also on its cyclic ordering. Consider
$\SS^1_{4'} := \SS^1_{++--}$. One can show that
$\kup{\SS^1_{4'}}$ is generated by
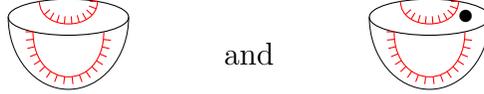
\[
\NB{\NB{\tikz[scale= 0.8]{\input{\imagesfolder/fe_S4p}}}}
  \]
and that these two elements are linearly independent, thus
$\kup{\SS^1_{4'}}$ is a free $R$-module of graded rank $q+q^{-1}$.
Table~\ref{tab:grade-rang-kup} collects rank and graded rank information for state spaces $\kup{\SS^1_{\undell}}$, for shortest balanced sequences $\undell$. Cyclic permutation of the sequence does not change the state space. 
\begin{table}
\centering
 \begin{tabular}{| c | c | c | } 
 \hline 
  Sequence $\undell \rule{0pt}{2.5ex}$ & rank  & graded rank  \\ [0.5ex] 
 \hline 
  $\emptyset$ & 2 & $q+q^{-1}$ \rule{0pt}{2.5ex}\\ [0.5ex] 
 \hline
  ${+}{-}$ \rule{0pt}{2.5ex}& 2 & $q+q^{-1}$ \\ 
 [0.5ex] 
 \hline
  ${+}{+}{-}{-}$ \rule{0pt}{2.5ex} & 2 & $q+q^{-1}$ \\
  [0.5ex] 
 \hline
  ${+}{-}{+}{-}$ \rule{0pt}{2.5ex} & 4 & $2(q+q^{-1})$ \\
 [0.7ex] 
 \hline
\end{tabular}

\vspace{0.2cm} 
\caption{Ranks and graded ranks of the space  $\kup{\SS^1_{\undell}}$ for balanced $\undell$ of short length.}
\label{tab:grade-rang-kup}
\end{table}

This somewhat unusual behavior improves in the alternative evaluation $\kup{\cdot}_\omega$. Our original reason for considering evaluation $\kup{\cdot}_\omega$ was to rederive some of the constructions of Caprau~\cite{CaprauTangle, CaprauWebFoams} and Clark, Morrison, and Walker~\cite{ClarkMorrisonWalkerFunctoriality} via an evaluation framework. We encourage the reader to match the relations in~\cite[Section 3.2]{CaprauWebFoams} with those in Lemma~\ref{lem:evo-loc-rel} above. Orienting a seamed line, as done  in~\cite{CaprauWebFoams}, is equivalent to co-orienting it, as in~\cite{ClarkMorrisonWalkerFunctoriality} and the present paper, as long as the ambient surface comes with an orientation.

Using the neck-cutting relation for $\kup{\cdot}_\omega$ one obtains an analogue of Proposition~\ref{prop:kupp-add-circle}.

\begin{prop}\label{prop:unmarkedcircles}
  Let $C$ be an object of $\SeSu$ obtained by inserting an unmarked circle $\SS^1$ into one of the regions of a meom $C'$ as an innermost circle. Then $\kupo{C} \simeq A
  \otimes_R \kupo{C'}$.  
\end{prop}

\begin{lem}\label{lem:balanced-circle}
  Let $C$ be an object of $\SeSu$ with (at least) one balanced
  circle $S$. Let $C'$ be obtained from $C$ by changing marked points on $S$ 
  to obtain another balanced circle $S'$. Then $\kup{C}_\omega \simeq
  \kup{C'}_\omega$.
\end{lem}

\begin{proof}
  It is enough to consider the case where $S'$ and $S$ differ by a pair of adjacent marked points with opposite
  co-orientation:

  \[
    S= \NB{\tikz[scale =0.6]{\input{\imagesfolder/fe_balanced-S}}}\qquad\text{and}\qquad S'= \NB{\tikz[scale=0.6]{\input{\imagesfolder/fe_balanced-Sp}}}.
  \]
  Relations (\ref{eq:circle-removal-omega-1}) and (\ref{eq:seam-saddle}) imply that the morphisms
  \[
    \kup{\NB{\tikz[scale=0.6]{\input{\imagesfolder/fe_balanced-iso-1}}}}_\omega\co \kup{S'}_\omega \to\kup{S}_\omega \ \qquad\text{and}\qquad
    \omega\kup{\NB{\tikz[scale=0.6]{\input{\imagesfolder/fe_balanced-iso-2}}}}_\omega \co \kup{S}_\omega \to\kup{S'}_\omega 
  \]
 are mutually inverse isomorphisms.
\end{proof}

\begin{cor}
  Let $C$ be an object of $\SeSu$ which consists of $k$ balanced
  circles, possibly nested. Then $\kup{C}_\omega$ is isomorphic to $A_\omega^{\otimes
  k}$ as a graded $R_\omega$-module. In particular, it is free of
  graded rank $(q+ q^{-1})^k$.
\end{cor}

\vspace{0.15in} 

{\bf Acknowledgments.} 
The authors would like to thank Yakov Kononov for valuable discussions.  
M.K. was partially supported by NSF grants DMS-1664240, and DMS-1807425 while working on this paper. L.-H.R.{} was supported by NCCR SwissMAP, funded by the Swiss National Science Foundation.

\bibliographystyle{abbrvurl}
\bibliography{biblio}

\end{document}

%% file: cf_cube.tex
\begin{scope}
  \newcommand{\mywidth}{0.5}
  \node (R) at (-2, 2) {$R$};
  \node (Rd) at ( 2, 2) {$R_{\mathcal{D}}$};
  \node (Ra) at (-2, -2) {$R_\alpha$};
  \node (Rad) at ( 2,-2) {$R_{\alpha{\mathcal{D}}}$};
  \node (A) at (0, 4) {$A$};
  \node (Ad) at ( 4, 4) {$A_{\mathcal{D}}$};
  \node (Aa) at (0, 0) {$A_\alpha$};
  \node (Aad) at ( 4,0) {$A_{\alpha{\mathcal{D}}}$};

  \foreach \x in { R, Ra, Rad, Rd}{
  \draw[red!50!black, dotted] ($(\x) + (135:0.5cm)$) arc(135:315:0.5cm) --
  +(2,2) arc(-45:135:0.5cm) -- cycle; }
  \node[red!50!black] at ($(R) + (-1,0)$) {\hyperlink{frobext1}{
  1}};
  \node[red!50!black] at ($(Ra) + (-1,0)$) {\hyperlink{frobext3}{ 3}};

  \node[red!50!black] at ($(Ad) + (1,0)$) {\hyperlink{frobext2}{
  2}};
  \node[red!50!black] at ($(Aad) + (1,0)$) {\hyperlink{frobext4}{ 4}};

  \begin{scope}[white, line width =2mm]
  \draw (R) -- (Rd) -- (Rad) -- (Ra) -- (R);
  \draw (A) -- (Ad) -- (Aad) -- (Aa) -- (A);
  \end{scope}

  \draw[-to] (Ra)  -- (Rad) node [above, midway, font= \small]
  {$\mathcal{D}^{-1}$};
  \draw[->] (A)  -- (Ad) node [midway, above, font= \small]
  {$\mathcal{D}^{-1}$};
  \draw[->] (Aa)  -- (Aad) node [pos=0.7, above, font= \small]
  {$\mathcal{D}^{-1}$};  

  \draw[->] (R) -- (A) node [left, midway, font = \small] {$X$} node
  [right, midway, font = \small] {$:2$};
  \draw[->] (Ra) -- (Aa) node [left, midway, font = \small] {$X$}
  node [right, midway, font = \small] {$:2$};
  \draw[->] (Rd) -- (Ad) node [left, midway, font = \small] {$X$}
  node [right, midway, font = \small] {$:2$};
  \draw[->] (Rad) -- (Aad) node [left, midway, font = \small] {$X$} node
  [right, midway, font = \small] {$:2$};
  
  \draw[->] (R) -- (Ra) node [left, midway, font = \small] {$\alpha$} node
  [right, midway, font = \small] {$:2$};
  \draw[->] (A) -- (Aa) node [left, pos=0.7, font = \small] {$\alpha$} node
  [right, pos=0.7, font = \small] {$:2$};
  \draw[->] (Ad) -- (Aad) node [left, midway, font = \small] {$\alpha$} node
  [right, midway, font = \small] {$:2$};
  
  \draw[white, line width =2mm] (R) -- (Rd);
  \draw[white, line width =2mm] (Rd) -- (Rad);
  
  \draw[->] (R)  -- (Rd) node [pos=0.7, above, font= \small]
  {$\mathcal{D}^{-1}$};

  \draw[->] (Rd) -- (Rad) node [left, pos=0.7, font = \small] {$\alpha$} node
  [right, pos=0.7, font = \small] {$:2$};

\end{scope}

%% file: cf_frob.tex
\newcommand{\yparam}{-2}
\begin{scope}[scale = 1, font= \small]
  \begin{scope}[xshift =0cm]
    \draw (1, 0) arc (0:360:1cm and 0.3cm);
    \draw[very thin] (1,0) arc  (0:-180:1cm);
    \node at (0, \yparam) {$1_A$};
  \end{scope}
  \begin{scope}[xshift =5cm]
    \draw (1, 0) arc (0:360:1cm and 0.3cm);
    \draw[very thin] (1,0) arc  (0:-180:1cm);
    \node at (0, -0.6) {$\bullet$};
    \node at (0, \yparam) {$X$};
  \end{scope}
  \begin{scope}[xshift = 10cm]
    \draw[densely dotted] (1, 0) arc (0:180:1cm and 0.3cm);
    \draw (1, 0) arc (0:-180:1cm and 0.3cm);
    \draw[very thin] (1,0) arc  (0: 180:1cm);
    \node at (0, \yparam) {$\epsilon$};
  \end{scope}
  \begin{scope}[xshift = 15cm]
    \draw[densely dotted] (1, -1) arc (0:180:1cm and 0.3cm);
    \draw (1, -1) arc (0:-180:1cm and 0.3cm);
    \draw (1, 1) arc (0:360:1cm and 0.3cm);
    \draw[very thin] (1,-1) -- +(0,2);
    \draw[very thin] (-1,-1) -- +(0,2);
    \node at (0, \yparam) {$\mathrm{id}_A$};
  \end{scope}
  \begin{scope}[xshift = 20cm]
    \draw[densely dotted] (1, -1) arc (0:180:1cm and 0.3cm);
    \draw (1, -1) arc (0:-180:1cm and 0.3cm);
    \draw (1, 1) arc (0:360:1cm and 0.3cm);
    \draw[very thin] (1,  -1) -- +(0,2);
    \draw[very thin] (-1, -1) -- +(0,2);
    \node at (0, 0) {$\bullet$};
    \node at (0, \yparam) {$\cdot X$};
  \end{scope}
 \begin{scope}[yshift = -5cm, xshift = -30cm]
  
  \begin{scope}[xshift = 27cm]
    \draw[densely dotted] (3, -1) arc (0:180:1cm and 0.3cm);
    \draw (3, -1) arc (0:-180:1cm and 0.3cm);
    \draw[densely dotted] (-1, -1) arc (0:180:1cm and 0.3cm);
    \draw (-1, -1) arc (0:-180:1cm and 0.3cm);
    \draw (1, 1) arc (0:360:1cm and 0.3cm);
    \draw[very thin] (1,  -1) .. controls +(0,0.5) and + (0, 0.5) .. +(-2,0);
    \draw[very thin] (3,  -1) .. controls +(0,0.3) and + (0, -0.3) .. +(-2,2);
    \draw[very thin] (-3,  -1) .. controls +(0,0.3) and + (0, -0.3) .. +(2,2);
    \node at (0, \yparam) {$m$};
  \end{scope}
  \begin{scope}[xshift = 34cm]
    \draw[densely dotted] (1, -1) arc (0:180:1cm and 0.3cm);
    \draw (1, -1) arc (0:-180:1cm and 0.3cm);
    \draw (-1, 1) arc (0:360:1cm and 0.3cm);
    \draw (3, 1) arc (0:360:1cm and 0.3cm);
    \draw[very thin] (1,  1) .. controls +(0,-0.5) and + (0, -0.5) .. +(-2,0);
    \draw[very thin] (-1,  -1) .. controls +(0,0.3) and + (0, -0.3) .. +(-2,2);
    \draw[very thin] (1,  -1) .. controls +(0,0.3) and + (0, -0.3) .. +(2,2);
    \node at (0, \yparam) {$\Delta$};
  \end{scope}
  \begin{scope}[xshift = 42cm]
    \draw[densely dotted] (1, -1) arc (0:180:1cm and 0.3cm);
    \draw (1, -1) arc (0:-180:1cm and 0.3cm);
    \draw (1, 1) arc (0:360:1cm and 0.3cm);
    \draw[very thin] (-1, -1) .. controls +(0, 0.3) and + (0, -0.3) ..   (-3,0) .. controls +(0, 0.3) and +(0, -0.3) ..  (-1, 1);
    \draw[very thin] ( 1, -1) .. controls +(0, 0.3) and + (0, -0.3) ..   ( 3,0) .. controls +(0, 0.3) and +(0, -0.3) ..  ( 1, 1);
    \draw[very thin] (0, -0.2) arc (270: 240: 2) coordinate[pos=0.8] (a);
    \draw[very thin] (0, -0.2) arc (270: 300: 2);
    \draw[very thin] (a) arc (114: 66: 2);
    \node at (0, \yparam) {$m\Delta$};
  \end{scope}
  \end{scope}
\end{scope}

%% file: fe_biggercube.tex
\newcommand{\myspecialcolor}{blue!50!black}
\begin{scope}[scale=2]
  \node (R) at (-2, 2) {$R=\ZZ[E_1, E_2]$};
  \node[align=center] (Rd) at ( 2, 2)
  {$R_{\mcD} = R[\mcD^{-1}]$ \\ $\mcD = E_1^2 -4E_2$};
  \node[align=center] (Ra) at (-2, -2)
  {$R_\alpha=\ZZ[\alpha_1, \alpha_2]$ \\ $\alpha_1+ \alpha_2 = E_1$ \\
  $\alpha_1 \alpha_2 = E_2$ };
  \node[align=center] (Rad) at ( 2,-2)
  {$R_{\alpha{\mcD}}=$ \\ $\ZZ[\alpha_1, \alpha_2, (\alpha_1 -
  \alpha_2)^{-1}]$ \\ $\alpha_1+ \alpha_2 = E_1$ \\
  $\alpha_1 \alpha_2 = E_2$};
  \node[ align=center] (A) at (0, 4) {$A =$ \\ $ R[X]/(X^2-E_1X + E_2)$};
  \node[align=center] (Ad) at ( 4, 4)
  {$A_{\mcD}=A[\mcD^{-1}]=$ \\ $ R_{\mcD}[X]/(X^2-E_1X + E_2)$};
  \node[align=center] (Aa) at (0, 0) {$A_\alpha=
  A[\alpha_1, \alpha_2]=$ \\ $R_{\alpha}[X]/((X-\alpha_1)(X-\alpha_2))$
  \\ $\alpha_1+ \alpha_2 = E_1$ \\
  $\alpha_1 \alpha_2 = E_2$};
  \node[below of=Aa, xshift = 0.8cm, yshift=-2mm,  font= \footnotesize, color= \myspecialcolor] (Aab)
  {\emph{not} an integral domain};
  \node[align=center] (Aad) at ( 4,0)
  {$A_{\alpha{\mcD}}=$ \\ $ R_{\alpha
  \mcD}[X]/ ((X-\alpha_1)(X-\alpha_2))$ \\ $=  R_{\alpha
  \mcD} e_1 \times  R_{\alpha
  \mcD} e_2$ \\ 
  $e_1 = \frac{X-\alpha_1}{\alpha_{2} -
  \alpha_1}, \ e_2 = \frac{X-\alpha_2}{\alpha_{1} - \alpha_2} $ 
  };
  \draw[-to] (Ra)  -- (Rad) node [above, midway, font= \small]
  {$\mcD^{-1}$};
  \draw[->] (A)  -- (Ad) node [midway, above, font= \small]
  {$\mcD^{-1}$};
  \draw[->] (Aa)  -- (Aad) node [pos=0.7, above, font= \small]
  {$\mcD^{-1}$};  

  \draw[->] (R) -- (A) node [left, midway, font = \small] {$X$} node
  [right, midway, font = \small] {$:2$};
  \draw[->] (Ra) -- (Aa) node [left, midway, font = \small] {$X$}
  node [right, midway, font = \small] {$:2$};
  \draw[->] (Rd) -- (Ad) node [left, pos=0.8, font = \small] {$X$}  node
  [right, pos=0.8, font = \small] {$:2$} node[sloped,
  color=\myspecialcolor, pos=0.3, font=\footnotesize, above] {separable};
  \draw[->] (Rad) -- (Aad) node [left, pos=.8, font = \small] {$X$} node
  [right, pos=0.8, font = \small] {$:2$} node[sloped,
  color=\myspecialcolor, pos=0.3, font=\footnotesize, above] {separable};
  
  \draw[->] (R) -- (Ra) node [left, midway, font = \small] {$\alpha$} node
  [right, midway, font = \small] {$:2$};
  \draw[->] (A) -- (Aa) node [left, pos=0.7, font = \small] {$\alpha$} node
  [right, pos=0.7, font = \small] {$:2$};
  \draw[->] (Ad) -- (Aad) node [left, pos= 0.5, font = \small]
  {$\alpha$} node [right, pos=0.5, font = \small] {$:2$};;  
  \draw[white, line width =2mm] (R) -- (Rd);
  \draw[white, line width =2mm] (Rd) -- (Rad);
  
  \draw[->] (R)  -- (Rd) node [pos=0.7, above, font= \small]
  {$\mcD^{-1}$};

  \draw[->] (Rd) -- (Rad) node [left, pos=0.7, font = \small] {$\alpha$} node
  [right, pos=0.7, font = \small] {$:2$};

\end{scope}

%% file: fe_cube-homologies.tex
\begin{scope}
  \node (R) at (-2, 2) {$C(D)$};
  \node (Rd) at ( 2, 2) {$C_{\mcD}(D)$};
  \node (Ra) at (-2, -2) {$C_\alpha(D)$};
  \node (Rad) at ( 2,-2) {$C_{\alpha{\mcD}}(D)$};
  \node (A) at (0, 4) {$\mtH(D)$};
  \node (Ad) at ( 4, 4) {$\mtH_{\mcD}(D)$};
  \node (Aa) at (0, 0) {$\mtH_\alpha(D)$};
  \node (Aad) at ( 4,0) {$\mtH_{\alpha{\mcD}}(D)$};
  \draw[-to] (Ra)  -- (Rad) node [above, midway, font= \small]
  {$\mathcal{D}^{-1}$};
  \draw[->] (A)  -- (Ad) node [midway, above, font= \small]
  {$\mathcal{D}^{-1}$};
  \draw[->] (Aa)  -- (Aad) node [pos=0.7, above, font= \small]
  {$\mathcal{D}^{-1}$};  
\begin{scope}[decoration={snake, segment length=1mm, amplitude=0.3mm,
  post length=2pt}, font= \scriptsize]
  \draw[-to, {decorate}] (R) -- (A) node
  [above, midway, sloped] {homology};
  \draw[decorate,->] (Ra) -- (Aa);
  \draw[decorate,->] (Rd) -- (Ad);
  \draw[decorate,->] (Rad) -- (Aad);
  \end{scope}
  \draw[->] (R) -- (Ra) node [left, midway, font = \small] {$\alpha$} node
  [right, midway, font = \small] {$:2$};
  \draw[->] (A) -- (Aa) node [left, pos=0.7, font = \small] {$\alpha$} node
  [right, pos=0.7, font = \small] {$:2$};
  \draw[->] (Ad) -- (Aad) node [left, midway, font = \small] {$\alpha$} node
  [right, midway, font = \small] {$:2$};
  
  \draw[white, line width =2mm] (R) -- (Rd);
  \draw[white, line width =2mm] (Rd) -- (Rad);
  
  \draw[->] (R)  -- (Rd) node [pos=0.7, above, font= \small]
  {$\mathcal{D}^{-1}$};

  \draw[->] (Rd) -- (Rad) node [left, pos=0.7, font = \small] {$\alpha$} node
  [right, pos=0.7, font = \small] {$:2$};

\end{scope}

%% file: fe_sigmaplus.tex
\begin{scope}
  \begin{scope}[yshift= 1.5cm]
    \draw (1,0) arc (0:360:1cm and 0.3cm)coordinate [pos=0](e) coordinate [pos=0.5](f);
  \end{scope}
  \begin{scope}[yshift =-1.5cm]
    \draw (1,0) arc (0:-180:1cm and 0.3cm)coordinate [pos=0](g) coordinate [pos=1](h);
    \draw[densely dotted] (1,0) arc (0:180:1cm and 0.3cm); 
  \end{scope}
  \draw[very thin] (e) -- (g);
  \draw[very thin] (f) -- (h);
  \fill[pattern=vertical lines, pattern color= red, draw = none] (1,0) arc  (0:180:1cm and 0.3cm) -- (-1, 0.2) arc (180:0: 1cm and 0.3cm) --cycle;
  \draw[red] (1,0) arc  (0:180:1cm and 0.3cm);
  \draw [line width =5mm, white, dashed ] (0.9,0.2) -- (-0.9, 0.20);  
  \fill[pattern=vertical lines, pattern color= red, draw = none] (1,0) arc  (0:-180:1cm and 0.3cm) -- (-1, 0.2) arc (-180:0: 1cm and 0.3cm) --cycle;
  \draw[red] (1,0) arc  (0:-180:1cm and 0.3cm);
\end{scope}

%% file: fe_sigmaminus.tex
\begin{scope}
  \begin{scope}[yshift= 1.5cm]
    \draw (1,0) arc (0:360:1cm and 0.3cm)coordinate [pos=0](e) coordinate [pos=0.5](f);
  \end{scope}
  \begin{scope}[yshift =-1.5cm]
    \draw (1,0) arc (0:-180:1cm and 0.3cm)coordinate [pos=0](g) coordinate [pos=1](h);
    \draw[densely dotted] (1,0) arc (0:180:1cm and 0.3cm); 
  \end{scope}
  \draw[very thin] (e) -- (g);
  \draw[very thin] (f) -- (h);
  \fill[pattern=vertical lines, pattern color= red, draw = none] (1,0) arc  (0:180:1cm and 0.3cm) -- (-1, -0.2) arc (180:0: 1cm and 0.3cm) --cycle;
  \draw[red] (1,0) arc  (0:180:1cm and 0.3cm);
  \draw [line width =5mm, white, dashed ] (0.9,0.2) -- (-0.9, 0.20);  
  \fill[pattern=vertical lines, pattern color= red, draw = none] (1,0) arc  (0:-180:1cm and 0.3cm) -- (-1, -0.2) arc (-180:0: 1cm and 0.3cm) --cycle;
  \draw[red] (1,0) arc  (0:-180:1cm and 0.3cm);
\end{scope}

%% file: fe_sigmaplusminus.tex
\begin{scope}
  \begin{scope}[yshift= 1.5cm]
    \draw (1,0) arc (0:360:1cm and 0.3cm)coordinate [pos=0](e) coordinate [pos=0.5](f);
  \end{scope}
  \begin{scope}[yshift =-1.5cm]
    \draw (1,0) arc (0:-180:1cm and 0.3cm)coordinate [pos=0](g) coordinate [pos=1](h);
    \draw[densely dotted] (1,0) arc (0:180:1cm and 0.3cm); 
  \end{scope}
  \draw[very thin] (e) -- (g);
  \draw[very thin] (f) -- (h);
  \begin{scope}[yshift= -0.6cm]
    \fill[pattern=vertical lines, pattern color= red, draw = none]
    (1,0) arc (0:180:1cm and 0.3cm) -- (-1, 0.2) arc (180:0: 1cm and
    0.3cm) --cycle; \draw[red] (1,0) arc (0:180:1cm and 0.3cm); \draw
    [line width =5mm, white, dashed ] (0.9,0.3) -- (-0.9,
    0.30); 
    \fill[pattern=vertical lines, pattern color= red, draw = none]
    (1,0) arc (0:-180:1cm and 0.3cm) -- (-1, 0.2) arc (-180:0: 1cm and
    0.3cm) --cycle; \draw[red] (1,0) arc (0:-180:1cm and 0.3cm);
  \end{scope}
  \begin{scope}[yshift=0.6cm]
      \fill[pattern=vertical lines, pattern color= red, draw = none] (1,0) arc  (0:180:1cm and 0.3cm) -- (-1, -0.2) arc (180:0: 1cm and 0.3cm) --cycle;
  \draw[red] (1,0) arc  (0:180:1cm and 0.3cm);
  \draw [line width =5mm, white, dashed ] (0.9,0.2) -- (-0.9, 0.20);  
  \fill[pattern=vertical lines, pattern color= red, draw = none] (1,0) arc  (0:-180:1cm and 0.3cm) -- (-1, -0.2) arc (-180:0: 1cm and 0.3cm) --cycle;
  \draw[red] (1,0) arc  (0:-180:1cm and 0.3cm);
  \end{scope}
\end{scope}

%% file: fe_sigmaminusplus.tex
\begin{scope}
 
  \begin{scope}[yshift =-1.5cm]
    \draw (1,0) arc (0:-180:1cm and 0.3cm)coordinate [pos=0](g) coordinate [pos=1](h);
    \draw[densely dotted] (1,0) arc (0:180:1cm and 0.3cm); 
  \end{scope}
 
  \begin{scope}[yshift= 0.6cm]
    \fill[pattern=vertical lines, pattern color= red, draw = none]
    (1,0) arc (0:180:1cm and 0.3cm) -- (-1, 0.2) arc (180:0: 1cm and
    0.3cm) --cycle; \draw[red] (1,0) arc (0:180:1cm and 0.3cm); \draw
    [line width =5mm, white, dashed ] (0.9,0.3) -- (-0.9,
    0.30); 
    \fill[pattern=vertical lines, pattern color= red, draw = none]
    (1,0) arc (0:-180:1cm and 0.3cm) -- (-1, 0.2) arc (-180:0: 1cm and
    0.3cm) --cycle; \draw[red] (1,0) arc (0:-180:1cm and 0.3cm);
  \end{scope}
  \begin{scope}[yshift=-0.6cm]
      \fill[pattern=vertical lines, pattern color= red, draw = none] (1,0) arc  (0:180:1cm and 0.3cm) -- (-1, -0.2) arc (180:0: 1cm and 0.3cm) --cycle;
  \draw[red] (1,0) arc  (0:180:1cm and 0.3cm);
  \draw [line width =5mm, white, dashed ] (0.9,0.2) -- (-0.9, 0.20);  
  \fill[pattern=vertical lines, pattern color= red, draw = none] (1,0) arc  (0:-180:1cm and 0.3cm) -- (-1, -0.2) arc (-180:0: 1cm and 0.3cm) --cycle;
  \draw[red] (1,0) arc  (0:-180:1cm and 0.3cm);
  \end{scope}
   \begin{scope}[yshift= 1.5cm]
    \draw (1,0) arc (0:360:1cm and 0.3cm)coordinate [pos=0](e) coordinate [pos=0.5](f);
  \end{scope}
 \draw[very thin] (e) -- (g);
  \draw[very thin] (f) -- (h);
\end{scope}

%% file: fe_3holed-sphere.tex
\begin{scope}[rotate=0, very thin]
  \draw[thin, red]            (1, -2) arc (0:180: 1cm and 0.3cm);
  \fill[pattern=vertical lines, pattern color= red, draw = none]
  (1,-2) arc  (0:180:1cm and 0.3cm) -- (-1, -1.8) arc (180:0: 1cm
  and 0.3cm) --cycle;
    \draw [line width =5mm, white, dashed ] (0.9,-1.7) -- (-0.9, -1.7);  
  \draw[thin, red]            (1, -2) arc (360:180: 1cm and 0.3cm);
  \fill[pattern=vertical lines, pattern color= red, draw = none]
  (1,-2) arc  (360:180:1cm and 0.3cm) -- (-1, -1.8) arc (180:360: 1cm
  and 0.3cm) --cycle;
    \draw (1, -3) -- (1,-2) .. controls +(0,0.5) and +(-120:0.5)  .. (-30:1.5);
    \draw (-1, -3) -- (-1,-2) .. controls +(0,0.5) and +(-60:0.5)  .. (-150:1.5);
  \draw[thin] (1, -3) arc (360:180: 1cm and 0.3cm);
  \draw[thin, densely dotted] (1, -3) arc (0:180: 1cm and 0.3cm);
  \end{scope}
\begin{scope}[rotate=120, very thin]
  \draw[thin, red]            (1, -2) arc (360:180: 1cm and 0.3cm);
  \fill[pattern=north east lines, pattern color= red, draw = none]
  (1,-2) arc  (360:180:1cm and 0.3cm) -- (-1, -1.8) arc (180:360: 1cm
  and 0.3cm) --cycle;
  \draw [line width =5mm, white, densely dashed ] (0.9,-2.2) -- (-0.9, -2.2);  
  \draw[thin, red]            (1, -2) arc (0:180: 1cm and 0.3cm);
  \fill[pattern=north east lines, pattern color= red, draw = none]
  (1,-2) arc  (0:180:1cm and 0.3cm) -- (-1, -1.8) arc (180:0: 1cm
  and 0.3cm) --cycle;
  \draw (1, -3) -- (1,-2) .. controls +(0,0.5) and +(-120:0.5)  .. (-30:1.5);
  \draw (-1, -3) -- (-1,-2) .. controls +(0,0.5) and +(-60:0.5)  .. (-150:1.5);
  \draw[thin] (1, -3) arc (360:180: 1cm and 0.3cm);
  \draw[thin,] (1, -3) arc (0:180: 1cm and 0.3cm);
\end{scope}
\begin{scope}[rotate=240, very thin]
  \draw[thin, red]            (1, -2) arc (360:180: 1cm and 0.3cm);
  \fill[pattern=north west lines, pattern color= red, draw = none]
  (1,-2) arc  (360:180:1cm and 0.3cm) -- (-1, -1.8) arc (180:360: 1cm
  and 0.3cm) --cycle;
  \draw [line width =5mm,white, densely dashed ] (0.9,-2.2) -- (-0.9, -2.2);  
  \draw[thin, red]            (1, -2) arc (0:180: 1cm and 0.3cm);
  \fill[pattern=north west lines, pattern color= red, draw = none]
  (1,-2) arc  (0:180:1cm and 0.3cm) -- (-1, -1.8) arc (180:0: 1cm
  and 0.3cm) --cycle;
  \draw (1, -3) -- (1,-2) .. controls +(0,0.5) and +(-120:0.5)  .. (-30:1.5);
  \draw (-1, -3) -- (-1,-2) .. controls +(0,0.5) and +(-60:0.5)  .. (-150:1.5);
  \draw[thin] (1, -3) arc (360:180: 1cm and 0.3cm);
  \draw[thin] (1, -3) arc (0:180: 1cm and 0.3cm);
\end{scope}

%% file: fe_3holed-sphere-2.tex
\begin{scope}[rotate=0, very thin]
  \draw[thin, red]            (1, -2) arc (0:180: 1cm and 0.3cm);
  \fill[pattern=vertical lines, pattern color= red, draw = none]
  (1,-2) arc  (0:180:1cm and 0.3cm) -- (-1, -2.2) arc (180:0: 1cm
  and 0.3cm) --cycle;
  \draw [line width =5mm, white, dashed ] (0.9,-1.9) -- (-0.9, -1.9);  
  \draw[thin, red]            (1, -2) arc (360:180: 1cm and 0.3cm);
  \fill[pattern=vertical lines, pattern color= red, draw = none]
  (1,-2) arc  (360:180:1cm and 0.3cm) -- (-1, -2.2) arc (180:360: 1cm
  and 0.3cm) --cycle;
    \draw (1, -3) -- (1,-2) .. controls +(0,0.5) and +(-120:0.5)  .. (-30:1.5);
    \draw (-1, -3) -- (-1,-2) .. controls +(0,0.5) and +(-60:0.5)  .. (-150:1.5);
  \draw[thin] (1, -3) arc (360:180: 1cm and 0.3cm);
  \draw[thin, densely dotted] (1, -3) arc (0:180: 1cm and 0.3cm);
  \end{scope}
\begin{scope}[rotate=120, very thin]
  \draw[thin, red]            (1, -2) arc (360:180: 1cm and 0.3cm);
  \fill[pattern=north east lines, pattern color= red, draw = none]
  (1,-2) arc  (360:180:1cm and 0.3cm) -- (-1, -2.2) arc (180:360: 1cm
  and 0.3cm) --cycle;
  \draw [line width =5mm, white, densely dashed ] (0.9,-2.3) -- (-0.9, -2.3);  
  \draw[thin, red]            (1, -2) arc (0:180: 1cm and 0.3cm);
  \fill[pattern=north east lines, pattern color= red, draw = none]
  (1,-2) arc  (0:180:1cm and 0.3cm) -- (-1, -2.2) arc (180:0: 1cm
  and 0.3cm) --cycle;
  \draw (1, -3) -- (1,-2) .. controls +(0,0.5) and +(-120:0.5)  .. (-30:1.5);
  \draw (-1, -3) -- (-1,-2) .. controls +(0,0.5) and +(-60:0.5)  .. (-150:1.5);
  \draw[thin] (1, -3) arc (360:180: 1cm and 0.3cm);
  \draw[thin,] (1, -3) arc (0:180: 1cm and 0.3cm);
\end{scope}
\begin{scope}[rotate=240, very thin]
  \draw[thin, red]            (1, -2) arc (360:180: 1cm and 0.3cm);
  \fill[pattern=north west lines, pattern color= red, draw = none]
  (1,-2) arc  (360:180:1cm and 0.3cm) -- (-1, -2.2) arc (180:360: 1cm
  and 0.3cm) --cycle;
  \draw [line width =5mm,white, densely dashed ] (0.9,-2.3) -- (-0.9, -2.3);  
  \draw[thin, red]            (1, -2) arc (0:180: 1cm and 0.3cm);
  \fill[pattern=north west lines, pattern color= red, draw = none]
  (1,-2) arc  (0:180:1cm and 0.3cm) -- (-1, -2.2) arc (180:0: 1cm
  and 0.3cm) --cycle;
  \draw (1, -3) -- (1,-2) .. controls +(0,0.5) and +(-120:0.5)  .. (-30:1.5);
  \draw (-1, -3) -- (-1,-2) .. controls +(0,0.5) and +(-60:0.5)  .. (-150:1.5);
  \draw[thin] (1, -3) arc (360:180: 1cm and 0.3cm);
  \draw[thin] (1, -3) arc (0:180: 1cm and 0.3cm);
\end{scope}

%% file: fe_3holed-sphere-clean.tex
\begin{scope}[rotate=0, very thin]
  \draw[thin] (1, -3) arc (360:180: 1cm and 0.3cm);
  \draw[thin, densely dotted] (1, -3) arc (0:180: 1cm and 0.3cm);
    \draw (1, -3) -- (1,-2) .. controls +(0,0.5) and +(-120:0.5)  .. (-30:1.5);
    \draw (-1, -3) -- (-1,-2) .. controls +(0,0.5) and +(-60:0.5)  .. (-150:1.5);
  \end{scope}
\begin{scope}[rotate=120, very thin]
  \draw[thin] (1, -3) arc (360:180: 1cm and 0.3cm);
  \draw[thin,] (1, -3) arc (0:180: 1cm and 0.3cm);
  \draw (1, -3) -- (1,-2) .. controls +(0,0.5) and +(-120:0.5)  .. (-30:1.5);
  \draw (-1, -3) -- (-1,-2) .. controls +(0,0.5) and +(-60:0.5)  .. (-150:1.5);
\end{scope}
\begin{scope}[rotate=240, very thin]
  \draw[thin] (1, -3) arc (360:180: 1cm and 0.3cm);
  \draw[thin] (1, -3) arc (0:180: 1cm and 0.3cm);
  \draw (1, -3) -- (1,-2) .. controls +(0,0.5) and +(-120:0.5)  .. (-30:1.5);
  \draw (-1, -3) -- (-1,-2) .. controls +(0,0.5) and +(-60:0.5)  .. (-150:1.5);
\end{scope}

%% file: fe_move-sigma-delta1.tex
\begin{scope}
  \draw[very thin, dotted] (1, -1) arc (0:180:1cm and 0.3cm);
  \draw[very thin] (1, -1) arc (0:-180:1cm);
  \draw[very thin, dotted] (1, -1) arc (0:-180:1cm and 0.3cm);
  \fill[pattern=vertical lines, pattern color= red, draw = none] (-1,1) arc  (0:180:1cm and 0.3cm) -- (-3, 1.15) arc (180:0: 1cm and 0.3cm) --cycle;
  \draw[red] (-1,1) arc  (0:180:1cm and 0.3cm);
  \draw [line width =5mm, white, dashed ] (-1.1,1.2) -- (-2.9, 1.20);  
  \fill[pattern=vertical lines, pattern color= red, draw = none]
  (-1,1) arc  (0:-180:1cm and 0.3cm) -- (-3, 1.15) arc (-180:0: 1cm and
  0.3cm) --cycle;
    \draw[red] (-1,1) arc  (0:-180:1cm and 0.3cm);
      \draw (-1, 1.8) arc (0:360:1cm and 0.3cm);
  \draw (3, 1.8) arc (0:360:1cm and 0.3cm);
  \draw[very thin] (1,  1.8) -- (1,1) .. controls +(0,-0.5) and + (0, -0.5)
  .. ++(-2,0)-- +(0,0.8);
    \draw[very thin] (-1,  -1) .. controls +(0,0.3) and + (0, -0.3)
    .. ++(-2,2) -- +(0, 0.8);
    \draw[very thin] (1,  -1) .. controls +(0,0.3) and + (0, -0.3)
    .. ++(2,2)-- +(0, 0.8);
\end{scope}

%% file: fe_move-sigma-delta2.tex
\begin{scope}
  \draw[very thin, dotted] (1, -1) arc (0:180:1cm and 0.3cm);
  \draw[very thin] (1, -1) arc (0:-180:1cm);
  \draw[very thin, dotted] (1, -1) arc (0:-180:1cm and 0.3cm);
  \fill[pattern=vertical lines, pattern color= red, draw = none] (3,1.15) arc  (0:180:1cm and 0.3cm) -- (1, 1) arc (180:0: 1cm and 0.3cm) --cycle;
  \draw[red] (3,1.15) arc  (0:180:1cm and 0.3cm);
  \draw [line width =5mm, white, dashed ] (2.9,1.4) -- (1.1, 1.4);  
  \fill[pattern=vertical lines, pattern color= red, draw = none]
  (3,1.15) arc  (0:-180:1cm and 0.3cm) -- (1, 1) arc (-180:0: 1cm and
  0.3cm) --cycle;
    \draw[red] (3,1.15) arc  (0:-180:1cm and 0.3cm);
    \draw (3, 1.8) arc (0:360:1cm and 0.3cm);
      \draw (-1, 1.8) arc (0:360:1cm and 0.3cm);
  \draw[very thin] (1,  1.8) -- (1,1) .. controls +(0,-0.5) and + (0, -0.5)
  .. ++(-2,0)-- +(0,0.8);
    \draw[very thin] (-1,  -1) .. controls +(0,0.3) and + (0, -0.3)
    .. ++(-2,2) -- +(0, 0.8);
    \draw[very thin] (1,  -1) .. controls +(0,0.3) and + (0, -0.3)
    .. ++(2,2)-- +(0, 0.8);
\end{scope}

%% file: fe_mdeltasigma2.tex
  \begin{scope}
    \draw[very thin] (0, -0.2) arc (270: 240: 3) coordinate[pos=0.8] (a);
    \draw[very thin] (0, -0.2) arc (270: 300: 3);
    \draw[very thin] (a) arc (114: 66: 3);
    \draw[very thin] (-1, -2) .. controls +(0, 1) and + (0, -1) ..
    (-3,0) .. controls +(0, 1) and +(0, -1) ..  (-1, 2)
    coordinate[pos=0.5] (b);
      \fill[pattern=north east lines, pattern color= red, draw = none]
      (a) -- ($(a)+ (-0.1,-0.1)$) .. controls +(0.1, 0.6) and +( 0.6, 0)
      .. ($(b)+ (-0.1,-0.1)$) -- (b) .. controls +(0.6,0) and +( 0.1,0.6)
      .. (a);
    \draw[red] (a).. controls +(0.1, 0.6) and +( 0.6, 0) .. (b);

    \draw[densely dotted] (1, -2) arc (0:180:1cm and 0.3cm);
    \draw (1, -2) arc (0:-180:1cm and 0.3cm);
    \draw (1, 2) arc (0:360:1cm and 0.3cm);
    \draw[very thin] (-1, -2) .. controls +(0, 1) and + (0, -1) ..   (-3,0) .. controls +(0, 1) and +(0, -1) ..  (-1, 2);
    \draw[very thin] ( 1, -2) .. controls +(0, 1) and + (0, -1) ..   ( 3,0) .. controls +(0, 1) and +(0, -1) ..  ( 1, 2);
  \end{scope}

%% file: fe_mdeltasigma1.tex
  \begin{scope}
    \draw[very thin] (0, -0.2) arc (270: 240: 3) coordinate[pos=0.8] (a);
    \draw[very thin] (0, -0.2) arc (270: 300: 3);
    \draw[very thin] (a) arc (114: 66: 3);
    \draw[very thin] (-1, -2) .. controls +(0, 1) and + (0, -1) ..
    (-3,0) .. controls +(0, 1) and +(0, -1) ..  (-1, 2)
    coordinate[pos=0.5] (b);
      \fill[pattern=north east lines, pattern color= red, draw = none]
      (a) -- ($(a)+ (0.1,0.1)$) .. controls +(0.1, 0.6) and +( 0.6, 0)
      .. ($(b)+ (0.1,0.1)$) -- (b) .. controls +(0.6,0) and +( 0.1,0.6)
      .. (a);
    \draw[red] (a).. controls +(0.1, 0.6) and +( 0.6, 0) .. (b);

    \draw[densely dotted] (1, -2) arc (0:180:1cm and 0.3cm);
    \draw (1, -2) arc (0:-180:1cm and 0.3cm);
    \draw (1, 2) arc (0:360:1cm and 0.3cm);
    \draw[very thin] (-1, -2) .. controls +(0, 1) and + (0, -1) ..   (-3,0) .. controls +(0, 1) and +(0, -1) ..  (-1, 2);
    \draw[very thin] ( 1, -2) .. controls +(0, 1) and + (0, -1) ..   ( 3,0) .. controls +(0, 1) and +(0, -1) ..  ( 1, 2);
  \end{scope}

%% file: fe_mdeltasigma4.tex
  \begin{scope}[xscale=-1]
    \draw[very thin] (0, -0.2) arc (270: 240: 3) coordinate[pos=0.8] (a);
    \draw[very thin] (0, -0.2) arc (270: 300: 3);
    \draw[very thin] (a) arc (114: 66: 3);
    \draw[very thin] (-1, -2) .. controls +(0, 1) and + (0, -1) ..
    (-3,0) .. controls +(0, 1) and +(0, -1) ..  (-1, 2)
    coordinate[pos=0.5] (b);
      \fill[pattern=north west lines, pattern color= red, draw = none]
      (a) -- ($(a)+ (0.1,0.1)$) .. controls +(0.1, 0.6) and +( 0.6, 0)
      .. ($(b)+ (0.1,0.1)$) -- (b) .. controls +(0.6,0) and +( 0.1,0.6)
      .. (a);
    \draw[red] (a).. controls +(0.1, 0.6) and +( 0.6, 0) .. (b);

    \draw[densely dotted] (1, -2) arc (0:180:1cm and 0.3cm);
    \draw (1, -2) arc (0:-180:1cm and 0.3cm);
    \draw (1, 2) arc (0:360:1cm and 0.3cm);
    \draw[very thin] (-1, -2) .. controls +(0, 1) and + (0, -1) ..   (-3,0) .. controls +(0, 1) and +(0, -1) ..  (-1, 2);
    \draw[very thin] ( 1, -2) .. controls +(0, 1) and + (0, -1) ..   ( 3,0) .. controls +(0, 1) and +(0, -1) ..  ( 1, 2);
  \end{scope}

%% file: fe_mdeltasigma3.tex
  \begin{scope}[xscale=-1]
    \draw[very thin] (0, -0.2) arc (270: 240: 3) coordinate[pos=0.8] (a);
    \draw[very thin] (0, -0.2) arc (270: 300: 3);
    \draw[very thin] (a) arc (114: 66: 3);
    \draw[very thin] (-1, -2) .. controls +(0, 1) and + (0, -1) ..
    (-3,0) .. controls +(0, 1) and +(0, -1) ..  (-1, 2)
    coordinate[pos=0.5] (b);
      \fill[pattern=north west lines, pattern color= red, draw = none]
      (a) -- ($(a)+ (-0.1,-0.1)$) .. controls +(0.1, 0.6) and +( 0.6, 0)
      .. ($(b)+ (-0.1,-0.1)$) -- (b) .. controls +(0.6,0) and +( 0.1,0.6)
      .. (a);
    \draw[red] (a).. controls +(0.1, 0.6) and +( 0.6, 0) .. (b);

    \draw[densely dotted] (1, -2) arc (0:180:1cm and 0.3cm);
    \draw (1, -2) arc (0:-180:1cm and 0.3cm);
    \draw (1, 2) arc (0:360:1cm and 0.3cm);
    \draw[very thin] (-1, -2) .. controls +(0, 1) and + (0, -1) ..   (-3,0) .. controls +(0, 1) and +(0, -1) ..  (-1, 2);
    \draw[very thin] ( 1, -2) .. controls +(0, 1) and + (0, -1) ..   ( 3,0) .. controls +(0, 1) and +(0, -1) ..  ( 1, 2);
  \end{scope}

%% file: fe_theta.tex
\begin{scope}
\begin{scope}[xshift=0cm]
\begin{scope}[yshift=0cm]
  \draw[very thin] (1,0) arc (0:360: 1cm);
  \draw[red] (1,0) arc  (0:180:1cm and 0.3cm);
  \fill[pattern=vertical lines, pattern color= red, draw = none] (1,0) arc  (0:180:1cm and 0.3cm) -- (-1, -0.2) arc (180:0: 1cm and 0.3cm) --cycle;
  \draw [line width =4.4mm, white, dashed ] (0.89,0.1) -- (-0.89, 0.10);  
  \fill[pattern=vertical lines, pattern color= red, draw = none] (1,0) arc  (0:-180:1cm and 0.3cm) -- (-1, -0.2) arc (-180:0: 1cm and 0.3cm) --cycle;
  \draw[red] (1,0) arc  (0:-180:1cm and 0.3cm); 
  \node at (0, 0.6) {$\bullet$};
\end{scope}
\begin{scope}[yshift=-2.5cm]
  \draw[very thin] (1,0) arc (0:360: 1cm);
  \draw[red, dashed] (1,0) arc  (0:180:1cm and 0.3cm);
  \draw[red] (1,0) arc  (0:-180:1cm and 0.3cm) coordinate[pos=0.5]
  (a);
  \draw[red] (a) -- +(0, -0.2);
  \node at (0, 0.6) {$\bullet$};
\end{scope}
\coordinate (E1) at (0, -4);
\end{scope}

\begin{scope}[xshift=4cm]
\begin{scope}[yshift=0cm]
  \draw[red] (1,0) arc  (0:180:1cm and 0.3cm);
  \fill[pattern=vertical lines, pattern color= red, draw = none] (1,0)
  arc  (0:180:1cm and 0.3cm) -- (-1,  0.2) arc (180:0: 1cm and 0.3cm) --cycle;
  \draw [line width =4.4mm, white, dashed ] (0.89,0.3) -- (-0.89, 0.30);  
  \fill[pattern=vertical lines, pattern color= red, draw = none] (1,0)
  arc  (0:-180:1cm and 0.3cm) -- (-1, 0.2) arc (-180:0: 1cm and 0.3cm) --cycle;
  \draw[red] (1,0) arc  (0:-180:1cm and 0.3cm); 
  \node at (0, 0.6) {$\bullet$};
    \draw[very thin] (1,0) arc (0:360: 1cm);
\end{scope}

\begin{scope}[yshift=-2.5cm]
  \draw[very thin] (1,0) arc (0:360: 1cm);
  \draw[red, dashed] (1,0) arc  (0:180:1cm and 0.3cm);
  \draw[red] (1,0) arc  (0:-180:1cm and 0.3cm) coordinate[pos=0.5]
  (a);
  \draw[red] (a) -- +(0, 0.2);
  \node at (0, 0.6) {$\bullet$};
\end{scope}
\coordinate (E2) at (0, -4);
\end{scope}
\begin{scope}[xshift=8.5cm]
\begin{scope}[yshift=0cm]
  \draw[white, very thin] (1.5,0) arc (0:360: 1.5cm and 1cm)
  coordinate[pos=0.15] (a) coordinate[pos=0.35] (b);
  \draw[red]
  let \p{a}=(a)
  in
  (a) arc (90: 270: {0.3*\y{a}} and \y{a})
  (b) arc (90: 270: {0.3*\y{a}} and \y{a});
  \fill[pattern=horizontal lines, pattern color= red, draw = none]
  let \p{a}=(a)
  in
  (a) arc  (90: 270: {0.3*\y{a}} and \y{a}) -- ++(0.2,  0) arc (270:90:{0.3*\y{a}} and \y{a}) --cycle
  (b) arc  (90: 270: {0.3*\y{a}} and \y{a}) -- ++(0.2,  0) arc
  (270:90:{0.3*\y{a}} and \y{a}) --cycle;
  \draw [line width =3cm, white, dashed ] (0,1) -- (0, -1);  
  \draw[red]
  let \p{a}=(a)
  in
  (a) arc (90: -90: {0.3*\y{a}} and \y{a})
  (b) arc (90: -90: {0.3*\y{a}} and \y{a});
  \fill[pattern=horizontal lines, pattern color= red, draw = none]
  let \p{a}=(a)
  in
  (a) arc  (90: -90: {0.3*\y{a}} and \y{a}) -- ++(0.2,  0) arc (-90:90:{0.3*\y{a}} and \y{a}) --cycle
  (b) arc  (90: -90: {0.3*\y{a}} and \y{a}) -- ++(0.2,  0) arc
  (-90:90:{0.3*\y{a}} and \y{a}) --cycle;
  \draw[black, very thin] (1.5,0) arc (0:360: 1.5cm and 1cm);
  \node at (0, 0.1) {$\bullet$};
\end{scope}

\begin{scope}[yshift=-2.5cm]
  \draw[white, very thin] (1.5,0) arc (0:360: 1.5cm and 1cm)
  coordinate[pos=0.15] (a) coordinate[pos=0.35] (b);
  \draw[dashed, red]
  let \p{a}=(a)
  in
  (a) arc (90: 270: {0.3*\y{a}} and \y{a})
  (b) arc (90: 270: {0.3*\y{a}} and \y{a});
  \draw[red]
  let \p{a}=(a)
  in
  (a) arc (90: -90: {0.3*\y{a}} and \y{a}) coordinate[pos = 0.5] (c) 
  (b) arc (90: -90: {0.3*\y{a}} and \y{a}) coordinate[pos = 0.5] (d);
  \draw[black, very thin] (1.5,0) arc (0:360: 1.5cm and 1cm);
  \draw[red] (c) -- +(0.2, 0);
  \draw[red] (d) -- +(0.2, 0);
  \node at (0, 0.1) {$\bullet$};
\end{scope}
\coordinate (E3) at (0, -4);
\end{scope}

\begin{scope}[xshift=13.5cm]
\begin{scope}[yshift=0cm]
  \draw[white, very thin] (1.5,0) arc (0:360: 1.5cm and 1cm)
  coordinate[pos=0.15] (a) coordinate[pos=0.35] (b);
  \draw[red]
  let \p{a}=(a)
  in
  (a) arc (90: 270: {0.3*\y{a}} and \y{a})
  (b) arc (90: 270: {0.3*\y{a}} and \y{a});
  \fill[pattern=horizontal lines, pattern color= red, draw = none]
  let \p{a}=(a)
  in
  (a) arc  (90: 270: {0.3*\y{a}} and \y{a}) -- ++(-0.2,  0) arc (270:90:{0.3*\y{a}} and \y{a}) --cycle
  (b) arc  (90: 270: {0.3*\y{a}} and \y{a}) -- ++(0.2,  0) arc
  (270:90:{0.3*\y{a}} and \y{a}) --cycle;
  \draw [line width =3cm, white, dashed ] (0,1) -- (0, -1);  
  \draw[red]
  let \p{a}=(a)
  in
  (a) arc (90: -90: {0.3*\y{a}} and \y{a})
  (b) arc (90: -90: {0.3*\y{a}} and \y{a});
  \fill[pattern=horizontal lines, pattern color= red, draw = none]
  let \p{a}=(a)
  in
  (a) arc  (90: -90: {0.3*\y{a}} and \y{a}) -- ++(-0.2,  0) arc (-90:90:{0.3*\y{a}} and \y{a}) --cycle
  (b) arc  (90: -90: {0.3*\y{a}} and \y{a}) -- ++(0.2,  0) arc
  (-90:90:{0.3*\y{a}} and \y{a}) --cycle;
  \draw[black, very thin] (1.5,0) arc (0:360: 1.5cm and 1cm);
  \node at (0, 0.1) {$\bullet$};
\end{scope}

\begin{scope}[yshift=-2.5cm]
  \draw[white, very thin] (1.5,0) arc (0:360: 1.5cm and 1cm)
  coordinate[pos=0.15] (a) coordinate[pos=0.35] (b);
  \draw[dashed, red]
  let \p{a}=(a)
  in
  (a) arc (90: 270: {0.3*\y{a}} and \y{a})
  (b) arc (90: 270: {0.3*\y{a}} and \y{a});
  \draw[red]
  let \p{a}=(a)
  in
  (a) arc (90: -90: {0.3*\y{a}} and \y{a}) coordinate[pos = 0.5] (c) 
  (b) arc (90: -90: {0.3*\y{a}} and \y{a}) coordinate[pos = 0.5] (d);
  \draw[black, very thin] (1.5,0) arc (0:360: 1.5cm and 1cm);
  \draw[red] (c) -- +(-0.2, 0);
  \draw[red] (d) -- +(0.2, 0);
  \node at (0, 0.1) {$\bullet$};
\end{scope}
\coordinate (E4) at (0, -4);
\end{scope}
\node at (E1) {$-1$};
\node at (E2) {$1$};
\node at (E3) {$-1$};
\node at (E4) {$1$};
\end{scope}

%% file: fe_eval-exa-1.tex
\begin{scope}
    \clip (-1.6, -1.1) rectangle (1.6, 1.1);
  \draw[white, very thin] (1.5,0) arc (0:360: 1.5cm and 1cm)
  coordinate[pos=0.15] (a) coordinate[pos=0.35] (b);
  \draw[red]
  let \p{a}=(a)
  in
  (a) arc (90: 270: {0.3*\y{a}} and \y{a})
  (b) arc (90: 270: {0.3*\y{a}} and \y{a});
  \fill[pattern=horizontal lines, pattern color= red, draw = none]
  let \p{a}=(a)
  in
  (a) arc  (90: 270: {0.3*\y{a}} and \y{a}) -- ++(-0.2,  0) arc (270:90:{0.3*\y{a}} and \y{a}) --cycle
  (b) arc  (90: 270: {0.3*\y{a}} and \y{a}) -- ++(0.2,  0) arc
  (270:90:{0.3*\y{a}} and \y{a}) --cycle;
  \draw [line width =3cm, white, dashed ] (0,1) -- (0, -1);  
  \draw[red]
  let \p{a}=(a)
  in
  (a) arc (90: -90: {0.3*\y{a}} and \y{a})
  (b) arc (90: -90: {0.3*\y{a}} and \y{a});
  \fill[pattern=horizontal lines, pattern color= red, draw = none]
  let \p{a}=(a)
  in
  (a) arc  (90: -90: {0.3*\y{a}} and \y{a}) -- ++(-0.2,  0) arc (-90:90:{0.3*\y{a}} and \y{a}) --cycle
  (b) arc  (90: -90: {0.3*\y{a}} and \y{a}) -- ++(0.2,  0) arc
  (-90:90:{0.3*\y{a}} and \y{a}) --cycle;
  \draw[black, very thin] (1.5,0) arc (0:360: 1.5cm and 1cm);
  \node at (0, 0.1) {$\bullet$};
\end{scope}

%% file: fe_eval-exa-2.tex
\begin{scope}
  \clip (-1.6, -1.1) rectangle (1.6, 1.1);
  \draw[white, very thin] (1.5,0) arc (0:360: 1.5cm and 1cm)
  coordinate[pos=0.15] (a) coordinate[pos=0.35] (b);
  \draw[red]
  let \p{a}=(a)
  in
  (a) arc (90: 270: {0.3*\y{a}} and \y{a})
  (b) arc (90: 270: {0.3*\y{a}} and \y{a});
  \fill[pattern=horizontal lines, pattern color= red, draw = none]
  let \p{a}=(a)
  in
  (a) arc  (90: 270: {0.3*\y{a}} and \y{a}) -- ++(-0.2,  0) arc (270:90:{0.3*\y{a}} and \y{a}) --cycle
  (b) arc  (90: 270: {0.3*\y{a}} and \y{a}) -- ++(0.2,  0) arc
  (270:90:{0.3*\y{a}} and \y{a}) --cycle;
  \draw [line width =3cm, white, dashed ] (0,1) -- (0, -1);  
  \draw[red]
  let \p{a}=(a)
  in
  (a) arc (90: -90: {0.3*\y{a}} and \y{a})
  (b) arc (90: -90: {0.3*\y{a}} and \y{a});
  \fill[pattern=horizontal lines, pattern color= red, draw = none]
  let \p{a}=(a)
  in
  (a) arc  (90: -90: {0.3*\y{a}} and \y{a}) -- ++(-0.2,  0) arc (-90:90:{0.3*\y{a}} and \y{a}) --cycle
  (b) arc  (90: -90: {0.3*\y{a}} and \y{a}) -- ++(0.2,  0) arc
  (-90:90:{0.3*\y{a}} and \y{a}) --cycle;
  \draw[black, very thin] (1.5,0) arc (0:360: 1.5cm and 1cm);
  \node at (0, 0.1) {$\bullet$};
  \fill[pattern = north east lines, opacity=0.3]
  let \p{a}=(a)
  in
  (\x{a}, \y{a}) arc (54: 126:1.5cm and 1cm)  arc (90:270:{0.3*\y{a}}
  and   \y{a}) arc (234: 306: 1.5cm and 1cm) arc (270: 90: {0.3*\y{a}} and
  \y{a});
  \fill[pattern = north west lines, opacity=0.3]
  let \p{a}=(a)
  in
  (a) arc (54: 126:1.5cm and 1cm)  arc (90:-90: {0.3*\y{a}} and
  \y{a}) arc (234: 306:1.5cm and 1cm) arc (-90: 90: {0.3*\y{a}} and
  \y{a});
\end{scope}

%% file: fe_eval-exa-3.tex
\begin{scope}
  \clip (-1.6, -1.1) rectangle (1.6, 1.1);
  \draw[white, very thin] (1.5,0) arc (0:360: 1.5cm and 1cm)
  coordinate[pos=0.15] (a) coordinate[pos=0.35] (b);
  \draw[red]
  let \p{a}=(a)
  in
  (a) arc (90: 270: {0.3*\y{a}} and \y{a})
  (b) arc (90: 270: {0.3*\y{a}} and \y{a});
  \fill[pattern=horizontal lines, pattern color= red, draw = none]
  let \p{a}=(a)
  in
  (a) arc  (90: 270: {0.3*\y{a}} and \y{a}) -- ++(-0.2,  0) arc (270:90:{0.3*\y{a}} and \y{a}) --cycle
  (b) arc  (90: 270: {0.3*\y{a}} and \y{a}) -- ++(0.2,  0) arc
  (270:90:{0.3*\y{a}} and \y{a}) --cycle;
  \draw [line width =3cm, white, dashed ] (0,1) -- (0, -1);  
  \draw[red]
  let \p{a}=(a)
  in
  (a) arc (90: -90: {0.3*\y{a}} and \y{a})
  (b) arc (90: -90: {0.3*\y{a}} and \y{a});
  \fill[pattern=horizontal lines, pattern color= red, draw = none]
  let \p{a}=(a)
  in
  (a) arc  (90: -90: {0.3*\y{a}} and \y{a}) -- ++(-0.2,  0) arc (-90:90:{0.3*\y{a}} and \y{a}) --cycle
  (b) arc  (90: -90: {0.3*\y{a}} and \y{a}) -- ++(0.2,  0) arc
  (-90:90:{0.3*\y{a}} and \y{a}) --cycle;
  \draw[black, very thin] (1.5,0) arc (0:360: 1.5cm and 1cm);
  \node at (0, 0.1) {$\bullet$};
  \fill[pattern = north east lines, opacity=0.3]
  let \p{a}=(a)
  in
  (\x{a}, \y{a}) arc (54: -54:1.5cm and 1cm)  arc (270:90:{0.3*\y{a}}
  and \y{a});
  \fill[pattern = north west lines, opacity=0.3]
  let \p{a}=(a)
  in
  (\x{a}, \y{a}) arc (54: -54:1.5cm and 1cm)  arc (-90:90:{0.3*\y{a}}   and   \y{a});

  \draw[black, very thin] (1.5,0) arc (0:360: 1.5cm and 1cm);
  \node at (0, 0.1) {$\bullet$};
  \fill[pattern = north east lines, opacity=0.3]
  let \p{a}=(a)
  in
  (b) arc (126: 234:1.5cm and 1cm)  arc (270:90:{0.3*\y{a}}
  and \y{a});
  \fill[pattern = north west lines, opacity=0.3]
  let \p{a}=(a)
  in
  (b) arc (126: 234:1.5cm and 1cm)  arc (-90:90:{0.3*\y{a}}   and   \y{a});
\end{scope}

%% file: fe_split-circle-col.tex
\begin{scope}[scale =1]
\begin{scope}[xshift=0cm]
\begin{scope}  [decoration={border,segment length=1mm,amplitude=1mm,angle=-90}]
  \begin{scope}[yshift= 1.5cm]
    \draw (1,0) arc (0:360:1cm and 0.3cm)coordinate [pos=0](e)
    coordinate [pos=0.5](f) coordinate[pos=0.65] (bt)
    coordinate[pos=0.85] (at);
    \draw[thin] (1,0) arc (0:-180:1cm and 1.2cm);
    \fill[pattern = north east lines, opacity=0.3] (1,0) arc (0:180:1cm and 0.3cm) arc(180:360: 1cm and 1.2cm);
    \fill[pattern = north west lines, opacity=0.3] (1,0) arc  (0:-180:1cm and 0.3cm)  arc(180:360: 1cm and 1.2cm);
  \end{scope}
  \begin{scope}[yshift =-1.5cm]
    \draw (1,0) arc (0:-180:1cm and 0.3cm)coordinate [pos=0](g)
    coordinate [pos=1](h) coordinate[pos=0.3] (ab)
    coordinate[pos=0.7] (bb);
    \draw (1,0) arc (0:180:1cm and 1.2cm);
    \draw[densely dotted] (1,0) arc (0:180:1cm and 0.3cm);
    \fill[pattern = north east lines, opacity=0.3] (1,0) arc (0: 180:1cm and 1.2cm) arc( 180:0: 1cm and 0.3cm);
    \fill[pattern = north west lines, opacity=0.3] (1,0) arc (0: 180:1cm and 1.2cm) arc( 180: 360: 1cm and 0.3cm);
  \end{scope}
\coordinate (A) at (0, -2.5);
\end{scope}
\end{scope}

\begin{scope}[xshift=4cm]
\begin{scope}  [decoration={border,segment length=1mm,amplitude=1mm,angle=-90}]
  \begin{scope}[yshift= 1.5cm]
    \draw (1,0) arc (0:360:1cm and 0.3cm)coordinate [pos=0](e)
    coordinate [pos=0.5](f) coordinate[pos=0.65] (bt)
    coordinate[pos=0.85] (at);
    \draw[thin] (1,0) arc (0:-180:1cm and 1.2cm);
  \end{scope}
  \begin{scope}[yshift =-1.5cm]
    \draw (1,0) arc (0:-180:1cm and 0.3cm)coordinate [pos=0](g)
    coordinate [pos=1](h) coordinate[pos=0.3] (ab)
    coordinate[pos=0.7] (bb);
    \draw (1,0) arc (0:180:1cm and 1.2cm);
    \draw[densely dotted] (1,0) arc (0:180:1cm and 0.3cm);
  \end{scope}
\coordinate (B) at (0, -2.5);
\end{scope}

\end{scope}

\begin{scope}[xshift=8cm]
\begin{scope}  [decoration={border,segment length=1mm,amplitude=1mm,angle=-90}]
  \begin{scope}[yshift= 1.5cm]
    \draw (1,0) arc (0:360:1cm and 0.3cm)coordinate [pos=0](e)
    coordinate [pos=0.5](f) coordinate[pos=0.65] (bt)
    coordinate[pos=0.85] (at);
    \draw[thin] (1,0) arc (0:-180:1cm and 1.2cm);
    \fill[pattern = north east lines, opacity=0.3] (1,0) arc (0:180:1cm and 0.3cm) arc(180:360: 1cm and 1.2cm);
    \fill[pattern = north west lines, opacity=0.3] (1,0) arc
    (0:-180:1cm and 0.3cm) arc(180:360: 1cm and 1.2cm);
    %
  \end{scope}
  \begin{scope}[yshift =-1.5cm]
    \draw (1,0) arc (0:-180:1cm and 0.3cm)coordinate [pos=0](g)
    coordinate [pos=1](h) coordinate[pos=0.3] (ab)
    coordinate[pos=0.7] (bb);
    \draw (1,0) arc (0:180:1cm and 1.2cm);
    \draw[densely dotted] (1,0) arc (0:180:1cm and 0.3cm);
  \end{scope}
\coordinate (C) at (0, -2.5);
\end{scope}

\end{scope}

\begin{scope}[xshift=12cm]
\begin{scope}  [decoration={border,segment length=1mm,amplitude=1mm,angle=-90}]
  \begin{scope}[yshift= 1.5cm]
    \draw (1,0) arc (0:360:1cm and 0.3cm)coordinate [pos=0](e)
    coordinate [pos=0.5](f) coordinate[pos=0.65] (bt)
    coordinate[pos=0.85] (at);
    \draw[thin] (1,0) arc (0:-180:1cm and 1.2cm);
  \end{scope}
  \begin{scope}[yshift =-1.5cm]
    \draw (1,0) arc (0:-180:1cm and 0.3cm)coordinate [pos=0](g)
    coordinate [pos=1](h) coordinate[pos=0.3] (ab)
    coordinate[pos=0.7] (bb);
    \draw (1,0) arc (0:180:1cm and 1.2cm);
    \draw[densely dotted] (1,0) arc (0:180:1cm and 0.3cm);
    \fill[pattern = north east lines, opacity=0.3] (1,0) arc (0: 180:1cm and 1.2cm) arc( 180:0: 1cm and 0.3cm);
    \fill[pattern = north west lines, opacity=0.3] (1,0) arc (0: 180:1cm and 1.2cm) arc( 180: 360: 1cm and 0.3cm);
  \end{scope}
  \coordinate (D) at (0, -2.5);
\end{scope}
\end{scope}
\node at (A) {(a)};
\node at (B) {(b)};
\node at (C) {(c)};
\node at (D) {(d)};
\end{scope}

%% file: fe_marked-circles-exa.tex
\begin{scope}
    \draw[->-, ->] (0:1.5) arc (0:-360:1.5cm);
    \foreach \angle in { 213, 307, 34, 146}{
       \draw[red, -latex, thick] (\angle:1.5cm) arc (\angle:(2+\angle):1.5cm);
       \node[font=\footnotesize] at (\angle:1.7) {$-$};
    }
    \foreach \angle in { -20, 100, 260}{
       \draw[red, -latex, thick] (\angle:1.5cm) arc (\angle:(-2+\angle):1.5cm);
    \node[font=\footnotesize] at (\angle:1.7) {$+$};
}

    \draw[->-, >-] (0:0.9) arc (0:360:0.9cm);
    \foreach \angle in { 120}{
       \draw[red, -latex, thick] (\angle:0.9cm) arc (\angle:(2+\angle):1.5cm);
       \node[font=\footnotesize] at (\angle:1.1) {$+$};
    }
    \foreach \angle in { 300}{
       \draw[red, -latex, thick] (\angle:0.9cm) arc (\angle:(-2+\angle):1.5cm);
    \node[font=\footnotesize] at (\angle:1.1) {$-$};
}

    \draw[->-, ->] (0:0.3) arc (0:-360:0.3cm);
    
  \end{scope}
\begin{scope}[xshift =4cm]
      \draw[->-, ->] (0:1.2) arc (0:-360:1.2cm);
    \foreach \angle in { 120, 30}{
       \draw[red, -latex, thick] (\angle:1.2cm) arc (\angle:(2+\angle):1.5cm);
       \node[font=\footnotesize] at (\angle:1.4) {$+$};
    }
    \foreach \angle in { 300, 210}{
       \draw[red, -latex, thick] (\angle:1.2cm) arc (\angle:(-2+\angle):1.5cm);
    \node[font=\footnotesize] at (\angle:1.4) {$-$};
}
  \end{scope}

%% file: fe_conv_orientations.tex
\begin{scope}[decoration={border,segment
  length=1mm,amplitude=1mm,angle=90}]
  \draw (-1, 1) -- (-2, 1) -- (-2, -1) -- (-1,-1);
  \draw[-to] (1,1) -- (-1,1);
  \draw[to-to] (1,1) -- (2, 1) -- (2, -1) -- (1, -1);
  \draw[-to] (1,-1) -- (-1,-1);
  \draw[red, -latex, thick] (0,1) -- +(0.1, 0);
  \draw[red, -latex, thick] (-1.5,-1) -- +(0.1, 0);
    \draw[red, -latex, thick] (0,-1) -- +(-0.1, 0);
  \draw[red, -latex, thick] ( 1.5,-1) -- +(0.1, 0);
  \draw[red, postaction={draw, decorate}] (-0.07, 1) .. controls +(0,
  -0.6) and + (0, 0.6) .. (-1.57, -1);
    \draw[red, postaction={draw, decorate}] (0.07, -1) .. controls +(0,
    1) and + (0, 1) .. (1.43, -1);
    \node at (0.75, -0.7) {$\circlearrowleft$};
    \node at (0.75, 0.5) {$\circlearrowleft$};
    \node at (-1, 0.5) {$\circlearrowleft$};
\draw[gray, very thin,<-, font= \footnotesize] (-2.1 ,-1) -- +(-1, 0) node [left, gray]
{$\mathbb{R}^2 \times \{0\}$};  
\draw[gray, very thin,<-, font= \footnotesize] (-2.1 ,1) -- +(-1, 0) node [left, gray]
{$\mathbb{R}^2 \times \{1\}$};  
\end{scope}

%% file: fe_id-circle-2pt.tex
\begin{scope}[scale=0.8, decoration={border,segment length=1mm,amplitude=1mm,angle=-90}]
  \begin{scope}[yshift= 1.5cm]
    \draw (1,0) arc (0:360:1cm and 0.3cm)coordinate [pos=0](e)
    coordinate [pos=0.5](f) coordinate[pos=0.65] (bt)
    coordinate[pos=0.85] (at);
  \end{scope}
    \draw[densely dashed, green!50!black] (1,0) arc (0:-180:1cm and
    0.3cm) coordinate [pos=1] (gamma);
    \draw[<-, very thin, green!50!black] (gamma) -- +(-0.5,0) node[left, green!50!black] {$\gamma$};
    \draw[dotted, green!50!black] (1,0) arc (0:180:1cm and 0.3cm);   
  \begin{scope}[yshift =-1.5cm]
    \draw (1,0) arc (0:-180:1cm and 0.3cm)coordinate [pos=0](g)
    coordinate [pos=1](h) coordinate[pos=0.3] (ab)
    coordinate[pos=0.7] (bb);
    \draw[densely dotted] (1,0) arc (0:180:1cm and 0.3cm); 
  \end{scope}
  \draw (e) -- (g);
  \draw (f) -- (h);
  \draw[red, postaction={draw, decorate}] (at) -- (ab);
  \draw[red, postaction={draw, decorate}] (bb) -- (bt);
\end{scope}

%% file: fe_split-circle-2pt-2.tex
\begin{scope}[scale =0.8]
\begin{scope}  [decoration={border,segment length=1mm,amplitude=1mm,angle=-90}]
  \begin{scope}[yshift= 1.5cm]
    \draw (1,0) arc (0:360:1cm and 0.3cm)coordinate [pos=0](e)
    coordinate [pos=0.5](f) coordinate[pos=0.65] (bt)
    coordinate[pos=0.85] (at);
    \draw[very thin] (1,0) arc (0:-180:1cm and 1.2cm);
  \end{scope}
  \begin{scope}[yshift =-1.5cm]
    \draw (1,0) arc (0:-180:1cm and 0.3cm)coordinate [pos=0](g)
    coordinate [pos=1](h) coordinate[pos=0.3] (ab)
    coordinate[pos=0.7] (bb);
    \draw[very thin] (1,0) arc (0:180:1cm and 1.2cm);
    \draw[densely dotted] (1,0) arc (0:180:1cm and 0.3cm); 
\node at (0,0.9) {$\bullet$};
  \end{scope}
  \draw[red, postaction={draw, decorate}] (at) .. controls +(0, -1)  and +(0, -1) .. (bt);
  \draw[red, postaction={draw, decorate}] (bb) .. controls +(0, 1)  and +(0, 1) .. (ab);
\end{scope}

\end{scope}

%% file: fe_split-circle-2pt-1.tex
\begin{scope}[scale =0.8]
\begin{scope}  [decoration={border,segment length=1mm,amplitude=1mm,angle=-90}]
  \begin{scope}[yshift= 1.5cm]
    \draw (1,0) arc (0:360:1cm and 0.3cm)coordinate [pos=0](e)
    coordinate [pos=0.5](f) coordinate[pos=0.65] (bt)
    coordinate[pos=0.85] (at);
    \draw[very thin] (1,0) arc (0:-180:1cm and 1.2cm);
    \node at (0,-0.6) {$\bullet$};
  \end{scope}
  \begin{scope}[yshift =-1.5cm]
    \draw (1,0) arc (0:-180:1cm and 0.3cm)coordinate [pos=0](g)
    coordinate [pos=1](h) coordinate[pos=0.3] (ab)
    coordinate[pos=0.7] (bb);
    \draw[very thin] (1,0) arc (0:180:1cm and 1.2cm);
    \draw[densely dotted] (1,0) arc (0:180:1cm and 0.3cm); 
  \end{scope}
  \draw[red, postaction={draw, decorate}] (at) .. controls +(0, -1)  and +(0, -1) .. (bt);
  \draw[red, postaction={draw, decorate}] (bb) .. controls +(0, 1)  and +(0, 1) .. (ab);
\end{scope}

\end{scope}

%% file: fe_split-circle-2pt-col.tex
\begin{scope}[scale =1]
\begin{scope}[xshift=0cm]
\begin{scope}  [decoration={border,segment length=1mm,amplitude=1mm,angle=-90}]
  \begin{scope}[yshift= 1.5cm]
    \draw (1,0) arc (0:360:1cm and 0.3cm)coordinate [pos=0](e)
    coordinate [pos=0.5](f) coordinate[pos=0.65] (bt)
    coordinate[pos=0.85] (at);
    \draw[thin] (1,0) arc (0:-180:1cm and 1.2cm);
    \fill[pattern = north east lines, opacity=0.3] (1,0) arc (0:180:1cm and 0.3cm) arc(180:360: 1cm and 1.2cm);
    \fill[pattern = north west lines, opacity=0.3] (1,0) arc  (0:-54:1cm and 0.3cm)  .. controls  +(0, -1)  and +(0, -1) .. (bt) arc (-126:-180:1cm and 0.3cm) arc(180:360: 1cm and 1.2cm);
    \filldraw[fill=white] (at) .. controls +(0, -1)  and +(0, -1) .. (bt) arc (-126:-54:1cm and 0.3cm);
  \end{scope}
  \begin{scope}[yshift =-1.5cm]
    \draw (1,0) arc (0:-180:1cm and 0.3cm)coordinate [pos=0](g)
    coordinate [pos=1](h) coordinate[pos=0.3] (ab)
    coordinate[pos=0.7] (bb);
    \draw (1,0) arc (0:180:1cm and 1.2cm);
    \draw[densely dotted] (1,0) arc (0:180:1cm and 0.3cm);
    \fill[pattern = north east lines, opacity=0.3] (1,0) arc (0: 180:1cm and 1.2cm) arc( 180:0: 1cm and 0.3cm);
    \fill[pattern = north west lines, opacity=0.3] (1,0) arc (0: 180:1cm and 1.2cm) arc( 180: 360: 1cm and 0.3cm);
    \filldraw[fill=white] (ab) .. controls +(0, 1)  and +(0, 1) .. (bb) arc (-126:-54:1cm and 0.3cm);
  \end{scope}
  \draw[red, postaction={draw, decorate}] (at) .. controls +(0, -1)  and +(0, -1) .. (bt);
  \draw[red, postaction={draw, decorate}] (bb) .. controls +(0, 1)  and +(0, 1) .. (ab);
\coordinate (A) at (0, -2.5);
\end{scope}
\end{scope}

\begin{scope}[xshift=4cm]
\begin{scope}  [decoration={border,segment length=1mm,amplitude=1mm,angle=-90}]
  \begin{scope}[yshift= 1.5cm]
    \draw (1,0) arc (0:360:1cm and 0.3cm)coordinate [pos=0](e)
    coordinate [pos=0.5](f) coordinate[pos=0.65] (bt)
    coordinate[pos=0.85] (at);
    \draw[thin] (1,0) arc (0:-180:1cm and 1.2cm);
    \filldraw[pattern = north east lines, opacity=0.3] (at) .. controls +(0, -1)  and +(0, -1) .. (bt) arc (-126:-54:1cm and 0.3cm);
  \end{scope}
  \begin{scope}[yshift =-1.5cm]
    \draw (1,0) arc (0:-180:1cm and 0.3cm)coordinate [pos=0](g)
    coordinate [pos=1](h) coordinate[pos=0.3] (ab)
    coordinate[pos=0.7] (bb);
    \draw (1,0) arc (0:180:1cm and 1.2cm);
    \draw[densely dotted] (1,0) arc (0:180:1cm and 0.3cm);
    \filldraw[pattern = north east lines, opacity=0.3] (ab) .. controls +(0, 1)  and +(0, 1) .. (bb) arc (-126:-54:1cm and 0.3cm);
  \end{scope}
  \draw[red, postaction={draw, decorate}] (at) .. controls +(0, -1)  and +(0, -1) .. (bt);
  \draw[red, postaction={draw, decorate}] (bb) .. controls +(0, 1)  and +(0, 1) .. (ab);
\coordinate (B) at (0, -2.5);
\end{scope}

\end{scope}

\begin{scope}[xshift=8cm]
\begin{scope}  [decoration={border,segment length=1mm,amplitude=1mm,angle=-90}]
  \begin{scope}[yshift= 1.5cm]
    \draw (1,0) arc (0:360:1cm and 0.3cm)coordinate [pos=0](e)
    coordinate [pos=0.5](f) coordinate[pos=0.65] (bt)
    coordinate[pos=0.85] (at);
    \draw[thin] (1,0) arc (0:-180:1cm and 1.2cm);
    \fill[pattern = north east lines, opacity=0.3] (1,0) arc (0:180:1cm and 0.3cm) arc(180:360: 1cm and 1.2cm);
    \fill[pattern = north west lines, opacity=0.3] (1,0) arc  (0:-54:1cm and 0.3cm)  .. controls  +(0, -1)  and +(0, -1) .. (bt) arc (-126:-180:1cm and 0.3cm) arc(180:360: 1cm and 1.2cm);
    \filldraw[fill=white] (at) .. controls +(0, -1)  and +(0, -1) .. (bt) arc (-126:-54:1cm and 0.3cm);
  \end{scope}
  \begin{scope}[yshift =-1.5cm]
    \draw (1,0) arc (0:-180:1cm and 0.3cm)coordinate [pos=0](g)
    coordinate [pos=1](h) coordinate[pos=0.3] (ab)
    coordinate[pos=0.7] (bb);
    \draw (1,0) arc (0:180:1cm and 1.2cm);
    \draw[densely dotted] (1,0) arc (0:180:1cm and 0.3cm);
    \filldraw[pattern = north east lines, opacity=0.3] (ab) .. controls +(0, 1)  and +(0, 1) .. (bb) arc (-126:-54:1cm and 0.3cm);
  \end{scope}
  \draw[red, postaction={draw, decorate}] (at) .. controls +(0, -1)  and +(0, -1) .. (bt);
  \draw[red, postaction={draw, decorate}] (bb) .. controls +(0, 1)  and +(0, 1) .. (ab);
\coordinate (C) at (0, -2.5);
\end{scope}

\end{scope}

\begin{scope}[xshift=12cm]
\begin{scope}  [decoration={border,segment length=1mm,amplitude=1mm,angle=-90}]
  \begin{scope}[yshift= 1.5cm]
    \draw (1,0) arc (0:360:1cm and 0.3cm)coordinate [pos=0](e)
    coordinate [pos=0.5](f) coordinate[pos=0.65] (bt)
    coordinate[pos=0.85] (at);
    \draw[thin] (1,0) arc (0:-180:1cm and 1.2cm);
    \filldraw[pattern = north east lines, opacity=0.3] (at) .. controls +(0, -1)  and +(0, -1) .. (bt) arc (-126:-54:1cm and 0.3cm);
  \end{scope}
  \begin{scope}[yshift =-1.5cm]
    \draw (1,0) arc (0:-180:1cm and 0.3cm)coordinate [pos=0](g)
    coordinate [pos=1](h) coordinate[pos=0.3] (ab)
    coordinate[pos=0.7] (bb);
    \draw (1,0) arc (0:180:1cm and 1.2cm);
    \draw[densely dotted] (1,0) arc (0:180:1cm and 0.3cm);
    \fill[pattern = north east lines, opacity=0.3] (1,0) arc (0: 180:1cm and 1.2cm) arc( 180:0: 1cm and 0.3cm);
    \fill[pattern = north west lines, opacity=0.3] (1,0) arc (0: 180:1cm and 1.2cm) arc( 180: 360: 1cm and 0.3cm);
    \filldraw[fill=white] (ab) .. controls +(0, 1)  and +(0, 1) .. (bb) arc (-126:-54:1cm and 0.3cm);
  \end{scope}
  \draw[red, postaction={draw, decorate}] (at) .. controls +(0, -1)  and +(0, -1) .. (bt);
  \draw[red, postaction={draw, decorate}] (bb) .. controls +(0, 1)  and +(0, 1) .. (ab);
  \coordinate (D) at (0, -2.5);
\end{scope}
\end{scope}
\node at (A) {(a)};
\node at (B) {(b)};
\node at (C) {(c)};
\node at (D) {(d)};
\end{scope}

%% file: fe_S4.tex
\begin{scope}
  \draw (1,0) arc (0:360:1cm);
  \draw[red, -latex, thick] (43:1cm) arc (43:45:1cm);
  \draw[red, -latex, thick] (137:1cm) arc (137:135:1cm);
  \draw[red, -latex, thick] (223:1cm) arc (223:225:1cm);
  \draw[red, -latex, thick] (317:1cm) arc (317:315:1cm);
\end{scope}

\begin{scope}[decoration={border,segment
  length=1mm,amplitude=1mm,angle=-90}, xshift= 4cm, yshift =0.6cm]
  \draw (1,0) arc (0:360:1cm and 0.3cm)coordinate [pos=0.17](al)
    coordinate[pos=0.33] (bl)
    coordinate [pos=0.5](f) coordinate[pos=0.65] (bt)
    coordinate[pos=0.85] (at);
    \draw[very thin] (1,0) arc (0:-180:1cm and 1.2cm);
    \node at (0.6, 0) {$\bullet$};
    \draw[red, postaction={draw, decorate}] (at) .. controls +(0, -1)  and +(0, -1) .. (bt);
    \draw[red, postaction={draw, decorate}] (al) .. controls +(0, -0.5)  and +(0, -0.5) .. (bl);
  \end{scope}

 \begin{scope}[xshift =8cm]
\begin{scope}[decoration={border,segment
  length=1mm,amplitude=1mm,angle=-90}]
  \draw (1,0) arc (0:360:1cm);
  \draw[red, postaction={draw, decorate}] (45:1cm) .. controls (45:0.5cm)  and (135:0.5cm) .. (135:1cm);
  \draw[red, postaction={draw, decorate}] (225:1cm) .. controls (225: 0.5cm)  and (315:0.5cm) .. (315:1cm);
  \node at (0,0) {$\bullet$};
\end{scope}    
 \end{scope}

%% file: fe_S4-elements.tex
\begin{scope}[decoration={border,segment
  length=1mm,amplitude=1mm,angle=-90}]
 \begin{scope}[xshift =0cm]
  \draw (1,0) arc (0:360:1cm);
  \draw[red, postaction={draw, decorate}] (45:1cm) .. controls (45:0.5cm)  and (315:0.5cm) .. (315:1cm);
  \draw[red, postaction={draw, decorate}] (225:1cm) .. controls (225: 0.5cm)  and (135:0.5cm) .. (135:1cm);
\end{scope}    
 \begin{scope}[xshift =4cm]
  \draw (1,0) arc (0:360:1cm);
  \draw[red, postaction={draw, decorate}] (45:1cm) .. controls (45:0.5cm)  and (135:0.5cm) .. (135:1cm);
  \draw[red, postaction={draw, decorate}] (225:1cm) .. controls (225: 0.5cm)  and (315:0.5cm) .. (315:1cm);
\end{scope}    
 \begin{scope}[xshift =8cm]
  \draw (1,0) arc (0:360:1cm);
  \draw[red, postaction={draw, decorate}] (45:1cm) .. controls (45:0.5cm)  and (315:0.5cm) .. (315:1cm);
  \draw[red, postaction={draw, decorate}] (225:1cm) .. controls (225: 0.5cm)  and (135:0.5cm) .. (135:1cm);
  \node at (0,0) {$\bullet$};
\end{scope}    
\begin{scope}[xshift =12cm]
  \draw (1,0) arc (0:360:1cm);
  \draw[red, postaction={draw, decorate}] (45:1cm) .. controls (45:0.5cm)  and (135:0.5cm) .. (135:1cm);
  \draw[red, postaction={draw, decorate}] (225:1cm) .. controls (225: 0.5cm)  and (315:0.5cm) .. (315:1cm);
  \node at (0,0) {$\bullet$};
\end{scope}    
 \end{scope}

%% file: fe_S4p.tex

\begin{scope}[decoration={border,segment
  length=1mm,amplitude=1mm,angle=-90}, xshift= 0cm, yshift =0.6cm]
  \draw (1,0) arc (0:360:1cm and 0.3cm)coordinate [pos=0.17](al)
    coordinate[pos=0.33] (bl)
    coordinate [pos=0.5](f) coordinate[pos=0.65] (bt)
    coordinate[pos=0.85] (at);
    \draw[very thin] (1,0) arc (0:-180:1cm and 1.2cm);
    \draw[red, postaction={draw, decorate}] (bt) .. controls +(0, -1)  and +(0, -1) .. (at);
    \draw[red, postaction={draw, decorate}] (al) .. controls +(0, -0.5)  and +(0, -0.5) .. (bl);
  \end{scope}

\node at (3,0)  {and};
\begin{scope}[decoration={border,segment
  length=1mm,amplitude=1mm,angle=-90}, xshift= 6cm, yshift =0.6cm]
  \draw (1,0) arc (0:360:1cm and 0.3cm)coordinate [pos=0.17](al)
    coordinate[pos=0.33] (bl)
    coordinate [pos=0.5](f) coordinate[pos=0.65] (bt)
    coordinate[pos=0.85] (at);
    \draw[very thin] (1,0) arc (0:-180:1cm and 1.2cm);
    \node at (0.6, 0) {$\bullet$};
    \draw[red, postaction={draw, decorate}] (bt) .. controls +(0, -1)  and +(0, -1) .. (at);
    \draw[red, postaction={draw, decorate}] (al) .. controls +(0, -0.5)  and +(0, -0.5) .. (bl);
  \end{scope}


%% file: fe_balanced-S.tex
\begin{scope}
  \draw (-30:1) arc (-30:-150:1cm);
  \draw [dotted] (-30:1) arc (-30: 210:1cm);
  \draw[red, -latex, thick] (233:1cm) arc (233:235:1cm);
  \draw[red, -latex, thick] (307:1cm) arc (307:305:1cm);
\end{scope}

%% file: fe_balanced-Sp.tex
\begin{scope}
  \draw (-30:1) arc (-30:-150:1cm);
  \draw [dotted] (-30:1) arc (-30: 210:1cm);
\end{scope}

%% file: fe_balanced-iso-1.tex
\begin{scope}[decoration={border,segment
  length=1mm,amplitude=1mm,angle=-90}, xshift= 4cm, yshift =0.6cm]
  \draw[dotted] (1,0) arc (0:360:1cm and 0.3cm)coordinate
  [pos=0.17](al) coordinate[pos=0.33] (bl)
  coordinate [pos=0.5](f) coordinate[pos= 0.6] (aa) coordinate[pos=0.65] (bt)
  coordinate[pos=0.85] (at);
  \draw (aa) arc (216:324:1cm and 0.3cm);    
  \draw[red, postaction={draw, decorate}] (at) .. controls +(0, -1)  and +(0, -1) .. (bt);
  \draw[dotted] (1,-2) arc (0:360:1cm and 0.3cm)
  coordinate[pos=0.9] (bb);
  \draw (bb) arc (324:216:1cm and 0.3cm);    
  \draw [dotted] (f) -- +(0, -2);
  \draw [dotted] (1,0) -- +(0, -2);
\end{scope}


%% file: fe_balanced-iso-2.tex
\begin{scope}[decoration={border,segment
  length=1mm,amplitude=1mm,angle=-90}, xshift= 4cm, yshift =0.6cm]
  \draw[dotted] (1,0) arc (0:360:1cm and 0.3cm)coordinate
  [pos=0.17](al) coordinate[pos=0.33] (bl)
  coordinate [pos=0.5](f) coordinate[pos= 0.6] (aa) coordinate[pos=0.65] (bt)
  coordinate[pos=0.85] (at);
  \draw (aa) arc (216:324:1cm and 0.3cm);    
  \draw[red, postaction={draw, decorate}] (bt) .. controls +(0, 1)  and +(0, 1) .. (at);
  \draw[dotted] (1,2) arc (0:360:1cm and 0.3cm)
  coordinate[pos=0.9] (bb);
  \draw (bb) arc (324:216:1cm and 0.3cm);    
  \draw [dotted] (f) -- +(0, 2);
  \draw [dotted] (1,0) -- +(0, 2);
\end{scope}
